\documentclass[12pt]{article}

\usepackage[ansinew]{inputenc}
\usepackage[UKenglish]{babel}
\usepackage{textcomp}
\usepackage{mathrsfs}
\usepackage{amsfonts}
\usepackage{mathtools}
\usepackage{amssymb}
\usepackage{amsmath}
\usepackage{amsthm}
\usepackage{amscd}
\usepackage{tensor}
\usepackage{stmaryrd}
\usepackage{rotating}
\usepackage[text={6.8in,8.5in},centering]{geometry}
\usepackage{tikz}
\usetikzlibrary{arrows, calc, decorations.markings, positioning}
\usepackage{tikz-cd}
\usepackage{relsize} 
\usepackage[none]{hyphenat}
\usepackage{hyperref}

\usepackage{graphicx}
\graphicspath{ {./Images/} }

\usepackage{hyperref,xcolor}
\usepackage{makeidx}
\definecolor{deepred}{rgb}{0.5,0,0}
\definecolor{deepblue}{rgb}{0,0,0.5}
\definecolor{deepgreen}{rgb}{0,0.5,0}

\hypersetup{
    colorlinks=true,
    linkcolor=deepred,
    filecolor=deepred,      
    urlcolor=deepred,
    citecolor=deepblue,
}

\newcommand{\Real}{\mathbb{R}}
\newcommand{\dimm}{\text{dim}}
\newcommand{\Cin}[1]{\text{C}^\infty(#1)}
\newcommand{\Tan}{\text{T}}
\newcommand{\Cot }{\text{T}^*}
\newcommand{\Sec}[1]{\Gamma(#1)}
\newcommand{\Der}{\text{D}}
\newcommand{\Jet}{\text{J}}
\newcommand{\Proj}{\text{pr}}
\newcommand{\Id}{\text{id}}
\newcommand{\rkk}[1]{\text{rk}(#1)}
\newcommand{\Hom}[1]{\text{Hom}(#1)}

\newcommand{\Dr}[1]{\text{Der}(#1)}

\newcommand{\Ker}[1]{\text{ker}(#1)}

\newcommand{\LGrph}[1]{\text{Lgrph}(#1)}

\newcommand{\Diff}{\text{Diff}}
\newcommand{\Obs}[1]{\text{Obs}(#1)}
\newcommand{\Dyn}[1]{\text{Dyn}(#1)}
\newcommand{\Lie}{\textsf{Lie}}

\newcommand{\LCrnt}{\textsf{LCrnt}}

\newcommand{\Man}{\textsf{Man}}

\newcommand{\Ring}{\textsf{Ring}}
\newcommand{\Vect}{\textsf{Vect}}
\newcommand{\LVect}{\textsf{LVect}}
\newcommand{\Line}{\textsf{Line}}
\newcommand{\Symp}{\textsf{Symp}}

\newcommand{\Cont}{\textsf{Cont}}

\newcommand{\Acts}{\mathbin{\rotatebox[origin=c]{-90}{$\circlearrowright$}}}
\newcommand{\dtimes}{\mathbin{\rotatebox[origin=c]{90}{$\ltimes$}}}
\newcommand{\utimes}{\mathbin{\rotatebox[origin=c]{-90}{$\ltimes$}}}

\newcommand{\LDer}{\mathcal{L}}

\newtheorem{prop}{Proposition}
\numberwithin{prop}{subsection}

\usepackage[
backend=biber,
style=alphabetic,
sorting=ynt
]{biblatex}
\begin{document}

\title{Jacobi Geometry and Hamiltonian Mechanics:\\ the Unit-Free Approach}

\author{Carlos Zapata-Carratal\'a \thanks{Email: \texttt{c.zapata.carratala@gmail.com} Address: Room 2254, School of Mathematics, The University of Edinburgh, James Clerk Maxwell Building, Peter Guthrie Tait Road, Edinburgh EH9 3FD}}
\date{}

\maketitle

\sloppy

\begin{abstract}
    We present a systematic treatment of line bundle geometry and Jacobi manifolds with an application to geometric mechanics that has not been noted in the literature. We precisely identify categories that generalise the ordinary categories of smooth manifolds and vector bundles to account for a lack of choice of a preferred unit, which in standard differential geometry is always given by the global constant function $1$. This is what we call the `unit-free' approach. After giving a characterisation of local Lie brackets via their symbol maps we apply our novel categorical language to review Jacobi manifolds and related notions such as Lichnerowicz brackets and Jacobi algebroids. The main advantage of our approach is that Jacobi geometry is recovered as the direct unit-free generalisation of Poisson geometry, with all the familiar notions translating in a straightforward manner. We then apply this formalism to the question of whether there is a unit-free generalisation of Hamiltonian mechanics. We identify the basic categorical structure of ordinary Hamiltonian mechanics to argue that it is indeed possible to find a unit-free analogue. This work serves as a prelude to the investigation of dimensioned structures, an attempt at a general mathematical framework for the formal treatment of physical quantities and dimensional analysis.
\end{abstract}

\newpage

\tableofcontents

\newpage

\section{Introduction} \label{Intro}

Geometric mechanics, a long-standing discipline with roots in Enlightenment science, has been established as the modern mathematical framework where classical mechanics, with its links to relativity and quantum physics, is best understood. The fundamental notion upon which geometric mechanics is built is the concept of `phase space', a set (generally a smooth manifold) representing the collection of all possible states of a given physical system, where we can identify physical observables with the real-valued functions. Hamiltonian mechanics, the theory found at the core of both classical and quantum dynamics, considers further geometric structure on the phase space that endows the space of observables with the structure of a Poisson algebra. For a recent account of the unifying and diverging patterns between classical and quantum kinematics see \cite{zalamea2016chasing}.\newline

It is common for theoreticians to develop mathematical models where physical quantities are abstracted simply to be real numbers (or approximations thereof), leaving any considerations about units of measurement to the ulterior application of their theories. This is particularly evident in the case of the aforementioned phase space formalism, where all the observables of a physical system are collectively considered as part of the same set of real-valued functions. In practice, when considering a concrete observable, one operates as if an arbitrary unit has been fixed experimentally ``outside'' the mathematical model, so that a single real value characterises the physical quantity being measured. Such an assumption makes a real-valued function on phase space into a valid abstract representation for an observable of the physical system under consideration. This sort of convention permeates all of geometric mechanics and, indeed, a great deal of modern theoretical physics. Used carefully and systematically, it has proven a powerful tool to connect abstract mathematics and practical science, however, from a formalist point of view, we claim that: 
\begin{center}
    \centering
    \emph{It would be desirable to have mathematical models where the freedom of choice of units of measurement is explicitly accounted for, appearing as a moving part of the formalism}.
\end{center}

Following this principle, we aim to develop a mathematical framework where physical observables are generally defined as \emph{unit-free} and a choice of unit is a precise construction that will allow to recover the usual notion of observable as a real-valued function. Furthermore, we will attempt to generalise phase space theory so as to naturally incorporate a \emph{unit-free} version of Hamiltonian mechanics and the canonical formalism.\newline

As it will be argued, this endeavour will inevitably take us to the realm of Jacobi geometry. The preponderance of symplectic and Poisson structures in geometric mechanics comes as no surprise given the historical origins of the discipline. However, Jacobi and contact structures, in some sense the odd-dimensional sisters of Poisson and symplectic ones, have received much less attention in the last couple of centuries of developments in geometric mechanics and theoretical physics. Fortunately, Jacobi geometry has seen a rise in popularity lately, both as a mathematical subject and due to its applications to physics and other sciences. Just a few examples of recent developments are: the integrability of Jacobi manifolds by contact groupoids \cite{crainic2015jacobi}, the generalization of Dirac geometry for Jacobi manifolds \cite{vitagliano2018dirac} \cite{schnitzer2019normal}, a dissipative version of Liouville's theorem in contact manifolds \cite{bravetti2015liouville}, the identification of a contact structure in thermodynamics \cite{mrugala1991contact} \cite{grmela2014contact} or even applications to neuroscience \cite{petitot2017elements}. More importantly for geometric mechanics, the work of de Le\'on and collaborators \cite{deLeon2017cosymplectic} \cite{valcazar2018contact} \cite{bravetti2020invariant} has shown how contact geometry can provide a natural framework for the dynamical formulation of mechanical systems subject to time-dependent forces and dissipative effects. An interesting branch of further research would be to investigate how the formalism of \emph{unit-free} Hamiltonian mechanics presented in this work fits with conventional Hamiltonian mechanics of forced and dissipative systems, since they are both described in terms of Jacobi and contact structures.\newline

This paper should be regarded as yet another avenue of application of Jacobi geometry to physics that had not been previously explored in the literature.\newline

As a closing remark, it will be pointed out that the \emph{unit-free} formalism, although largely successful in accounting for some basic notions of geometric mechanics, crucially lacks some key features of a complete theory of classical mechanics, such as a commutative product of observables or a tensor product of phase spaces. This hints at a larger category of new structures where such a complete description may be possible: the \emph{dimensioned structures} first introduced in \cite{zapata2019landscape}.\newline

The contents are organised as follows:\newline

In Section \ref{Preliminaries} we present a fairly detailed account of the theory of line bundles building on the approach of Vitagliano \cite{vitagliano2018dirac} and collaborators \cite{tortorella2017deformations} \cite{schnitzer2019normal}, with some care to provide detailed definitions and constructions that are often used in the literature but rarely elaborated upon. In preparation for the \emph{unit-free} approach to line bundles and Jacobi geometry, in Section \ref{LVectorBundles} we define the category of line-vector (lvector) bundles. Section \ref{DerJetBundles} shows how the categories of line bundles and lvector bundles directly generalise the ordinary categories of manifolds and vector bundles while retaining most of the structures that are found in those. Section \ref{TrivialLineBundles} contains a brief summary of results for trivial line bundles. In Sections \ref{JacobiAlgebroids} and \ref{LDiracGeometry} Jacobi algebroids and Dirac-Jacobi structures (what we call LDirac structures, for short) are naturally identified  within the category of lvector bundles in direct analogy with Lie algebroids and Dirac structures within the category of ordinary vector bundles.\newline

In Section \ref{LocalLieAlgebras} we summarise a few basic facts about general local Lie algebras due to Kirillov \cite{kirillov1976local} and introduce a characterisation of derivative Lie algebras, a class of local Lie algebras including Jacobi structures and Lie algebroids. Section \ref{JacobiManifolds} presents the familiar theory of Jacobi manifolds in terms of the categories introduced in Section \ref{Preliminaries}, very much in the spirit of Tortorella \cite{tortorella2017deformations}. In Section \ref{UnitPoisson} we introduce the notion of unit on a Jacobi manifold in order to recover the known notions of Jacobi brackets of functions in the sense of Lichnerowicz \cite{lichnerowicz1978jacobi}, in particular conformal Poisson and \emph{unit-free} Poisson manifolds are clearly identified in this context. Section \ref{JacobiManifoldAlgebroids} contains the usual constructions of the Jacobi algebroid associated with a Jacobi manifold and the correspondence between linear Jacobi manifolds and Jacobi algebroids.\newline

In Section \ref{UnitFreeHamiltonian} we present all the arguments supporting the thesis that Jacobi geometry provides the natural context for a \emph{unit-free} theory of Hamiltonian mechanics. To this end we firstly summarise the standard canonical formalism of Hamiltonian mechanics on cotangent bundles of configuration spaces in Section \ref{OrdinaryHamiltonian} by identifying what we call the Hamiltonian functor. Then, in Section \ref{CanonicalContact} we present a novel approach to the constructions involving the canonical contact structure on the first jet of a line bundle. By introducing the notion of \emph{unit-free} observable as a section of a line bundle over phase space in Section \ref{UnitFreeGeneralisation}, we finish by showing in Section \ref{UnitFreeCanonicalHamiltonian} that the Hamiltonian functor naturally generalises to the unit-free setting.\newline

In Section \ref{Conclusion} we comment on the success of our approach and discuss some of the features of conventional phase space theories that are missing in our \emph{unit-free} formulation. In particular, we note the lack of an explicit algebraic structure on the space of \emph{unit-free} observables, which prevents us from formally considering products of \emph{unit-free} observables, and the difficulties encountered when trying to make sense of the tensor product of the Lie algebras of two Jacobi manifolds. These issues motivate the definition of \emph{dimensioned structures}, which generalise ordinary algebraic structures by considering a partially-defined addition.

\section{Line Bundles, LVector Bundles and Jacobi Algebroids} \label{Preliminaries}

\subsection{The Categories of Lines and LVector Spaces} \label{Lines}

We identify \textbf{the category of lines}, $\Line$, as a subcategory of vector spaces $\Vect$. Objects are vector spaces over the field of real numbers $\mathbb{R}$ of dimension 1, a useful way to think of these in the context of the present work is as sets of numbers without the choice of a unit. An object $L\in\Line$ will be called a \textbf{line}. A morphism in this category $B:L\to L'$, is an invertible linear map. Composition in the category $\Line$ is simply the composition of maps. If we think of $L$ and $L'$ as numbers without a choice of a unit, a morphism $B$ between them can be thought of as a unit-free conversion factor, for this reason we will often refer to a morphism of lines as a \textbf{factor}. We consider the field of real numbers, trivially a line when regarded as real a vector space, as a singled out object in the category of lines $\Real\in \Line$.\newline 

Note that all the morphisms in this category are, by definition, isomorphisms, thus making $\Line$ into a groupoid; however, one should not think of all the objects in the category as being equivalent. As we shall see below, there are times when one finds factors between lines (invertible morphisms by definition) in a canonical way, that is, without making any further choices beyond the information that specifies the lines, these will be called \textbf{canonical factors}. \newline

It is a simple linear algebra fact that any two lines $L,L'\in\Line$ satisfy
\begin{equation*}
    \dimm (L \oplus L')= \dimm L + \dimm L' = 2 > 1, \qquad \dimm L^* = \dimm L = 1, \qquad \dimm (L \otimes L') = 1.
\end{equation*}
Then, we note that the direct sum $\oplus$ is no longer defined in $\Line$, however, it is straightforward to check that $(\Line, \otimes, \mathbb{R})$ forms a symmetric monoidal category and that $*:\Line \to \Line$ is a duality contravariant autofunctor. The usual isomorphisms associated with the monoidal structure and the duality functor induce canonical factors between combinations of tensor products and duals of lines. In particular, $\Vect(L,L)\cong L^*\otimes L\in\Line$ has a distinguished non-zero element, the identity id$_L$, thus we find that it is canonically isomorphic to $\Real$ as lines. Therefore, for any line $L\in\Line$ we find a canonical factor
\begin{equation*}
    p_L:L^*\otimes L\to \Real.
\end{equation*}
This last result, under the intuition of lines as numbers without a choice of unit, allows us to reinterpret the singled out line $\Real$ informally as the set of procedures common to all lines by which a number gives any other number in a linear way (preserving ratios). This interpretation somewhat justifies the following adjustment in terminology: we will refer to the tensor unit $\Real\in\Line$ as \textbf{the patron line}. In anticipation of our discussion about units of measurement and physical quantities of Section \ref{UnitFreeGeneralisation}, the term \textbf{unit} is reserved for non-vanishing elements of a line $u\in L^\bullet$, where we have denoted $L^\bullet:= L\backslash \{0\}$. \newline

We can see the role of the patron line $\Real$ more explicitly by considering the following map for any line $L\in\Line$:
\begin{align*}
\lambda: L\times L^\bullet & \to L^*\otimes L\\
(a,b) & \mapsto \lambda_{ab} \text{ such that } a=\lambda_{ab}(b).
\end{align*}
And, reciprocally, also define $\rho:L^\bullet\times L\to L^*\otimes L$. The following proposition gives mathematical foundation to the intuition of $L$ being a unit-free field of numbers and $\Real$ being the patron of ratios between unit-free numbers.
\begin{prop}[Ratio Maps]\label{RatMaps}
The maps $\lambda$ and $\rho$ are well-defined and linear in their vector arguments; furthermore, for any $a,b,c\in L^\bullet$, we have the following identities for the maps $l:=p_L\circ \lambda$ and $r:=p_L\circ \rho$
\begin{equation}\label{ratio1}
    l_{ab}\cdot l_{bc}\cdot l_{ca}=1, \qquad l_{ab}\cdot r_{ab}=1, \qquad r_{ab}\cdot r_{bc}\cdot r_{ca}=1.
\end{equation}
The maps $l$ and $r$ are called the \textbf{ratio maps} and the first and third equations will be called the \textbf{2-out-of-3 identity} for the maps $l$ and $r$ respectively.
\end{prop}
\begin{proof}
This statement is essentially a reformulation of the fact that all three $a,b,c\in L^\bullet$ are choices of basis in the 1-dimensional vector space $L$. Since both $L$ and $\text{Hom}_\Vect(L,L)$ are 1-dimensional real vector spaces, it follows that there is a unique linear isomorphism mapping any two non-zero elements in $L$ and the zero element is only mapped by the zero linear map. This can be seen more explicitly with the use of the canonical factor $p_L:\Hom{L,L}\to \Real$, which allows us to define the map $l:=p_L\circ \lambda$ that when applied to two non-zero elements $a,b\in L$ gives the unique non-zero real number $l_{ab}$ acting as proportionality factor: $a=l_{ab}\cdot b$. To prove the 2-out-of-3 identity consider $a=\lambda_{ab}(b)$, $b=\lambda_{bc}(c)$ and $c=\lambda_{ca}(a)$. Combining the three equations we find $\lambda_{bc}\circ\lambda_{ca}(a)=\lambda_{ba}(a)$. Noting that $\lambda_{ab}=\lambda^{-1}_{ba}$ gives the desired result. Similarly for $\rho$ and the reciprocal identity follows by construction.
\end{proof}
Consider now two lines $L_1,L_2\in\Line$ with their corresponding maps $l^1$, $r^1$ and $l^2$, $r^2$. It follows from the definitions above that for any factor $B:L_1\to L_2$ we have $l^1_{ab}=l^2_{B(a)B(b)}$ and $r^1_{ab}=r^2_{B(a)B(b)}$. More generally, we can define functions on the space of factors between the lines $L_1$ and $L_2$ from pairs of line elements:
\begin{align}\label{ratio2}
    a_1\in L_1, b_2\in L_2^{\bullet} &\quad \mapsto \quad l^{12}_{a_1b_2}(B):=l^2_{B(a_1)b_2}\\
    b_1\in L_1^{\bullet}, a_2\in L_2 &\quad \mapsto \quad r^{12}_{b_1a_2}(B):=r^1_{b_1B^{-1}(a_2)}
\end{align}
for all $B:L_1 \to L_2$ a factor. It follows then by construction, that for any pair of non-zero line elements $b_1\in L_1^{\bullet}$ and $b_2\in L_2^{\bullet}$ the following identity holds
\begin{equation*}
    l^{12}_{b_1b_2}\cdot r^{12}_{b_1b_2} =1
\end{equation*}
as functions over the space of factors $\Line(L_1,L_2)$. These are called the \textbf{ratio functions}. Whenever there is no room for confusion, we will employ the following abuse of fraction notation for non-zero elements:
\begin{equation*}
    l_{ab}=\frac{a}{b}=\frac{1}{r_{ab}} \qquad \qquad l^{12}_{a_1b_2}=\frac{a_1}{b_2}=\frac{1}{r^{12}_{b_1a_2}}.
\end{equation*}

In the same manner to how we have generalised the field of real numbers $\Real$ to the category of lines, the notion of a vector space over the field of scalars $\Real$ ought to be generalised in a larger category. We do so by identifying the category of \textbf{line vector spaces} or \textbf{lvector spaces} defined as the product category
\begin{equation*}
    \LVect:= \Vect \times \Line.
\end{equation*}
Our notation for objects in this category will be $V^L:=(V,L)$ with $V\in\Vect$, $L\in\Line$, and similarly for morphisms $\psi^B:V^L\to W^{L'}$. Objects $V^L$ will be called \textbf{lvector spaces} and morphisms $\psi^B$ will be called \textbf{linear factors}.\newline

Aiming to recover the objects that naturally appear when discussing line bundles and Jacobi manifolds (see Sections \ref{LineBundles} and \ref{JacobiManifolds}) we extend the usual constructions for vector spaces to lvector spaces by taking a ``L-rooted'' approach. This means that the abelian, monoidal and duality categorical structures present in $\Vect$ will be generalised to $\LVect$ by preferring to fix the line component of lvector spaces and actively avoiding tensor products and duals of lines. The first clear example of the L-rooted approach is our definition of \textbf{direct sums} of lvector spaces sharing line component:
\begin{equation*}
    V^L\oplus_L W^L:=(V\oplus W)^L.
\end{equation*}
The notion of \textbf{subspace} and \textbf{quotient} are similarly defined:
\begin{equation*}
    U^L\subset V^L \text{ when } U \subset V \text{ is subspace, } \qquad V^L/U^L:= (V/U)^L.
\end{equation*}
Since the line component of a lvector space $V^L$ plays the role of the unit-free field of scalars, the natural notion of dual space should be given by the space of linear maps from the vector space to the line, $\Vect(V,L)\cong V^*\otimes L$, we thus define \textbf{lduality} as follows: the ldual of a lvector space $V^L$ is
\begin{equation*}
    V^{*L}:=(V^*\otimes L)^L
\end{equation*}
and the ldual of a linear factor $\psi^B:V^L\to W^{L'}$ is
\begin{equation*}
    \psi^{*B}:W^{*L'}\to V^{*L}
\end{equation*}
with $\psi^{*B}(\beta^{l'})=\alpha^l$ such that
\begin{equation*}
    \alpha=B^{-1}\circ \beta \circ \psi \qquad l=B^{-1}(l').
\end{equation*}
There is also a natural notion of \textbf{lannihilator} of a subspace in a lvector space:
\begin{equation*}
    U^{0L}:=\{\alpha\in V^{*L} | \quad \alpha(u)=0\in L, \forall u\in U\}\cong U^0\otimes L.
\end{equation*}

\begin{prop}[$L$-Rooted Subcategories of $\LVect$]\label{LSumsDual}
By fixing a line $\normalfont L\in\Line$, the subcategory of lvector spaces sharing $L$ as line component form an abelian category with duality $\normalfont (\Vect^L,\oplus_L, ^{*L})$ that directly generalises the analogous categorical structures in ordinary vector spaces $\normalfont \Vect$. In particular, we have canonical isomorphic linear factors:
\begin{equation*}
    (V^{*L})^{*L}\cong V^L, \qquad (V \oplus W)^{*L}\cong V^{*L}\oplus_L W^{*L}, \qquad V^{*L}/U^{0L}\cong U^{*L}.
\end{equation*}
\end{prop}
\begin{proof}
This follows from simple linear algebra arguments exploiting the peculiarities of the category of lines $\Line$: linear maps between lines are also 1-dimensional vector spaces (hence also lines) and the endomorphisms of a line are canonically isomorphic to the patron line $L^*\otimes L\cong \Real$.
\end{proof}

We can continue to generalise linear algebra following the L-rooted approach to define \textbf{ltensors} as follows:
\begin{align*}
    \mathcal{T}^k(V^L) & := (\mathcal{T}^k(V))^L=(V\otimes \stackrel{k}{\dots} \otimes V)^L,\\
    \mathcal{T}_k(V^L) & := (\Vect(V, \stackrel{k}{\dots} ,V,L))^L=(V^*\otimes \stackrel{k}{\dots} \otimes V^*\otimes L)^L.
\end{align*}
By taking the conventions $\mathcal{T}^0(V^L)=\Real^L$ and $\mathcal{T}_0(V^L)=L$ we can define the graded lvector spaces of tensors in the obvious way:
\begin{align*}
    \mathcal{T}^\bullet (V^L) & := \bigoplus_{k=0}^\infty \mathcal{T}^k(V^L),\\
    \mathcal{T}_\bullet (V^L) & := \bigoplus_{k=0}^\infty \mathcal{T}_k(V^L).
\end{align*}
Given a linear factor $\psi^B:V^L\to W^{L'}$ we can define its \textbf{push-forward} simply from the ordinary push-forward of contravariant tensors
\begin{equation*}
    \psi^B_*:= (\psi_*)^B :\mathcal{T}^k(V^L) \to \mathcal{T}^k(W^{L'}),
\end{equation*}
where we take the convention that at degree $0$ the push-forward $\psi^B_*:\Real \to \Real$ is simply the identity $\Id_\Real$. The \textbf{pull-back} a linear factor is obtained by extending the definition of ldual map above for an arbitrary number of vector arguments
\begin{equation*}
    \psi^{*B}: \mathcal{T}_k(W^{L'}) \to \mathcal{T}_k(V^L).
\end{equation*}
At degree $0$ the pull-back $\psi^{*B}: L' \to L$ is simply the inverse of the factor $B^{-1}$. Note that the definition of the $\LVect$ category has forced us to define contravariant and covariant tensors in an asymmetrical way. This peculiarity of the L-rooted approach is further exacerbated when we attempt to identify a generalisation of the tensor product of lvector spaces. One could simply use the Cartesian product of the monoidal structures in $\Vect$ and $\Line$ to endow $\LVect$ with a natural monoidal structure, however, this will produce tensor products of lines, which we are trying to avoid. This will imply that the tensor algebra of an ordinary vector space doesn't generalise to an object of the same algebraic nature in the context of lvector spaces. In fact, the price we pay for enforcing the L-rooted approach is that we give up algebra structures for modules over algebras. In the case of contravariant ltensors this is evident simply by using the ordinary tensor product of the vector component of lvector spaces: let $a\in \mathcal{T}^p(V)$ and $b^l\in \mathcal{T}^q(V^L)$ then we define the module product as
\begin{equation*}
    a\cdot b^l := (a\otimes b)^l \in \mathcal{T}^{p+q}(V^L).
\end{equation*}
The following proposition motivates to define \textbf{module of contravariant ltensors} as the (infinite-dimensional) lvector space $\mathcal{T}^\bullet (V^L)$ with the module product above extended by linearity.

\begin{prop}[Modules of Contravariant LTensors]\label{ContraLTensor}
The space of contravariant ltensors of a given lvector space $V^L$ with the product defined above $(\mathcal{T}^\bullet (V^L),+,\cdot)$ is a module over the associative algebra of contravariant tensors of the vector component $(\mathcal{T}^\bullet (V),+,\otimes)$. A linear factor $\psi^B:V^L\to W^{L'}$ induces a morphism of modules covering a morphism of associative algebras via its push-forward $\psi^B_*$, i.e. for all $a\in \mathcal{T}^p(V)$ and $b,c\in \mathcal{T}^q(V^L)$
\begin{equation*}
    \psi^B_*(b+c)=\psi^B_*b+\psi^B_*c, \qquad \psi^B_*(a\cdot b)= \psi_*a \cdot \psi^B_*b.
\end{equation*}
\end{prop}
\begin{proof}
This follows trivially from the fact that the module structure is essentially just the ordinary tensor product of contravariant tensors on $V$.
\end{proof}

For covariant tensors we can proceed explicitly by using the vector space structure of lines: let $\alpha \in \mathcal{T}_p(V)$ and $\beta\in \mathcal{T}_q(V^L)$, the module product
\begin{equation*}
    \alpha \cdot \beta \in \mathcal{T}_{p+q}(V^L)
\end{equation*}
is defined by its action on vector arguments
\begin{equation*}
    \alpha \cdot \beta (v_1,\dots v_p,w_1,\dots w_q):= \alpha(v_1,\dots v_p) \cdot \beta(w_1,\dots w_q)\in L.
\end{equation*}
This product indeed allows us to define the \textbf{module of covariant ltensors}.

\begin{prop}[Modules of Covariant LTensors]\label{CoLTensor}
The space of covariant ltensors of a given lvector space $V^L$ with the product defined above $(\mathcal{T}_\bullet (V^L),+,\cdot)$ is a module over the associative algebra of covariant tensors of the vector component $(\mathcal{T}_\bullet (V),+,\otimes)$. A linear factor $\psi^B:V^L\to W^{L'}$ induces a morphism of modules covering a morphism of associative algebras via its pull-back $\psi^{*B}$, i.e. for all $\alpha\in \mathcal{T}_p(V)$ and $\beta,\gamma\in \mathcal{T}_q(V^L)$
\begin{equation*}
    \psi^{*B}(\beta+\gamma)=\psi^{*B}\beta+\psi^{*B}\gamma, \qquad \psi^{B*}(\alpha\cdot \beta)= \psi^*\alpha \cdot \psi^{*B}\beta.
\end{equation*}
\end{prop}
\begin{proof}
This follows trivially from basic linear algebra arguments noting that the module product is simply constructed from the vector space structure of lines and the fact that pull-backs are explicitly given by:
\begin{equation*}
    \psi^{*B}\beta (v_1,\dots v_q):=B^{-1}(\beta(\psi(v_1),\dots \psi(v_q)), \qquad \psi^{*B}(l):=B^{-1}(l)
\end{equation*}
for all $\beta\in \mathcal{T}_q(V^L)$ and $l\in L'$.
\end{proof}

The module structures defined above admit natural symmetrisation and antisymmetrisation. In particular, we can identify the \textbf{exterior modules} of \textbf{lmultivectors} and \textbf{lforms} in the obvious way:
\begin{align*}
    \wedge^\bullet (V^L) & := (\wedge^\bullet (V))^L,\\
    \wedge_\bullet (V^L) & := (\wedge^\bullet (V^*)\otimes L)^L.
\end{align*}
It is then a simple check to see that these are finite-dimensional lvector spaces with module structures over the ordinary exterior algebras of multivectors $(\wedge^\bullet (V),+,\wedge)$ and forms $(\wedge^\bullet (V^*),+,\wedge)$ respectively. Furthermore, push-forwards and pull-backs restrict to module morphisms of lmultivectors and lforms respectively.\newline

In Section \ref{Conclusion} we briefly discuss the unit-free generalisations of conventional linear algebra that result when we don't take the L-rooted approach and consider tensor powers and duals of lines.

\subsection{The Category of Line Bundles} \label{LineBundles}

As it is customary in differential geometry, now that we have identified some interesting structures at the linear level it is time to smoothly ``smear'' them on manifolds and develop the corresponding bundle generalisation. This will result in the identification of the category of line bundles $\Line_\Man$ and the category of lvector bundles $\LVect_\Man$ whose objects will be fibrations with bases in the category of smooth manifold $\Man$ and fibres in the corresponding linear categories $\Line$ and $\LVect$, respectively. In the present section and in Section \ref{LVectorBundles} we shall give precise definitions for these categories and show some elementary constructions in them.\newline

We define the \textbf{category of line bundles} $\Line_\Man$ as the subcategory of $\Vect_\Man$ whose objects are rank $1$ vector bundles $\lambda: L \to M$ and whose morphisms are regular, i.e. fibre-wise invertible, bundle morphisms covering general smooth maps
\begin{equation*}
\begin{tikzcd}
L_1 \arrow[r, "B"] \arrow[d, "\lambda_1"'] & L_2 \arrow[d, "\lambda_2"] \\
M_1 \arrow[r, "\varphi"'] & M_2
\end{tikzcd}
\end{equation*}
In the interest of brevity, we may refer to line bundles $L\in \Line_\Man$ as \textbf{lines} and regular line bundle morphisms $B\in\Line_\Man(L_1,L_2)$ as \textbf{factors}. A factor covering a diffeomorphism, i.e. a line bundle isomorphism, is called a \textbf{diffeomorphic factor}. Similarly, a factor covering an embedding or submersion is called an \textbf{embedding factor} or \textbf{submersion factor}, respectively.\newline

The usual structural constructions on smooth manifolds -- submanifolds, quotients and products -- find a natural generalisation within the category of line bundles as illustrated by the next few propositions.

\begin{prop}[Submanifolds of Line Bundles]\label{SubManLineBundle}
Let an embedded manifold $i:S\hookrightarrow M$ and a line bundle $\lambda: L \to M$, then restriction to the submanifold gives a canonical embedding factor
\begin{equation*}
\begin{tikzcd}
L_S \arrow[r, "\iota"] \arrow[d] & L \arrow[d] \\
S \arrow[r,hook, "i"'] & M
\end{tikzcd}
\end{equation*}
where $L_S:=i^*L$.
\end{prop}
\begin{proof}
This follows trivially by construction of pull-back bundle.
\end{proof}

Let $G$ be a Lie group and $\lambda:L\to M$ a line bundle, we say that \textbf{$G$ acts on $L$} and denote $G\Acts L$ when there is a smooth map $\Phi:G\times L\to L$ such that $\Phi_g:L\to L$ is a diffeomorphic factor for all $g\in G$ and the usual axioms of a group action are satisfied
\begin{equation*}
    \Phi_g\circ \Phi_h = \Phi_{gh}\qquad \Phi_{e}=\Id_L \qquad \forall g,h\in G.
\end{equation*}
We call the map $\Phi$ a \textbf{line bundle $G$-action}. It follows by construction that any such action induces a standard group action of $G$ on the base $\phi:G\times M\to M$. The \textbf{orbits} of a line bundle $G$-action $\Phi$ can be simply defined as the images of all group elements acting on a single fibre, and thus they are naturally regarded as the line bundle restricted to the orbits of the base action $\phi$. In analogy with the case of smooth actions, we denote the set of orbits by $L/G$. Any notion defined for usual group actions on smooth manifolds extends to a notion for line bundle actions simply requiring the base action to satisfy the corresponding conditions.

\begin{prop}[Group Quotient of Line Bundles]\label{QuotientLineBUndle}
Let a line bundle $\lambda: L \to M$ and a free and proper line bundle action $G\Acts L$, then there is a canonical submersion factor
\begin{equation*}
\begin{tikzcd}
L \arrow[r, "\zeta"] \arrow[d] & L/G \arrow[d] \\
M \arrow[r,two heads, "q"'] & M/G
\end{tikzcd}
\end{equation*}.
\end{prop}
\begin{proof}
By assumption, the base action is free and proper, then $q:M\to M/G$ is simply the surjection to the smooth manifold of orbits. The line bundle structure on the set of orbits $L/G$ is induced by the fact that all pairs of fibres over the same orbit are mapped isomorphically by some group element, allowing to identify a single fibre on $M/G$ as an equivalence class of fibre on $M$ and $\zeta:L\to L/G$ being the canonical projection to the quotient.
\end{proof}

Let us now define the analogue of the Cartesian product of manifolds in the category $\Line_\Man$. Consider two line bundles $\lambda_i:L_i\to M_i$, $i=1,2$, we will use the notations $L_i$ and $L_{M_i}$ indistinctly. We begin by defining the set of all linear invertible maps between fibres:
\begin{equation*}
    M_1 \dtimes M_2:=\{B_{x_1x_2}:L_{x_1}\to L_{x_2}, \text{ factor }, (x_1,x_2)\in M_1\times M_2\},
\end{equation*}
we call this set the \textbf{base product} of the line bundles. Let us denote by $p_i:M_1\dtimes M_2 \to M_i$ the obvious projections.

\begin{prop}[Base Product of Line Bundles]\label{BaseProductLineBundles}
The base product $M_1\dtimes M_2$ is a smooth manifold, furthermore, the natural $\Real^\times$-action given by fibre-wise multiplication makes $M_1\dtimes M_2$ into a principal bundle
\begin{equation*}
    \begin{tikzcd}[column sep=0.1em]
        \Real^\times & \Acts & M_1\dtimes M_2  \arrow[d,"p_1\times p_2"]\\
         & & M_1\times M_2.
    \end{tikzcd}
\end{equation*}
\end{prop}
\begin{proof}
This is shown by taking trivializations $L|_{U_i}\cong U_i\times \Real$, $i=1,2$ that give charts of the form $U_1\times U_2\times \Real^\times$ for the open neighbourhoods $(p_1\times p_2)^{-1}(U_1 \times U_2)\subset M_1\dtimes M_2$. The cocycle condition of the transition functions of the two line bundles gives the local triviality condition for the $\Real^\times$-action to define a principal bundle.
\end{proof}

The construction of the base product of two line bundles $M_1\dtimes M_2$ allows to identify a factor $B:L_1\to L_2$ covering a smooth map $\varphi:M_1\to M_2$ with a submanifold that we may regard as the line bundle analogue of a graph:
\begin{equation*}
    \LGrph{B}:=\{C_{x_1x_2}:L_{x_1}\to L_{x_2}|\quad x_2=\varphi(x_1), \quad C_{x_1x_2}=B_{x_1}\} \subset M_1\dtimes M_2,
\end{equation*}
we call this submanifold the \textbf{lgraph} of the factor $B$.\newline

We define the \textbf{line product} of the line bundles as $L_1\utimes L_2 :=p_1^*L_1$, which is a line bundle over the base product 
\begin{equation*}
    \lambda_{12}:L_1\utimes L_2\to M_1\dtimes M_2.
\end{equation*}

\begin{prop}[Product of Line Bundles]\label{ProductLineBundle}
The line product construction $\utimes$ is a well-defined categorical product
\begin{equation*}
    \normalfont\utimes:\Line_\Man \times \Line_\Man \to \Line_\Man.
\end{equation*}
\end{prop}
\begin{proof}
Despite the apparent asymmetry of the definition, note that we can define the following factors
\begin{align*}
    P_1(B_{x_1x_2},l_{x_1})&:= l_1\in L_{x_1}\\
    P_2(B_{x_1x_2},l_{x_1})&:= B_{x_1x_2}(l_{x_1})\in L_{x_2}
\end{align*}
where $(B_{x_1x_2},l_{x_1})\in p_1^*L_1$, thus giving the following commutative diagram
\begin{equation}\label{LineProductCommutativeDiagram}
\begin{tikzcd}
L_1 \arrow[d, "\lambda_1"'] & L_1\utimes L_2 \arrow[l,"P_1"']\arrow[d, "\lambda_{12}"]\arrow[r,"P_2"] & L_2\arrow[d,"\lambda_2"] \\
M_1 & M_1 \dtimes M_2\arrow[l,"p_1"]\arrow[r,"p_2"'] & M_2
\end{tikzcd}
\end{equation}
where $P_1$ and $P_2$ are submersion factors. It is clear from this definition that taking the pull-back bundle $p_2^*L_2$ instead of $p_1^*L_1$ as our definition for line product will give the line bundle that we have denoted by $L_2\utimes L_1$. By construction of the base product $M_1\dtimes M_2$ we can construct the following smooth map:
\begin{align*}
c_{12}: M_1\dtimes M_2 & \to M_2\dtimes M_1\\
B_{x_1x_2} & \mapsto B_{x_1x_2}^{-1}
\end{align*}
which is clearly invertible, $c_{12}^{-1}=c_{21}$, and induces a factor on the line products by setting
\begin{equation*}
    C_{12}(B_{x_1x_2},l_{x_1})=(B_{x_1x_2}^{-1},B_{x_1x_2}(l_{x_1})).
\end{equation*}
We have thus found a canonical factor covering a diffeomorphism
\begin{equation*}
\begin{tikzcd}
L_1\utimes L_2 \arrow[r,"C_{12}"]\arrow[d, "\lambda_{12}"] & L_2\utimes L_1\arrow[d, "\lambda_{21}"] \\
M_1 \dtimes M_2\arrow[r,"c_{12}"'] & M_2 \dtimes M_1
\end{tikzcd}
\end{equation*}
hence proving that the definition of line product is indeed symmetrical, i.e the two products are canonically isomorphic as line bundles $L_1\utimes L_2\cong L_2\utimes L_1$. For $\utimes$ to be a categorical product in $\Line_\Man$ first we need the two canonical projection morphisms, these are given by the factors $P_1$ and $P_2$ defined above
\begin{equation*}
\begin{tikzcd}
L_1 & L_1\utimes L_2 \arrow[l,"P_1"']\arrow[r,"P_2"] & L_2.
\end{tikzcd}
\end{equation*}
Now we must check the universal property that, given two morphisms $B_i:L\to L_i$, $i=1,2$, we get a unique morphism $B:L\to L_1 \utimes L_2$ such that the following diagram commutes
\begin{equation*}
\begin{tikzcd}
 & L\arrow[dl,"B_1"'] \arrow[dr,"B_2"]\arrow[d,"B"] & \\
L_1 & L_1\utimes L_2 \arrow[l,"P_1"]\arrow[r,"P_2"'] & L_2.
\end{tikzcd}
\end{equation*}
We can indeed define the \textbf{line product of factors} as
\begin{align*}
B_1\utimes B_2: L & \to L_1\utimes L_2\\
l_x & \mapsto (B_2|_x\circ B_1|_{x}^{-1},B_1(l_x))
\end{align*}
noting that, by assumption, $B_2|_x\circ B_1|_{x}^{-1}:L_{\varphi_1(x)}\to L_{\varphi_2(x)}\in M_1\dtimes M_2$ since factors are fibre-wise invertible. We readily check that setting $B=B_1\utimes B_2$ makes the above diagram commutative, thus completing the proof.
\end{proof}

If we consider a single line bundle $\lambda:L\to M$ the construction of the base product $M\dtimes M$ coincides with the usual definition of the \textbf{general linear groupoid} $p_1,p_2:\text{GL}(L)\to M$ with the projections $p_1$ and $p_2$ acting as the source and target maps respectively. This is clearly a transitive groupoid, with single orbit $M$, and the isotropies are $\text{GL}(L_x)$ which form a bundle of Lie groups isomorphic to the frame bundle $L^\times:= L^*\backslash \{0\}$. The set of diffeomorphic factors from $L$ to itself is called the \textbf{group of automorphisms} $\text{Aut}(L)$ and it is a simple check to verify that they correspond to the bisections of the general linear groupoid $\text{GL}(L)$. These objects represent the intrinsic symmetries of line bundles; indeed, the unit-free analogue of how the groups of diffeomorphisms capture the intrinsic symmetries of ordinary manifolds.\newline

Having shown that the category of line bundles is a direct generalisation of the category of smooth manifolds in terms of geometric constructions, we now move on to discuss how sections of line bundles represent a unit-free generalisation of the functions on a manifold. The facts discussed in the remainder of this section should convey the notion that modules of sections are the algebraic counterpart of line bundles in the same way that commutative algebras (rings) of functions are the algebraic counterpart of manifolds.\newline

Let a $B:L_1\to L_1$ be a factor covering the smooth map $\varphi:M_1\to M_2$, we can define the \textbf{factor pull-back} on sections as:
\begin{align*}
B^*: \Sec{L_2} & \to \Sec{L_1}\\
s_2 & \mapsto B^*s_2
\end{align*}
where
\begin{equation*}
    B^*s_2(x_1):=B^{-1}_{x_1}s_2(\varphi(x_1))
\end{equation*}
for all $x_1\in M_1$. Note that the fibre-wise invertibility of $B$ has critically been used for this to be a well-defined map of sections. The $\Real$-Linearity of $B$ implies the following interaction of factor pull-backs with the module structures:
\begin{equation*}
    B^*(f\cdot s)=b^*f\cdot B^*s \quad \forall f\in\Cin{M_2}, s\in \Sec{L_{M_2}}.
\end{equation*}
Then, if we consider $\textsf{RMod}$, the category of modules over rings with module morphisms covering ring morphisms, the assignment of sections becomes a contravariant functor
\begin{equation*}
    \Gamma:\Line_\Man \to \textsf{RMod}.
\end{equation*}
This is in contrast with the similar situation for sections of general vector bundles, since pull-backs of sections are not always defined for non-regular vector bundle morphisms.\newline

We may give an algebraic characterization of submanifolds of the base space of a line bundle as follows: for an embedded submanifold $i:S\hookrightarrow M$ we define its \textbf{vanishing submodule} as
\begin{equation*}
    \Gamma_S:=\Ker{\iota^*}=\{s\in\Sec{L_M}|\quad s(x)=0\in L_x \quad \forall x\in S\}.
\end{equation*}
This is the line bundle analogue to the characterization of submanifolds with their vanishing ideals. In fact, the two notions are closely related since (depending on the embedding $i$, perhaps only locally) we have
\begin{equation*}
    \Gamma_S=I_S\cdot \Sec{L_M},
\end{equation*}
where $I_S:=\Ker{i^*}$ is the multiplicative ideal of functions vanishing on the submanifold. This ideal gives a natural isomorphism of rings $\Cin{S}\cong \Cin{M}/I_S$ which, in turn, gives the natural isomorphism of $\Cin{S}$-modules
\begin{equation*}
    \Sec{L_S}\cong \Sec{L}/\Gamma_S.
\end{equation*}

In the case of a free and proper action $G\Acts L$, we can identify the sections of the quotient line bundle with the submodule of \textbf{$G$-invariant sections}:
\begin{equation*}
    \Sec{L/G}\cong \Sec{L}^G:=\{s\in\Sec{L}| \quad \Phi_g^*s=s \quad \forall g\in G\}.
\end{equation*}

The construction of the line product as a pull-back bundle over the base product indicates that the sections $\Sec{L_1\utimes L_2}$ are spanned by the sections of each factor, $\Sec{L_1}$, $\Sec{L_2}$, and the functions on the base product. More precisely, we have the isomorphisms of $\Cin{M_1\dtimes M_2}$-modules:
\begin{equation*}
    \Cin{M_1\dtimes M_2}\cdot P_1^*\Sec{L_1}\cong \Sec{L_1 \utimes L_2} \cong \Cin{M_1\dtimes M_2}\cdot P_2^*\Sec{L_2},
\end{equation*}
where the second isomorphism is induced by the canonical factor $C_{12}$. We can summarize this by writing the image of the line product commutative diagram (\ref{LineProductCommutativeDiagram}) under the section functor
\begin{equation*}
\begin{tikzcd}[row sep=tiny]
\Sec{L_1} \arrow[r,"P_1^*"] & \Sec{L_1\utimes L_2}  & \Sec{L_2}\arrow[l,"P_2^*"'] \\
\bullet & \bullet & \bullet \\
\Cin{M_1}\arrow[r,hook,"p_1^*"'] & \Cin{M_1 \dtimes M_2} & \Cin{M_2}\arrow[l,hook',"p_2^*"]
\end{tikzcd}
\end{equation*}
where $\bullet$ denotes ring-module structure. We clearly see that pull-backs of functions on the Cartesian product of base manifolds $M_1\times M_2$ form a subring of $\Cin{M_1\dtimes M_2}$ but a quick computation introducing trivializations shows that differentials of these do not span the cotangent bundle of the base product everywhere. However, we shall show below that it is possible to identify a subspace of spanning functions on $M_1 \dtimes M_2$ defined from local sections of $L_1$ and $L_2$. Consider two trivializations $L|_{U_i}\cong U_i\times \Real$, $i=1,2$ so that the spaces of non-vanishing sections over the trivializing neighbourhoods, denoted by $\Sec{L^\bullet_i}$, are non-empty. Recall that the ratio maps introduced in Section \ref{Lines} allowed us to define functions over the sets of invertible maps between lines. Working fibre-wise, we can use these together with the choice of two local sections to define the \textbf{ratio functions} on $M_1\dtimes M_2$ as follows: let the local sections $s_1\in\Sec{L_1}$ and $s_2\in\Sec{L_2^\bullet}$, then the function $\tfrac{s_1}{s_2}\in \Cin{M_1\dtimes M_2}$ is defined by
\begin{equation*}
    \frac{s_1}{s_2}(B_{x_1x_2}):=\frac{B_{x_1x_2}(s_1(x_1))}{s_2(x_2)}
\end{equation*}
where we have made use of the map $l^{12}$ in (\ref{ratio2}) in fraction notation. Similarly, using the map $r^{12}$ in (\ref{ratio2}) we can define $\tfrac{s_2}{s_1}\in \Cin{M_1\dtimes M_2}$ for some local sections $s_1\in\Sec{L_1^\bullet}$ and $s_2\in\Sec{L_2}$. If we consider two non-vanishing local sections $c_1\in\Sec{L_1^\bullet}$ and $c_2\in\Sec{L_2^\bullet}$, it is a direct consequence of (\ref{ratio1}) that the relation
\begin{equation*}
    \frac{c_1}{c_2}\cdot \frac{c_2}{c_1}=1.
\end{equation*}
holds on the open set $(p_1\times p_2)^{-1}(U_1 \times U_2)$ where $U_i\subset M_i$ are the neighbourhoods where the local sections are defined.
\begin{prop}[Spanning Functions of $M_1\dtimes M_2$]\label{SpanningFunctionsLineProduct}
Let $B_{x_1x_2}\in M_1\dtimes M_2$ any point in the base product of two line bundles $L_{M_1}$, $L_{M_2}$, consider trivializing neighbourhoods $x_i\in U_i\subset M_1$ so that $\Sec{L_i^\bullet}$ are non-empty; then the cotangent bundle in a neighbourhood of $B_{x_1x_2}$ is spanned by the differentials of all possible ratio functions, i.e.
\begin{equation*}
\normalfont
    \Cot(M_1\dtimes M_2)=\text{span} \left( d\frac{\Sec{L_1}}{\Sec{L_2^\bullet}} \right)
\end{equation*}
at every point of the open set $(p_1\times p_2)^{-1}(U_1 \times U_2)$.
\end{prop}
\begin{proof}
First note that it is a general fact from basic vector bundle geometry that such open neighbourhoods $U_i$ can be chosen for any points $x_i\in M_i$. The trivializations induce the following diffeomorphism
\begin{equation*}
    (p_1\times p_2)^{-1}(U_1 \times U_2) \cong U_1\times U_2\times \Real^\times
\end{equation*}
and thus we can assign coordinates $(y_1,y_2,b)$, where $y_i\in\Real^{\text{dim}M_i}$ and $b>0$, to any element $B_{x_1x_2}\in (p_1\times p_2)^{-1}(U_1 \times U_2)$. Local sections $s_i\in\Sec{L_i}$ are given by  functions $f_i:U_i\to \Real$ under the trivialization. It follows directly from the definition that ratio functions have the following coordinate expressions
\begin{equation*}
    \frac{s_1}{s_2}(y_1,y_2,b)=f_1(y_1) f_2^{-1}(y_2)b
\end{equation*}
where $f_1$, $f_2$ are the trivialized expressions for the sections $s_1\in\Sec{L_1}$, $s_2\in\Sec{L_2^\bullet}$. A simple computation shows that
\begin{equation*}
    d_{(y_1,y_2,b)}\frac{s_1}{s_2}=f^{-1}_2b\cdot d_{y_1}f_1-f_1f_2^{-2}b\cdot d_{y_2}f_2+f_1f_2^{-1}\cdot d_bb
\end{equation*}
Since $f_1(y_1)$ is allowed to vanish in general, these differentials span all the cotangent vectors at any coordinate point $(y_1,y_2,b)$.
\end{proof}

\subsection{The Category of LVector Bundles} \label{LVectorBundles}

Following the idea of smoothly ``smearing'' lvector spaces as defined in Section \ref{Lines} over a manifold, we define a \textbf{lvector bundle} $\epsilon:E^L\to M$ as a pair $(E,L)$ with $E$ a vector bundle and $L$ a line bundle over the same base $M$. The fibres are indeed identified with objects in the category of lvector spaces $E_x^{L_x}\in\LVect$ for all $x\in M$. A \textbf{morphism of lvector bundles} $F^B:E_1^{L_1}\to E_2^{L_2}$ is defined as a pair $(F,B)$ with $F:E_1\to E_2$ a vector bundle morphism and $B:L_1\to L_2$ a factor of lines. Clearly, the fibre-wise components of a morphism of lvector bundles are linear factors. We have thus identified the \textbf{the category of lvector bundles}, which will be denoted by $\LVect_\Man$.\newline

The L-rooted approach discussed in section \ref{Lines} is naturally extended to the context bundles by considering \textbf{restricted subcategories} $\LVect^L$ whose objects are lvector bundles with a fixed line bundle (hence also a base manifold) and morphisms are taken to be covering the identity. Under these restrictions, we can simply apply the constructions of the linear theory of lvector spaces of Section \ref{Lines} fibre-wise and define \textbf{subbundles}, \textbf{quotients}, \textbf{direct sums}, \textbf{ldual bundles} and \textbf{ltensor bundles} in a natural way. Explicitly, for lvector bundles $\epsilon:E^L\to M$ and $\delta:D^L\to M$ we have
\begin{align*}
    E^L\oplus_L D^L & := (E\oplus_M D)^L,\\
    E^{*L} & := (E^*\otimes_M L)^L,\\
    \mathcal{T}^k(E^L) & := (E\otimes_M \stackrel{k}{\dots} \otimes_M E)^L,\\
    \mathcal{T}_k(E^L) & := (E^* \otimes_M \stackrel{k}{\dots} \otimes_M E^* \otimes_M L)^L.
\end{align*}

Beyond the restricted categories, we can exploit the structure present in manifolds and line bundles to give two further constructions of lvector bundles. Firstly, let $\epsilon:E^L\to M$ be a lvector bundle and $\phi:N\to M$ a smooth map, the \textbf{pull-back lvector bundle} is then defined in the obvious way by taking the simultaneous pull-back of both the vector bundle and the line bundle over the same base $\phi^*(E^L):=(\phi^*E)^{(\phi^*L)}$. Secondly, the presence of the categorical product of line bundles (\ref{ProductLineBundle}) ensures that we can define a compatible categorical product for lvector bundles.

\begin{prop}[Product of LVector Bundles] \label{ProductLVectorBundles}
Let $\epsilon_i: E_i^{L_i}\to M_i$, $i=1,2$, be two lvector bundles, then the \textbf{product of lvector bundles} defined by means of the line product as
\begin{equation*}
    E_1^{L_1}\boxplus E_2^{L_2}:=(p_1^*E_1\oplus p_2^*E_2)^{L_1\utimes L_2},
\end{equation*}
is a well-defined a lvector bundle over the base product $M_1\dtimes M_2$. This construction makes 
\begin{equation*}
    \normalfont\boxplus: \LVect \times \LVect \to \LVect
\end{equation*}
into a categorical product.
\end{prop}
\begin{proof}
To see that $E_1^{L_1}\boxplus E_2^{L_2}$ is well-defined we first note that the commutative diagram (\ref{LineProductCommutativeDiagram}) allows for lvector bundle pull-backs $p_i^*E_i^{L_i}\to M_1\dtimes M_2$ which, by construction of the line product $L_1 \utimes L_2$, share the same line bundle component. It is then possible to take their direct sum in the restricted category $\LVect^{L_1 \utimes L_2}$. In order to show that $\boxplus$ is a categorical product it will suffice to construct the product of two morphisms of lvector bundles $ F_i^{B_i}: E^L \to E_i^{L_i}$, $i=1,2$. The line product of factors $B_1\utimes B_2$ allows for an explicit point-wise definition:
\begin{equation*}
    F_1^{B_1}\boxplus F_2^{B_2}(a^l):=(B_2|_{p_2(x)}\circ B_1|_{p_1(x)}^{-1},F_1|_{p_1(x)}(\Proj_1(a))\oplus F_2|_{p_2(x)}(\Proj_2(a)))^{B_1\utimes B_2(l)},
\end{equation*}
for a generic element of the  element $a^l\in (E_1^{L_1}\boxplus E_2^{L_2})_x $.
\end{proof}

Sections of a lvector bundle $\epsilon: E^L:\to M$ are considered formally as $\Sec{E^L}:=\Sec{E}^{\Sec{L}}$, which we interpret as the $\Cin{M}$-module generalisation of lvector spaces, i.e. pairs of modules over the same ring with one of them serving as base module and the other as the line module together with the axioms of lvector space. We call these \textbf{lmodules}. The following lmodule isomorphisms are direct consequences of the fibre-wise constructions on the restricted subcategories above:
\begin{align*}
    \Sec{E^L\oplus_L D^L} & \cong \Sec{E^L}\oplus_{\Sec{L}}\Sec{D^L},\\
    \Sec{E^{*L}} & \cong \Sec{E}^{*\Sec{L}},\\
    \Sec{\mathcal{T}^k(E^L)} & \cong (\Sec{\mathcal{T}^k(E)})^{\Sec{L}},\\
    \Sec{\mathcal{T}_k(E^L)} & \cong (\Sec{\mathcal{T}_k(E)} \otimes_{\Cin{M}} \Sec{L})^{\Sec{L}}.
\end{align*}
The linear module structures of ltensors of Propositions \ref{ContraLTensor} and \ref{CoLTensor} make $\Sec{\mathcal{T}^\bullet(E^L)}$ into a $(\Sec{\mathcal{T}^\bullet(E)},\otimes)$-module and $\Sec{\mathcal{T}_\bullet(E^L)}$ into a $(\Sec{\mathcal{T}_\bullet(E)},\otimes)$-module respectively.\newline

The same difficulties one encounters when discussing maps induced on sections by ordinary vector bundle morphisms appear in the context of lvector bundles. Consider a lvector bundle morphism $F^B:E_1^{L_1}\to E_2^{L_2}$ covering a smooth map $\varphi:M_1\to M_2$, in general we can define the notion of \textbf{$F^B$-relatedness} as follows: we say that to sections $s_1^{a_1}\in \Sec{E_1^{L_1}}$ and $s_2^{a_2}\in \Sec{E_2^{L_2}}$ are $F^B$-related when
\begin{equation*}
    s_1^{a_1} \sim_F^B s_2^{a_2} \Leftrightarrow s_1 \sim_F s_2 \quad \text{and} \quad a_1 \sim_B a_2
\end{equation*}
where
\begin{equation*}
    s_1 \sim_F s_2 \Leftrightarrow F\circ s_1 =s_2 \circ \varphi, \qquad a_1 \sim_B a_2 \Leftrightarrow a_1=B^*a_2.
\end{equation*}
This relation only gives a well-defined morphism of lmodules when the map $\varphi:M_1\to M_2$ is a diffeomorphism, in this case the \textbf{push-forward} $F^B_*:\Sec{E_1^{L_1}}\to \Sec{E_2^{L_2}}$ is explicitly defined 
\begin{equation*}
    F^B_*(s^a):=(F_*s)^{B_*a}=(F_*s)^{(B^{-1})^*a}.
\end{equation*}
As in the case of ordinary vector bundles, we may induce well-defined maps on sections from general lvector bundle morphisms by taking lduals. Indeed, we define the \textbf{pull-back} of a lvector bundle morphism $F^B:E_1^{L_1}\to E_2^{L_2}$ covering a smooth map $\varphi:M_1 \to M_2$ as follows:
\begin{align*}
    F^{*B} :& \Sec{E_2^{*L_2}}\to \Sec{E_1^{*L_1}}\\
    &\beta^b\to \alpha^a
\end{align*}
where
\begin{equation*}
    a=B^*b, \qquad \alpha(x)=B_x^{-1}\circ F_x^*\beta(\varphi(x))
\end{equation*}
for all $x\in M_1$ and where the second equality is as maps of the form $E_1|_x\to L_1|_x$.\newline

\begin{prop}[Induced LModule Maps] \label{InducedLModuleMaps}
A lvector bundle morphism $F^B:E_1^{L_1}\to E_2^{L_2}$ induces a morphism of lmodules of covariant ltensors via its pull-back
\begin{equation*}
    F^{*B}:\Sec{\mathcal{T}_\bullet(E_2^{L_2})}\to \Sec{\mathcal{T}_\bullet(E_1^{L_1})},
\end{equation*}
i.e.
\begin{equation*}
    F^{*B}(\alpha +\beta)=F^{*B}\alpha + F^{*B}\beta, \qquad F^{*B}(\omega\cdot \alpha)=F^*\omega \cdot F^{*B}\alpha
\end{equation*}
for all $\alpha,\beta\in \Sec{\mathcal{T}_\bullet(E_2^{L_2})}$ and $\omega\in \Sec{\mathcal{T}_\bullet(E_2)}$.
\end{prop}
\begin{proof}
The definition of pull-back at degree $0$ is simply $F^{*B}=B^*:\Sec{L_2}\to\Sec{L_1}$. In arbitrary degree, we can see that the definition pull-back on dual sections $F^{*B}:\Sec{E_2^{*L_2}}\to \Sec{E_1^{*L_1}}$ given above works just as well for covariant ltensors of higehr degree. The fact that $F^{*B}$ is additive follows trivially for the fibre-wise linear structure of the bundles involved. The compatibility with the $(\Sec{\mathcal{T}_\bullet(E)},\otimes)$-module structures can be easily checked explicitly: let $\alpha\in \Sec{\mathcal{T}_p(E_2^{L_2})}$, $\omega\in \Sec{\mathcal{T}_q(E_2)}$ and $x\in M_1$, then
\begin{align*}
    F^{*B}(\omega\cdot \alpha)_x(v_1,\dots,v_p,w_1,\dots,w_q)& =B_x^{-1}\{(\omega\cdot \alpha)_{\varphi(x)} (F_x(v_1),\dots,F_x(v_p),F_x(w_1),\dots,F_x(w_q))\}\\
    &= B_x^{-1}\{\omega_{\varphi(x)}(F_x(v_1),\dots,F_x(v_p))\cdot \alpha_{\varphi(x)} (F_x(w_1),\dots,F_x(w_q))\}\\
    &=(F^*\omega)_x(v_1,\dots,v_p)\cdot (F^{*B}\alpha)_x(w_1,\dots,w_q)\\
    &= (F^*\omega\cdot F^{*B}\alpha)_x(v_1,\dots,v_p,w_1\dots,w_q).
\end{align*}
\end{proof}
This result establishes the lmodules of covariant ltensors as the natural algebraic counterpart of lvector bundles, again in direct analogy to the similar result regarding modules of covariant tensors or ordinary vector bundles.\newline

Following an analogous reasoning to the linear case, it is easy to show that induced maps restrict nicely to symmetrisations and antisymmetrisations of the lmodules of ltensors. In particular, push-forwards (when defined) and pull-backs give morphisms of the exterior modules of \textbf{lmultisections} $\Sec{\wedge^\bullet(E^L)}$ and \textbf{lforms} $\Sec{\wedge_\bullet(E^L)}$ respectively.

\subsection{Der Bundles and Jet Bundles} \label{DerJetBundles}

In Section \ref{LineBundles} we established line bundles as the direct generalisation of smooth manifolds where the notion of unit-free functions is given by sections. We now turn our attention to the dynamical aspects of manifolds. In this section we shall find the appropriate unit-free generalisations of curves on manifolds, vector fields, tangent bundles, cotangent bundles, induced maps, etc.\newline

Let $M$ be a smooth manifold, the presence of a line bundle structure $\lambda:L\to M$ forces us to generalise the notion of curve on $M$ to a $1$-parameter group of fibre-wise factors covering an ordinary curve on the base. In this direction, we define a \textbf{curve} on a line bundle $L$ formally as a curve on the general linear groupoid $\text{GL}(L)$ passing through the identity submanifold and whose target projection gives an ordinary curve on the base manifold. Following the characterisation of line bundle automorphisms as bisections of $\text{GL}(L)$ we easily see that a curve on $L$ can be regarded as the flow of an infinitesimal line bundle automorphism $D:\Sec{L}\to \Sec{L}$, which can be checked to be a $\Real$-linear map that interacts with he module structure via the following Leibniz identity: for $s\in \Sec{L}$ and $f\in\Cin{M}$
\begin{equation}\label{ModuleDerivation}
    D(f\cdot s)=f\cdot D(s) +X[f]\cdot s
\end{equation}
where $X\in \Sec{\Tan M}$ is the vector field of the infinitesimal diffeomorphism on the base. Note that this is the line bundle analogue of the directional derivative and thus we are compelled to give the definition of \textbf{der bundle} as the bundle of such objects:
\begin{align*}
    \Der L:=\{ & a_x:\Sec{L}\to L_x \quad \Real\text{-linear} \quad|\quad \exists\quad  \delta(a_x)\in \Tan_x M, \\ 
    &a_x(f\cdot s)=f(x)a_x(s)+\delta(a_x)(f)s(x)\quad \forall f\in \Cin{M},s\in\Sec{L}\}.
\end{align*}
where the action of $\delta(a_x)$ on $f$ is as a point-wise directional derivative, which can be shown to be unique for any $a_x$ given as an infinitesimal automorphism at $x\in M$. This makes $\delta:\Der A\to \Tan M$ into a vector bundle map covering the identity. $\Der L$ is clearly a vector bundle whose fibres realize the unit-free version of local directional derivatives. In analogy with the usual notation of the action a vector field $X$ on a function $f$ as a directional derivative $X[f]$, given a section of the der bundle $a\in \Sec{\Der L}$, we define its \textbf{action as a derivation} on sections by setting
\begin{equation*}
    a[s](x):=a(x)(s)\in L_x, \quad \forall x\in M.
\end{equation*}
We then write the \textbf{Lie bracket of derivations} as a bracket on sections of $\Der L$ explicitly
\begin{equation*}
    [a,a'][s]:=a[a'[s]]-a'[a[s]]
\end{equation*}
for all $a,a'\in \Sec{\Der L}$ and $s\in\Sec{L}$. We have thus recovered the standard notion of module \textbf{derivations} of sections $\Dr{L}$, which are generally defined as $\Real$-linear endomorphisms of $\Sec{L}$ satisfying the Leibniz identity (\ref{ModuleDerivation}). The map $\delta$ is the vector bundle morphism associated with the usual symbol map of derivations (see \cite[Ch. 11]{nestruev2006smooth}), which in this case is a $\Cin{M}$-linear map $\sigma:\Dr{L}\to \Sec{\Tan M \otimes \text{End}(L)}$, via the isomorphism $\text{End}(L)\cong \Real_M$. If follows that the der bundle naturally carries the structure of a Lie algebroid $(\Der L, \delta, [,])$.\newline

The notions of der bundle and derivations present themselves as the obvious candidates for the unit-free generalisations of tangent bundles and vector fields respectively. The results that follow below on how these notions interact with factors, submanifolds, quotients and products shall indeed confirm that this is the correct interpretation.\newline

Let a factor $B:L_1\to L_2$ covering a smooth map $\varphi:M_1\to M_2$, we define its \textbf{der map} at a point $x_1\in M_1$ as follows:
\begin{align*}
\Der_{x_1} B: \Der_{x_1} L_1 & \to \Der_{\varphi(x_1)} L_2\\
a_{x_1} & \mapsto \Der_{x_1}a_{x_1} 
\end{align*}
where
\begin{equation*}
    \Der_{x_1}a_{x_1}(s_2):=B_{x_1}(a_{x_1}(B^*s_2)) \quad \forall s_2\in\Sec{L_2}.
\end{equation*}
As it will be shown in the next proposition below, the constructions of der bundles and der maps takes us naturally to the category of lvector bundles introduced in Section \ref{LVectorBundles}; in fact, they serve as the main motivation to define and study them in general in the first place. Notions of induced maps find their particular realisations in this context. Given two derivations, $a_1\in\Sec{\Der L_1}$ and $a_2\in\Sec{\Der L_2}$ we say that they are $B$\textbf{-related}, $a_1\sim_B a_2$, if the following diagram commutes
\begin{equation*}
    \begin{tikzcd}
    \Der L_1 \arrow[r, "\Der B"] & \Der L_2 \\
    M_1 \arrow[r, "\varphi"']\arrow[u,"a_1"] & M_2.\arrow[u,"a_2"]
\end{tikzcd}
\end{equation*}
When the base map of the factor, $\varphi$, is a diffeomorphism this relation defines the \textbf{push-forward} of derivations:
\begin{align*}
B_*: \Sec{\Der L_1} & \to \Sec{\Der L_2}\\
a & \mapsto \Der B \circ a \circ \varphi^{-1}.
\end{align*}
With our identification of derivations with sections of the der bundle and the definition of pull-back of sections, we can readily check the following identity:
\begin{equation*}
    B_*a[s]=(B^{-1})^*a[B^*s] \qquad \forall s\in \Sec{L_2},
\end{equation*}
which gives an alternative definition of der push-forward only in terms of pull-backs of diffeomorphic factors.

\begin{prop}[The Der Functor]\label{DerFunctor}
The der map of a factor $B:L_1 \to L_2$ gives a well-defined lvector bundle morphism
\begin{equation*}
\normalfont
\begin{tikzcd}
\Der L_1 \arrow[r, "\Der B"] \arrow[d, "\lambda_1"'] & \Der L_2 \arrow[d, "\lambda_2"] \\
M_1 \arrow[r, "b"'] & M_2
\end{tikzcd}
\end{equation*}
which is a Lie algebroid morphism, i.e.
\begin{equation*}
\normalfont
    \Tan b\circ \delta_1 = \delta_2\circ \Der B, \qquad  \qquad a_1\sim_B a_2 ,\quad a'_1\sim_B a'_2 \quad \Rightarrow \quad [a_1,a'_1]_1\sim_B [a_2,a'_2]_2
\end{equation*}
for $\normalfont a_i,a_i'\in\Sec{\Der L_i}$. Furthermore, for any other factor $F:L_2 \to L_3$ and $\normalfont \Id_L:L\to L$ the identity factor, we have
\normalfont
\begin{align*}
    \Der (F\circ B) & = \Der F \circ \Der B,\\
    \Der (\Id_L) & = \Id_{\Der L}.
\end{align*}
\end{prop}
\begin{proof}
Note first that  $B^*(f\cdot s)=b^*f\cdot B^*s$ for $f\in\Cin{M_2}$ and $s\in\Sec{L_2}$, then applying the definition of point-wise der map one obtains
\begin{equation*}
    \Der_xBa_x(f\cdot s)=f(b(x))\cdot \Der_x B a_x(s)+\delta_1(a_x)(b^*f)(x)\cdot s(b(x)).
\end{equation*}
Simply using the definition of tangent map $\Tan b$ we can rewrite the second term to find
\begin{equation*}
    \Der_xBa_x(f\cdot s)=f(b(x))\cdot \Der_x B a_x(s)+\Tan_xb\circ \delta_1(a_x)(f)(b(x))\cdot s(b(x))
\end{equation*}
which shows that indeed $\Der_xBa_x\in \Der_{b(x)}L_2$, making $\Der B$ a well-defined vector bundle morphism. We can then use the anchor $\delta_2$ to rewrite the LHS
\begin{equation*}
    \Der_xBa_x(f\cdot s)=f(b(x))\cdot \Der_x B a_x(s)+\delta_2(\Der_xB a_x)(f)(b(x))\cdot s(b(x))
\end{equation*}
and thus we obtain
\begin{equation*}
    \Tan_xb\circ \delta_1(a_x)=\delta_2(\Der_xB a_x),
\end{equation*}
which must hold for any $a_x\in\Der_xL_1$, thus giving the desired compatibility condition with the anchors. The $B$-relatedness condition follows directly from the definition of der map by noting that
\begin{equation*}
    a_2(b(x))(a'_2[s])=(\Der_x B a_1(x))(a'_2[s])=B_x(a_1(x)(B^*a_2[s]))
\end{equation*}
for all $s\in\Sec{L_2}$. The two functorial identities follow directly from the contravariant functoriality of pull-backs:

\begin{align*}
    (F\circ B)^* &=B^*\circ F^*,\\
    \Id_L^* &=\Id_{\Sec{L}}.
\end{align*}
\end{proof}

\begin{prop}[Der Bundle of a Submanifold]\label{DerBundleLineSubmanifold}
Let $\lambda:L\to M$ be a line bundle and $i:S\hookrightarrow M$ an embedded submanifold of the base. The canonical embedding factor $\iota : L_S\hookrightarrow L$ gives an injective lvector bundle morphism
\begin{equation*}
\normalfont
    \Der \iota : \Der L_S \hookrightarrow \Der L.
\end{equation*}
By a slight abuse of notation we will write $\normalfont\Der L_S\subset \Der L$ in the same way that we write $\normalfont\Tan S\subset \Tan M$, then we have
\begin{equation*}
\normalfont
    \delta_L(\Der L_S)=\Tan S,
\end{equation*}
where $\normalfont\delta_L:\Der L\to \Tan M$ is the anchor of the der bundle.
\end{prop}
\begin{proof}
That $\Der \iota$ is a lvector bundle morphism follows by construction since $L_S:=i^*L$, so the line bundle morphism is simply the fibre-wise identity map. Then, injectivity of $\Der \iota$ follows simply from injectivity of $i:S\hookrightarrow M$. Note that using the full notation the second identity in the proposition reads
\begin{equation*}
    \delta_L(\Der \iota (\Der L_S))=\Tan i(\Tan S)
\end{equation*}
which is clearly a direct consequence of $\Der \iota$ being a Lie algebroid morphism, in particular compatible with the anchors, and the anchors being surjective so that $\delta_L(\Der L)=\Tan M$ and $\delta_{L_S}(\Der L_S)=\Tan S$.
\end{proof}

\begin{prop}[Derivations of a Submanifold]\label{DerivationsLineSubmanifold}
Let $\lambda:L\to M$ be a line bundle, $i:S\hookrightarrow M$ an embedded submanifold of the base and denote by $\Gamma_S\subset\Sec{L}$ its vanishing submodule. We define the derivations that tangentially restrict to $S$ as
\begin{equation*}
\normalfont
    \text{Der}_S(L):=\{D\in\Dr{L}|\quad D[\Gamma_S]\subset \Gamma_S\}
\end{equation*}
and the derivations that vanish on $S$ as
\begin{equation*}
\normalfont
    \text{Der}_{S_0}(L):=\{D\in\Dr{L}|\quad D[\Sec{L}]\subset \Gamma_S\},
\end{equation*}
then there is a natural isomorphism of $\normalfont\Cin{M}$-modules and Lie algebras
\begin{equation*}
\normalfont
    \Dr{L_S}\cong \text{Der}_S(L)/\text{Der}_{S_0}(L). 
\end{equation*}
\end{prop}
\begin{proof}
The isomorphism as modules follows directly from proposition \ref{DerBundleLineSubmanifold} using the correspondence between sections of the der bundle and derivations. The Lie algebra isomorphism is then a consequence of the following simple facts
\begin{align*}
    [\text{Der}_S(L),\text{Der}_S(L)] & \subset  \text{Der}_S(L)\\
    [\text{Der}_S(L),\text{Der}_{S_0}(L)] & \subset  \text{Der}_{0S}(L)\\
    [\text{Der}_{S_0}(L),\text{Der}_{S_0}(L)] & \subset  \text{Der}_{S_0}(L),
\end{align*}
easily derived from the definitions above, thus showing that $\text{Der}_S(L)\subset \Dr{L}$ is a Lie subalgebra and $\text{Der}_{S_0}(L)\subset\text{Der}_S(L)$ is a Lie ideal making the subquotient $\text{Der}_S(L)/\text{Der}_{S_0}(L)$ into a Lie algebra reduction.
\end{proof}

\begin{prop}[Der Bundle of a Group Action Quotient]\label{DerBundleGroupAction}
Let $\Phi:G\times L\to L$ be a free and proper line bundle $G$-action with infinitesimal counter part $\normalfont \Psi:\mathfrak{g}\to \Dr{L}$ and let us denote by $\zeta: L\to L/G$ the submersion factor given by taking the quotient onto the space of orbit line bundles. Defining the $G$-invariant derivations as
\begin{equation*}
\normalfont
    \Dr{L}^G:=\{D\in\Dr{L}|\quad (\Phi_g)_*D=D \quad \forall g\in G\}=\{D\in\Dr{L}|\quad D[\Sec{L}^G]\subset \Sec{L}^G\}
\end{equation*}
and the derivations that tangentially restrict to the orbits as
\begin{equation*}
\normalfont
     \Dr{L}^{G_0}:=\{D\in\Dr{L}|\quad D[\Sec{L}^G]=0\}
\end{equation*}
we find the following isomorphism of modules and Lie algebras
\begin{equation*}
\normalfont
    \Dr{L/G}\cong \Dr{L}^G/\Dr{L}^{G_0}.
\end{equation*}
When $G$ is connected, this isomorphism becomes
\begin{equation*}
\normalfont
    \Dr{L/G}\cong \Dr{L}^{\mathfrak{g}}/\Psi(\mathfrak{g}).
\end{equation*}
where the $\mathfrak{g}$-invariant derivations are defined as
\begin{equation*}
\normalfont
    \Dr{L}^{\mathfrak{g}}:=\{D\in\Dr{L}|\quad [\Psi(\xi),D]=0 \quad \forall \xi\in \mathfrak{g}\}.
\end{equation*}
When applied point-wise, this last isomorphism gives a fibre-wise isomorphism of der spaces over the base orbit space $M/G$:
\begin{equation*}
\normalfont
    \Der_{[x]}(L/G)\cong \Der_xL/\mathfrak{g}.
\end{equation*}
\end{prop}
\begin{proof}
The module isomorphism $\Dr{L/G}\cong \Dr{L}^G/\Dr{L}^{G_0}$ is a direct consequence of the ring isomorphism $\Cin{M/G}\cong \Cin{M}^G$ and the module isomorphism $\Sec{L/G}\cong\Sec{L}^G$. The isomorphism as Lie algebras follows from the fact that, by construction, we have
\begin{align*}
    [\Dr{L}^G,\Dr{L}^G] & \subset  \Dr{L}^G\\
    [\Dr{L}^G, \Dr{L}^{G_0}] & \subset   \Dr{L}^{G_0}\\
    [ \Dr{L}^{G_0}, \Dr{L}^{G_0}] & \subset   \Dr{L}^{G_0},
\end{align*}
thus showing that $\Dr{L}^G\subset \Dr{L}$ is a Lie subalgebra and $\Dr{L}^{0G}\subset\Dr{L}^G$ is a Lie ideal making the subquotient $\Dr{L}^G/\Dr{L}^{G_0}$ into a Lie algebra reduction. When $G$ is connected, its action is uniquely specified by the infinitesimal counterpart, which can be equivalently regarded as a Lie algebroid morphism $\Psi:(\mathfrak{g}_M,\psi,[,])\to (\Der L,\delta,[,])$, where $\mathfrak{g}_M=M\times \mathfrak{g}$ here denotes the action Lie algebroid with anchor given by the infinitesimal action on the base manifold $\psi:\mathfrak{g}\to \Sec{\Tan M}$. It is clear by construction that for the action of a connected $G$ we have $\Dr{L}^{0G}=\Psi(\mathfrak{g})$ and that $G$-invariance under push-forward, $(\Phi_g)_*D=D$, becomes vanishing commutator with the infinitesimal action, $[\Psi(\xi),D]=0$. Then the second isomorphism follows. Since the action is free and proper, the map $\Psi$ will be injective as a vector bundle morphism covering the identity map on $M$, then applying the second isomorphism point-wise and injectivity of the infinitesimal action, so that $\Psi(\mathfrak{g})_x\cong \mathfrak{g}$, we find the last desired result.
\end{proof}

\begin{prop}[Der Bundle of the Line Product]\label{DerLineProduct}
Let two line bundles $\normalfont L_1, L_2\in \Line_\Man$ and take their line product $L_1\utimes L_2$, then there is a canonical isomorphism of lvector bundles covering the identity map on the base product:
\begin{equation*}
\normalfont
    \Der (L_1\utimes L_2)\cong \Der L_1 \boxplus \Der L_2.
\end{equation*}
\end{prop}
\begin{proof}
Note first that both vector bundles are, by construction, vector bundles over the same base manifold $M_1\dtimes M_2$. The first term $p_1^*\Der L_1$ is clearly an lvector bundle with line $p_1^*L_1$ and although the second term is an lvector bundle with line $p_2^*L_2$ we can use the canonical factor $C_{12}$ to isomorphically regard $p_2^*\Der L_2$ as an lvector bundle with line $p_1^*L_1$. Then the direct sum as lvector bundles is well-defined and giving an lvector bundle isomorphism now reduces to finding a vector bundle isomorphism. A quick check introducing trivializations (see the end of this section for details) shows that both vector bundles have the same rank and so it is enough to find a fibre-wise surjective map $\Phi$ between the vector bundles. We can write this map explicitly as
\begin{align*}
\Phi_{B_{x_1x_2}}:\Der_{B_{x_1x_2}}(L_1\utimes L_2) & \to \Der_{x_1}L_1\oplus \Der_{x_2}L_2\\
a & \mapsto \Der_{B_{x_1x_2}}P_1(a) \oplus \Der_{B_{x_1x_2}}P_2(a)
\end{align*}
where $P_1$ and $P_2$ are the line product projection factors as in (\ref{LineProductCommutativeDiagram}). Surjectivity of this map follows directly from the definition of the der map and the fact that projection factors are surjective. The line bundle component of the isomorphism of lvector bundles is given simply by the fact that the lvector bundle product $\boxplus$ is defined as a lvector bundle with line bundle component given by the line product $\utimes$.
\end{proof}
\begin{prop}[Derivations of the Line Product]\label{DerivationsLineProduct}
Let two line bundles $\normalfont L_1, L_2\in \Line_\Man$ and take their line product $L_1\utimes L_2$, then we find the derivations of each factor as submodules of the derivations of the product
\begin{equation*}
\normalfont
\begin{tikzcd}
\Dr{L_1}\arrow[r,hook, "k_1"] & \Dr{L_1\utimes L_2} & \Dr{L_2}\arrow[l,hook',"k_2"'].
\end{tikzcd}
\end{equation*}
The maps $k_i$ are Lie algebra morphisms making $\normalfont \Dr{L_i}\subset\Dr{L_1\utimes L_2}$ into Lie subalgebras which, furthermore, satisfy
\begin{equation*}
\normalfont
    [\Dr{L_1},\Dr{L_1}]\subset \Dr{L_1} \qquad [\Dr{L_1},\Dr{L_2}]=0 \qquad [\Dr{L_2},\Dr{L_2}]\subset\Dr{L_2}.
\end{equation*}
\end{prop}
\begin{proof}
This can be proved using the isomorphism in proposition \ref{DerLineProduct} above and the fact that sections of the der bundle are identified with derivations. However, we give an independent, more algebraic proof that explicitly involves the module structure of the sections of the line product. Recall that sections of the line product $\Sec{L_1\utimes L_2}$ are spanned by the pull-backs $P^*_i\Sec{L_i}$ over the functions on the base product $\Cin{M_1\dtimes M_2}$. This means that a derivation $D$ is characterised by its action on projection pull-backs and the action of its symbol $X$ on spanning functions of the base product, which are the ratio functions defined in proposition \ref{SpanningFunctionsLineProduct}. With this in mind, given derivations $D_i\in\Dr{L_i}$ we give the derivations on the line product $k_i(D_i)=\overline{D}_i\in\Dr{L_1\utimes L_2}$ determined uniquely by the conditions
\begin{align*}
    \overline{D}_1(P^*_1s_1)&=P^*D_1(s_1) &\overline{D}_1(P^*_2s_2)&=0\\
    \overline{D}_2(P^*_1s_1)&=0 &\overline{D}_2(P^*_2s_2)&=P^*_2D_2(s_2)
\end{align*}
and with symbols $\overline{X}_i\in\Sec{T(M_1\dtimes M_2)}$ defined on ratio functions by
\begin{align*}
    \overline{X}_1[\tfrac{s_1}{u_2}]=\tfrac{D_1(s_1)}{u_2}\\
    \overline{X}_2[\tfrac{s_2}{u_1}]=\tfrac{D_2(s_2)}{u_1}
\end{align*}
for all sections $s_i\in\Sec{L_i}$ and local non-vanishing sections $u_i\in\Sec{L_i^\bullet}$. The fact that $k_i$ are Lie algebra morphisms and that $[\overline{D}_1,\overline{D}_2]=0$ follow directly from the defining conditions above.
\end{proof}

In summary, we find the \textbf{der functor} for line bundles
\begin{equation*}
    \Der : \Line_\Man \to \LVect_\Man,
\end{equation*}
which, in light of the last few propositions above, is shown to play a categorical role entirely analogous to that of the tangent functor for smooth manifolds:
\begin{equation*}
    \Tan : \Man \to \Vect_\Man.
\end{equation*}

Similar to the idea of differential of functions giving local linear approximations in ordinary manifolds, a unit-free counterpart of this idea is encapsulated in the standard definition of \textbf{jets} of sections: let a line bundle $\lambda: L\to M$, a point $x\in M$ and a neighbourhood $U\subset M$, the jet of a local section $s\in \Sec{L|_U}$ is the equivalence class of local sections:
\begin{equation*}
    j_x^1s:=\{r\in \Sec{L|_U}\quad | \quad s(x)=r(x), \qquad \Tan_xs=\Tan_xr\}.
\end{equation*}
Following standard constructions, it is clear that the \textbf{jet bundle}
\begin{equation*}
    \Jet^1 L:=\bigsqcup_{x\in M} \Jet^1_x L, \qquad \Jet^1_x L:= \{j_x^1 s, \quad s\in \Sec{L}\}.
\end{equation*}
corresponds to the ldual of the der bundle
\begin{equation*}
    \Jet^1 L:=(\Der L)^*\otimes L = (\Der L)^{*L}.
\end{equation*}
By construction, the assignment of the jet to a section gives a differential operator
\begin{equation*}
    j^1:\Sec{L}\to \Sec{\Jet^1 L}
\end{equation*}
called the \textbf{jet map}. The action of a section of the der bundle $a\in\Sec{\Der L}$ on a section $s\in\Sec{L}$ as a local derivation can now be rewritten as
\begin{equation*}
    a[s]=j^1s(a)
\end{equation*}
thus we see that the jet map $j^1$ gives the unit-free generalization of the exterior derivative in ordinary manifolds.\newline

The symbol and Spencer short exact sequences, defined in general for differential operators between vector bundles (see e.g. \cite[Ch. 11]{nestruev2006smooth}), find a straightforward reformulation in the category of lvector bundles. Note that, for any line bundle $\lambda:L\to M$, the anchor map $\delta:\Der L\to \Tan M$ is surjective by construction and elements of its kernel correspond to bundle endomorphisms induced from fibre-wise scalar multiplication, $\Ker{\delta}\cong \Real_M$. Then, regarding $\Der L$ and $\Tan M$ as lvector bundles and $\delta$ as a lvector bundle morphism in a trivial way, the symbol sequence is simply the short exact sequence of lvector bundles induced by the fact that $\delta$ is surjective
\begin{equation*}
\begin{tikzcd}
0 \arrow[r] & \Real_M \arrow[r] & \Der L \arrow[r, "\delta"] & \Tan M \arrow[r] & 0.
\end{tikzcd}
\end{equation*}
This will be sometimes called the \textbf{der sequence} of the line bundle $L$. The Spencer sequence then corresponds precisely to the ldual of the short exact sequence above
\begin{equation*}
\begin{tikzcd}
0 &  L \arrow[l] & \Jet^1L \arrow[l]& (\Tan M)^{*L} \arrow[l,"i"'] & 0 \arrow[l]
\end{tikzcd}
\end{equation*}
where $i=\delta^{*\Id_L}$ is injective. We will refer to this sequence as the \textbf{jet sequence} of the line bundle $L$. We will use the notation $\Tan^{*L} M:= (\Tan M)^{*L}$ in analogy with the usual notation for cotangent bundles. This sequence allows to write the Leibniz characterisation of the jet map $j^1$ as a differential operator:
\begin{equation*}
    j^1(f\cdot s)=f\cdot j^1 + i(df \otimes s).
\end{equation*}

\subsection{Trivial Line Bundles} \label{TrivialLineBundles}

In this section we briefly summarise the results presented so far for the case of globally trivialised line bundles. Note that the results of this section will apply to general line bundles locally, when restricted to trivializing neighbourhoods. In anticipation to our discussion of unit-free phase spaces in Section \ref{UnitFreeGeneralisation}, a choice of a local non-vanishing section in a line bundle $L$ is called a \textbf{unit} $u\in\Sec{L^\bullet}$ defined (possibly only) over an open subset of the base manifold $U\subset M$. The choice a unit $u$ is, of course, equivalent to a local trivialization $L|_U\cong \Real_U$, which in turn induces an isomorphism of modules of sections $\Cin{U}\cong \Sec{L|_U}$ given by the fact that for any local section $s\in \Sec{L|_U}$ there is a unique function $f_s\in\Cin{M}$ such that $s=f_s\cdot u$. Two units are related by a non-vanishing local function $z\in\Cin{U\cap U'}$ via $u'=z\cdot u$, we call such functions \textbf{conversion factors}.\newline

Sections of a trivial line bundle $\Real_M$ are isomorphic to the smooth functions of the base $\Sec{\Real_M}\cong \Cin{M}$ with the module structure being simply point-wise multiplication. This implies that there is now a natural inclusion $\Dr{\Cin{M}}\subset\Dr{\Real_M}$ making the der short exact sequence split and thus giving an isomorphism of modules
\begin{equation*}
    \Dr{\Real_M} \cong \Sec{\Tan M}\oplus \Cin{M}.
\end{equation*}
The action of a derivation on a section $s\in \Sec{\Real_M}$, regarded as a function $s\in \Cin{M}$, is given by
\begin{equation*}
    (X\oplus f)[s]=X[s]+fs.
\end{equation*}
It then trivially follows that the der bundle is
\begin{equation*}
    \Der \Real_M \cong \Tan M \oplus \Real_M
\end{equation*}
with anchor $\delta=\Proj_1$ and Lie bracket bracket
\begin{equation*}
    [X\oplus f,Y\oplus g]=[X,Y]\oplus X[g]-Y[f].
\end{equation*}
Note how this Lie bracket is entirely determined by the fact that vector fields are the Lie algebra of derivations on functions. Taking $\Real_M$-duals corresponds to taking ordinary duals, therefore the jet bundle is
\begin{equation*}
    \Jet^1 \Real_M\cong \Cot M\oplus \Real_M.
\end{equation*}
The jet prolongation map then becomes
\begin{align*}
j^1: \Cin{M} & \to \Sec{\Cot M}\oplus \Cin{M}\\
s & \mapsto ds\oplus s,
\end{align*}
which indeed only carries the information of the ordinary exterior differential. The base product of two trivial line bundles $\Real_M$ and $\Real_N$ is
\begin{equation*}
    M\dtimes N\cong M\times N\times \Real^\times
\end{equation*}
and the line product is again a trivial line bundle
\begin{equation*}
    \Real_M\utimes \Real_N \cong \Real_{M\dtimes N}.
\end{equation*}
A factor between trivial line bundles $B:\Real_M\to \Real_N$ is given by a pair $B=(\varphi,\beta)$ with $\varphi:M\to N$ a smooth map and $\beta\in\Cin{M}$ a nowhere-vanishing function. We have explicitly
\begin{align*}
(\varphi,\beta): \Real_M & \to \Real_N\\
(x,l) & \mapsto (\varphi(x),\beta(x)l).
\end{align*}
Pull-backs then become
\begin{equation*}
    (\varphi,\beta)^*s=\tfrac{1}{\beta}\cdot \varphi^*s
\end{equation*}
for all $s\in\Cin{N}$. A simple computation shows that the der map of a factor $(\varphi,\beta)$ gives a map of the form
\begin{align*}
\Der (\varphi,\beta): \Tan M\oplus \Real_M & \to \Tan N\oplus \Real_N\\
v_x\oplus a_x & \mapsto \Tan_x \varphi(v_x)\oplus a_x-\tfrac{1}{\beta(x)}d_x\beta(v_x).
\end{align*}
For a diffeomorphic factor $(\varphi,\beta):\Real_M\to \Real_M$, i.e. when $\varphi$ is a diffeomorphism, the der push-forward of a derivation is given by
\begin{equation*}
    (\varphi,\beta)_*(X\oplus f)=\varphi_* X \oplus \varphi_*f +\beta\cdot X[\tfrac{1}{\beta}].
\end{equation*}

\subsection{Jacobi Algebroids and Cartan Calculus} \label{JacobiAlgebroids}

In the context of ordinary smooth manifolds, one can motivate the introduction of Lie algebroids simply from the canonical structure found on tangent bundles. Let $M$ be a smooth manifold, then $\Tan M$ is a vector bundle whose sections carry a natural Lie algebra structure interacting with the module product via the Leibniz rule: the ordinary Lie bracket of vector fields. This can be trivially regarded as a vector bundle morphism given by the identity $\Id_{\Tan M}:\Tan M \to \Tan M$ which induces a Lie algebra morphism between the modules of sections. Lie algebroids appear naturally when one replaces the identity map into $\Tan M$ by an arbitrary bundle map covering the identity on $M$ called the anchor $\rho: A\to \Tan M$, and the sections are assumed to carry a Lie algebra structure $(\Sec{A},[,])$ which interacts with the module structure via the Leibniz formula enabled by the anchor:
\begin{equation*}
    [a,f\cdot b]=f\cdot [a,b] + \rho_*(a)[f]\cdot b.
\end{equation*}
This condition can be shown to imply that the anchor $\rho$ induces a morphism of Lie algebras
\begin{equation*}
    \rho_*:\Sec{A}\to \Sec{\Tan M}.
\end{equation*}

Our discussion in Section \ref{DerJetBundles} showed that the der bundles of line bundles are the unit-free analogues of tangent bundles and, in particular, Proposition \ref{DerFunctor} proved that der bundles carry canonical Lie algebroid structures in much the same way. At this point we could mirror the discussion above but now in the category of lvector bundles and attempt to define a \textbf{Lie lalgebroid} as generalisation of the der bundle: let $\lambda: L\to M$ be a line bundle and $(\Der L, \delta, [,])$ its der bundle, a Lie lalgebroid could then be defined as a lvector bundle $A^L$, together with an anchor lvector bundle map covering the identity $\rho^{\Id_L}:A^L\to \Der L$ and a Lie algebra structure on sections $(\Sec{A},[,])$ that interacts with the module structure via the Leibniz formula enabled by the anchor $\rho$ and the symbol $\delta: \Der L \to \Tan M$:
\begin{equation*}
    [a,f\cdot b]=f\cdot [a,b] + \delta(\rho_*(a))[f]\cdot b.
\end{equation*}
Mirroring the simple computation which proves that the anchor is a Lie algebra morphism in the Lie algebroid case, we now obtain the condition that $\rho$ is a Lie algebra morphism up to the kernel of the symbol $\delta$:
\begin{equation*}
    \delta(\rho_*([a,b]))=\delta([\rho_*(a),\rho_*(b)]).
\end{equation*}
We will comment on this issue again in Section \ref{Conclusion}, since it is a clear example of one of the limitations of unit-free geometry. Motivated both by analogy with the Lie algebroid case and by the many examples of similar structures found in the Jacobi geometry literature, we define \textbf{Jacobi algebroid} following the conventional definition in terms of Lie algebroid representations, which in our case means a Lie lalgebroid $(A^L,\rho,[,])$ with an anchor that, furthermore, induces a Lie algebra morphism on sections
\begin{equation*}
    \rho_*:\Sec{A}\to \Sec{\Der L}.
\end{equation*}

Jacobi algebroids are thus regarded as the natural direct generalisation of Lie algebroids in the category of lvector bundles and, as such, we should expect to find a discussion about morphisms and associated Cartan calculus that is entirely analogous to the ordinary case.\newline

The lmultisections of a Jacobi algebroid $(A^L,\rho,[,])$ carry a natural Gerstenhaber algebra $(\Sec{\wedge^\bullet A},\wedge, \llbracket , \rrbracket)$ induced by the ordinary Lie algebroid structure present on $A$ with anchor $\delta \circ \rho_*$. We recall that this structure is simply the graded Lie bracket induced from:
\begin{align*}
    \llbracket a,b \rrbracket &= [a,b]\\
    \llbracket a,f \rrbracket &=\delta(\rho_*(a))[f] \\
    \llbracket f,g \rrbracket &= 0
\end{align*}
for $a,b\in \Sec{\wedge^1 A^L}=\Sec{A}$, $f,g\in \Sec{\wedge^0 A^L}=\Cin{M}$ and extending via lmodule derivations. In the case when $A^L=\Der L$, the lmultisections together with the Gerstenhaber algebra structure $\text{Der}^\bullet (L):=(\Sec{\wedge^\bullet \Der L}, \wedge, \llbracket , \rrbracket)$ are called the \textbf{multiderivations} of the line bundle $L$.\newline

Dually to the Lie algebra structure of a Jacobi algebroid $(A^L,\rho,[,])$ we find the line bundle version of the Cartan calculus on the lmodule of lforms $\Omega^\bullet(A^L):=(\Sec{\wedge^\bullet A^* \otimes L},\wedge, d_{A^L})$, this complex is appropriately called \textbf{the de Rham complex of the Jacobi algebroid} $A^L$. Let us spell out the details of this construction. The \textbf{differential} $d_{A^L}$ is induced by the anchor map in degree $0$:
\begin{align*}
    d_L^0:\Sec{L} &\to \Sec{A^{*L}}\\
    s &\mapsto d_{A^L}s \quad | \quad d_{A^L}s(a)=\rho_*(a)[s]=j^1s(\rho_*(a)),
\end{align*}
extending via lmodule derivations from the following fundamental relation:
\begin{equation*}
    j^1(f\wedge s):=j^1(f\cdot s)=f\cdot j^1s + i(df\otimes s) =: f\wedge j^1s + df \wedge s,
\end{equation*}
and by capturing in degree $1$ the Lie bracket via the usual formula of the differential of a $1$-form:
\begin{equation*}
    d_{A^L}^1 \alpha (a,b):=a[\alpha(b)]-b[\alpha(a)]-\alpha([a,b])
\end{equation*}
for $\alpha\in \Sec{\wedge^1 A^{*L}}=\Sec{A^{*L}}$ and $a,b\in \Sec{A}$. Let us denote by $\Omega^\bullet(A):=(\Sec{\wedge^\bullet A^*},\wedge, d_A)$ the ordinary de Rham complex of the Lie algebroid structure present in $A$ with anchor $\delta \circ \rho$. The differential $d_{A^L}$ defined above is easily shown to be nilpotent and it behaves as a lmodule derivation of degree $+1$ in the followig sense:
\begin{equation*}
    d_{A^L}(\omega \wedge \alpha)=d_A\omega \wedge \alpha +(-)^k\omega \wedge d_{A^L}\alpha
\end{equation*}
for $\omega\in \Omega^k(A)$ and $\alpha\in\Omega^\bullet(A^L)$. We thus say that the de Rham complex of a Jacobi algebroid $\Omega^\bullet(A^L)$ is a lmodule over the de Rham complex of the Lie algebroid associated with it $\Omega^\bullet(A)$. Further to this lmodule structure, the fundamental relation imposed by the jet map above enables a module structure over the de Rham complex of the base manifold defined by
\begin{equation*}
    \theta \wedge \alpha := \delta^*\theta \wedge \alpha
\end{equation*}
for $\theta\in \Omega^k (M)$ and $\alpha\in \Omega^\bullet(A^L)$. It then follows from the above general expression for $d_{A^L}$ that
\begin{equation*}
    d_{A^L}(\theta\wedge \alpha)=d\theta \wedge \alpha +(-)^k \theta \wedge d_{A^L}\alpha
\end{equation*}
where $d$ denotes the ordinary de Rham differential on $M$. The \textbf{interior product} with a section of the Jacobi algebroid $a\in\Sec{A^L}$ is defined in the obvious way
\begin{equation*}
    i_a\alpha (\cdot , \cdot , \dots, \cdot):=\alpha(a,\cdot , \dots , \cdot)
\end{equation*}
for $a\in \Omega^\bullet(A^L)$ and it is readily checked to behave as a lmodule derivation of degree $-1$. Lastly, the \textbf{Lie derivative} is defined via Cartan's magic formula:
\begin{equation*}
    \LDer_a\alpha := i_ad_{A^L}+d_{A^L}i_a,
\end{equation*}
which is then clearly a lmodule derivation of degree $0$. These operations constitute what we call the \textbf{Cartan calulus of the Jacobi algebroid} $(A^L,\rho, [,])$ since it follows from routine computations that the usual Cartan calculus relations hold
\begin{align*}
    [\LDer_a,\LDer_b] &= \LDer_{[a,b]}\\
    [\LDer_a,i_b] &=i_{[a,b]} \\
    [i_a,i_b] &= 0.
\end{align*}
In the case when $A^L=\Der L$, the de Rham complex of $\Der L$ regarded as a Jacobi algebroid with identity anchor will be called \textbf{the de Rham complex of the line bundle} $L$ and it will be denoted by $\Omega^\bullet(L):=(\Sec{\wedge^\bullet(\Der L)^*\otimes L},\wedge,d_L)$.\newline

The de Rham complex of a Jacobi algebroid $\Omega^\bullet(A^L)$ induces a natural notion of \textbf{Jacobi algebroid cohomology} $\text{H}^\bullet(A^L)$. In the case of a line bundle $\lambda:L\to M$, an important remark should be made since the de Rham complex $\Omega^\bullet(L)$ is always acyclic, thus having trivial cohomology at all degrees. This can be easily shown by the presence of a contracting homotopy induced by the identity map $\Id_L\in\Dr{L}$. What this tells us is that the lmodule of lforms $\Omega^\bullet(L)$ doesn't contain any more cohomological information beyond the Lie algebroid cohomology of $\Der L$ and the ordinary de Rham cohomology of the base manifold $M$.\newline

Morphisms of Jacobi algebroids can now be easily defined as morphisms of differential graded lmodules via the induced maps of lvector bundle morphisms. Let $\lambda: A^L\to M$ and $\mu: B^{L'}\to N$ be two Jacobi algebroids and $F^B:A^L\to B^{L'}$ a lvector bundle morphism, we say that $F^B$ is a \textbf{Jacobi algebroid morphism} when the induced map $F^{*B}$ is compatible with the graded module differentials:
\begin{equation*}
    F^{*B}\circ d_{B^{L'}} = d_{A^L} \circ F^{*B}.
\end{equation*}
The \textbf{category of Jacobi algebroids} is denoted by $\textsf{Jacb}_\Man$. In the case of the der bundle, the der map of a factor between line bundles $B:L\to L'$ recovers the notion of a morphism of Jacobi algebroids since a simple computation shows that:
\begin{equation*}
    B^*\circ d_{L'} = d_L \circ B^*,
\end{equation*}
where the pull-back of lforms has been defined by
\begin{equation*}
    B^*:=(\Der B)^{*B}:\Omega^\bullet(L')\to \Omega^\bullet(L).
\end{equation*}

\subsection{LDirac Geometry} \label{LDiracGeometry}

One of the strengths of the unit-free approach via the identification of the category of lvector bundles is that it enables a clear path to generalising the ordinary notions of derivations, tangent bundles, de Rham complexes and Lie algebroids to the context of line bundles, as shown in the preceding sections. Although we will not be using any aspects of Dirac-Jacobi geometry beyond the occasional mention of Jacobi algebroids in the present work, this section is devoted to succinctly showing how Courant algebroids and Dirac structures also naturally generalise in the category of lvector bundles.\newline

Let us start by discussing the linear case. Recall that the basic example of an ordinary Courant space (see Courant's thesis \cite{courant1990dirac}) is the direct sum of a generic vector space with its dual, in the same spirit, we can consider a lvector space $V^L\in\LVect$ and construct the direct sum
\begin{equation*}
    \mathbb{V}^L:=V^L\oplus_L V^{*L}
\end{equation*}
It is clear that this space carries a linear factor $\Proj_1^{\Id_L}:\mathbb{V}^L\to V^L$ and a non-degenerate symmetric ltensor $\langle,\rangle \in \mathcal{T}_2(\mathbb{V}^L)$ defined by
\begin{equation*}
    \langle v+\alpha ,w+\beta\rangle_{\mathbb{V}}:=\tfrac{1}{2}(\alpha(v)+\beta(w)).
\end{equation*}
The lvector space $\mathbb{V}^L$ is called the \textbf{standard LCourant space}. This motivates the definition of a general \textbf{LCourant space} as a lvector space $C^L$ together with a linear factor to another lvector space $\rho^{B}:C^L\to V^L$ called the \textbf{anchor} and a non-degenerate symmetric ltensor $\langle,\rangle \in \mathcal{T}_2(C^L)$. Note that the tensor $\langle,\rangle$ is a non-degenerate $L$-valued bilinear form, for some fixed line $L$, and, by dimension counting, we can regard this datum coherent family of ordinary bilinear forms indexed by the choices of basis in $L$. Following from this remark, we see that the usual notions of musical isomorphism, orthogonal complements, isotropic and coisotropic subspaces appear naturally in LCourant spaces. A LCourant space $C^L$ is called \textbf{exact} when it sits in the following short exact sequence of lvector spaces
\begin{equation*}
\begin{tikzcd}
0 \arrow[r] & V^{*L} \arrow[r, "i"] & C^L \arrow[r, "\rho^B"] & V^L \arrow[r]  & 0
\end{tikzcd}
\end{equation*}
where $i:=\sharp \circ \rho^{*B}$ and $\sharp:C^{*L}\to C^L$ is the musical isomorphism induced by $\langle,\rangle$.\newline

We say that $D\subset C^L$ is a \textbf{LDirac space} in the LCourant space $(C^L, \rho:C^L\to V^L,\langle,\rangle)$ when it is a maximally isotropic subspace. Once more, results for ordinary Dirac spaces naturally generalize to the LCourant setting. In particular, we note that the 2-forms that naturally appear in ordinary Courant spaces, namely, differences of isotropic splittings and the 2-forms associated to Dirac spaces, are now 2-lforms. A \textbf{LCourant morphism} is a LDirac space covering graphs of linear factors that compose as isotropic relations to form \textbf{the category of LCourant spaces} $\LCrnt$.\newline

As it will become apparent from the discussion below, it turns out that the notion of \textbf{Dorfman algebra}, i.e. a Leibniz algebra $(\mathfrak{a},[,])$ together with an algebra morphism to a Lie algebra $\rho: \mathfrak{a}\to \mathfrak{g}$ (see \cite{mandal2016exact}), is broad enough to capture the algebraic structures that generalise Courant algebroids in the context of line bundles.\newline 

In the same spirit of how Courant algebroids can be motivated by Lie bialgebroids (see \cite{gualtieri2004generalized}), let us introduce the notion of \textbf{Jacobi bialgebroid} as a pair of a lvector bundle and its ldual $(A^L,A^{*L})$ each with a Jacobi algebroid structure, denoted by $(A^L,\rho,[,])$ and $(A^{*L},\rho_*,[,]_*)$, satisfying the compatibility condition
\begin{equation*}
    d_{A^{*L}}[a,b]=[d_{A^{*L}}a,b]+[a,d_{A^{*L}}b]
\end{equation*}
for all $a,b\in\Sec{A^L}$, where $[,]$ denotes the Schouten bracket of multisections $(\Sec{\wedge^\bullet A^L},[,])$ and $d_{A^{*L}}$ is the graded differential of $A^{*L}$ defined on the complex $\Sec{\wedge^\bullet (A^{*L})^{*L}}\cong \Sec{\wedge^\bullet A^L}$. Just like in the case of ordinary Lie bialgebroids, this compatibility condition gives Leibniz representations of $\Sec{A^L}$ and $\Sec{A^{*L}}$ on each other via the Cartan calculi. These, in turn, induce two independent semidirect product Dorfman algebras on $\Sec{A^L}\oplus\Sec{A^{*L}}$ which can be written explicitly as direct sum:
\begin{equation*}
    [a\oplus\alpha ,b\oplus \beta]:= [a,b] + \LDer_\alpha b -i_\beta d_{A^{*L}}a \oplus [\alpha,\beta]_* + \LDer_a\beta -i_bd_{A^L}\alpha.
\end{equation*}
This is then a Dorfman algebra structure on the sections of the lvector bundle $A^L \oplus A^{*L}$ whose fibres are clearly LCourant spaces with the obvious bilinear form induced by the ldual pairing and the anchor map is given by the sums of anchors $\rho + \rho_*:A^L \oplus A^{*L}\to \Der L$.\newline

The construction above motivates the definition of \textbf{LCourant algebroid} as a lvector bundle $\epsilon:E^L\to M$ together with $(E^L,\langle,\rangle,\rho^B: E^L\to \Der L,[,])$ where $\langle,\rangle\in\Sec{\mathcal{T}_2(E^L)}$ is symmetric and non-degenerate ltensor and $\rho: (\Sec{E},[,])\to (\Der L,[,])$ is a Dorfman algebra structure such that
\begin{align*}
    [a,f\cdot b] &=f\cdot [a,b]+\delta\circ\rho_*(a)[f]\cdot b\\
    \rho(a)[\langle b , c \rangle] &=\langle [a,b] , c \rangle + \langle b , [a,c] \rangle\\
    [a,a] &=D\langle a , a \rangle
\end{align*}
for all $a,b\in\Sec{E}$, $f\in\Cin{M}$, where $D:=\sharp\circ \rho^{*B} \circ j^1:\Sec{L}\to \Sec{E^L}$ is defined from the musical isomorphism $\sharp:E^{*L}\to E^L$ induced by $\langle,\rangle$ and the jet map $j^1:\Sec{L}\to \Sec{\Jet^1 L}\cong\Sec{\Der L^{*L}}$. These defining axioms should be regarded as the unit-free analogues of the axioms of ordinary Courant algebroids (see \cite{bursztyn2007reduction}) since they are identical to those after simply interchanging der bundles and tangent bundles and working in the line categories where unit-free functions are represented by sections of line bundles.\newline

A \textbf{LDirac structure} in a LCourant algebroid $(E^L,\langle,\rangle,\rho: E^L\to \Der L,[,])$ is a maximally isotropic subbundle $S\subset E^L$ (possibly supported on a submanifold) with involutive sections $[\Sec{S},\Sec{S}]\subset \Sec{S}$. At this point, we could clearly carry out definitions and constructions for LCourant algebroids and LDirac structures that mirror those given in the standard Dirac geometry literature. We find the obvious definitions of \textbf{LCourant tensor}, \textbf{product LCourant algebroid} with the product $\boxplus$ of lvector bundles, \textbf{morphisms of LCourant algebroids} as LDirac structures supported on lgraphs of factors of line bundles and \textbf{LCourant maps} as graphs of lvector bundle morphisms that are LDirac structures in the product LCourant algebroid. These then allow us to identify \textbf{the category of LCourant algebroids} $\LCrnt_\Man$ in a natural way.\newline

We conclude this section with two results that show that LDirac geometry, as introduced in this section, encompasses both ordinary Dirac geometry as well as Vitagliano's Dirac-Jacobi geometry \cite{vitagliano2018dirac}.

\begin{prop}[Courant Algebroids are LCourant Algebroids]\label{LCourantCourant}
Let $\normalfont (E,\langle,\rangle,\rho: E\to \Tan M,[,])$ be a Courant algebroid, then the Lvector bundle $E^{\Real_M}$ is naturally a LCourant algebroid. Furthermore, if $K\subset E$ is a Dirac structure, $K\subset E^{\Real_M}$ is a LDirac structure.
\end{prop}
\begin{proof}
This follows from the results about trivial line bundles presented in Section \ref{TrivialLineBundles}. In particular, recall that the der bundle of a trivial line bundle is $\Der \Real_M\cong \Tan M\oplus \Real_M$, the jet bundle is $\Jet^1\Real_M\cong \Cot M\oplus \Real_M$ and that the jet prolongation map $j^1:\Sec{\Real_M}\to \Jet^1 \Real_M$ is given by $j^1=d\oplus \Id_{\Real_M}$, where $d$ is the ordinary exterior derivative. The Courant bilinear form $\langle,\rangle$ is trivially promoted to a $\Real_M$-valued bilinear form and the anchor map defines a LCourant anchor map simply by setting $\rho \oplus 0: E\to \Tan M\oplus \Real_M$. Since the symbol of the Lie bracket of derivations of a trivial line bundle becomes $\delta = \Proj_1 : \Tan M\oplus \Real_M \to \Tan M$ and we have $\Sec{\Real_M}\cong\Cin{M}$, it is clear that the three compatibility conditions of the Courant bracket on $E$ are equivalent to the same compatibility conditions of the LCourant bracket on $E^{\Real_M}$. It is also easy to see that these make the defining conditions for a Dirac structure on $E$ equivalent to the defining conditions of a LDirac structure on $E^{\Real_M}$ with respect to the LCourant structure defined above.
\end{proof}

We define \textbf{exact LCourant algebroid} $(E^L,\langle,\rangle,\rho^B, [,])$ by imposing that the anchor induces a short exact sequence of lvector bundle morphisms:
\begin{equation*}
\begin{tikzcd}
0 \arrow[r] & E^{*L} \arrow[r, "i"] & E^L \arrow[r, "\rho^B"] & \Der L \arrow[r]  & 0
\end{tikzcd}
\end{equation*}
where $i:=\sharp \circ \rho^{*B}$. At the level of sections we readily find that $(\Sec{E},[,])$ now becomes an exact Dorfman algebra. Associated to any exact Dorfman algebra there is a characteristic class in Leibniz cohomology (\cite[Th. 4.2]{mandal2016exact}) which in the case of ordinary Courant algebroids corresponds to the well-known Severa class. In the context of LCourant algebroids we have the following result.

\begin{prop}[The Characteristic Class of an Exact LCourant Algebroid] \label{CharClass}
An exact LCourant algebroid structure on the lvector bundle $\lambda:E^L\to M$ uniquely determines a Leibniz cohomology class $\normalfont[H]\in \text{H}^2(\Der L;\Jet^1 L)$.
\end{prop}
\begin{proof}
Under the mild assumption that $M$ is a paracompact manifold, splittings of exact sequences of vector bundles over $M$ always exist, then, by a fibre-wise argument, we can see that isotropic splittings $\nabla: \Der L\to E$ always exist for LCourant algebroids. A choice of isotropic splitting allows us to construct a lvector bundle isomorphism explicitly as:
\begin{align*}
\Phi_\nabla^{\Id_L}: \Der L\oplus \Jet^1 L  &\to E^L\\
a + \alpha  &\mapsto \nabla( a ) + j( \alpha)
\end{align*}
Now, $\Der L$ is canonically a Jacobi algebroid, then assuming the zero Jacobi algebroid structure on $\Jet^1 L$ the isotropic splitting makes $(\Der L,\Jet^1 L)$ into a Jacobi bialgebroid. This is seen explicitly via the Leibniz representations than can be built from the datum of the LCourant algebroid:
\begin{equation*}
[\alpha,a]_R=[i(\alpha),\nabla (a)]=i(-i_ad_L\alpha) \quad\quad [a,\alpha]_L=[\nabla (a),i(\alpha)]= i(\LDer_a\alpha).
\end{equation*}
We also identify the Leibniz cocycle induced by the isotropic splitting:
\begin{equation*}
    \eta_\nabla(a,b):=[\nabla(a),\nabla(b)]-\nabla([a,b])
\end{equation*}
which can be easily shown to correspond to a 3-lform via the identity
\begin{equation*}
    \eta_\nabla (a,b) = i(H(a,b,\cdot)), \qquad H_\nabla(a,b,c):=\langle[\nabla(a),\nabla(b)],\nabla(c)\rangle.
\end{equation*}
The lform $H_\nabla\in\Omega^3(L)$ can be shown to be closed $d_LH_\nabla=0$ by construction, hence it defines a cohomology class $[H_\nabla]$. It follows from a simple computation that for any two isotropic splittings, we have $\nabla - \nabla' = i\circ \beta^\flat$ with $b\in \Omega^2(L)$, then we verify:
\begin{align*}
i_ci_bi_aH_{\nabla'} & = i_ci_bi_aH_\nabla + \LDer_ai_bi_cb-i_b\LDer_ai_c\beta-i_ci_b\LDer_a\beta+i_c\LDer_ai_b\beta+i_ci_bi_ad_L\beta b\\
& = i_ci_bi_a(H_\nabla+db) + [\LDer_a,i_b]i_c\beta + i_c[\LDer_a,i_b]\beta \\
& = i_ci_bi_a(H_\nabla+d\beta),
\end{align*}
which implies that $[H_{\nabla'}]=[H_\nabla]$. This shows that the cohomology class is independent of the choice of isotropic splitting.
\end{proof}

The proposition above can be summarised in saying that isotropic splittings induce isomorphisms of exact LCourant algebroids with the \textbf{standard LCourant algebroid}. This is defined as the lvector bundle $\mathbb{D}L:=\Der L \oplus \Jet^1 L$ with anchor $\Proj_1^{\Id_L}:\mathbb{D}L\to \Der L$, ltensor $\langle,\rangle$ given by the ldual pairing and Dorfman bracket explicitly defined in terms of the Cartan calculus:
\begin{equation*}
    [a\oplus \alpha,b\oplus \beta]_H:=[a,b] \oplus \LDer_a\beta -i_bd_L\alpha + i_ai_bH
\end{equation*}
with $H\in \Omega^3(L)$ and $d_LH=0$. Setting $H=0$ recovers the notion of Courant-Jacobi algebroid used by Vitagliano in \cite{vitagliano2018dirac}, which in turn encompassed a wide range of similar, more special, structures previously known in the literature with various names. A LDirac structure in a standard LCourant with $H=0$ algebroid precisely recovers the definition of Dirac-Jacobi structures due to the aforementioned author.

\section{Jacobi Geometry} \label{JacobiGeometry}

The line bundle and lvector bundle machinery developed in the previous section shall make the presentation of Jacobi geometry entirely analogous to the standard presentation of Poisson geometry. The natural approach is to consider Lie bracket structures on the unit-free functions, i.e. sections of line bundles, that are compatible with the module structure. To motivate the precise definition of a Jacobi manifold we first give some general results on local Lie algebras.

\subsection{Local Lie Algebras} \label{LocalLieAlgebras}

The idea of generalising Poisson brackets via Lie algebra structures on sections of vector bundles, as pioneered by Kirillov \cite{kirillov1976local} in the 1970s, requires a generalisation of the Leibniz property of derivations with respect to the point-wise product of functions. The obvious way to generalise this kind of compatibility property is to consider local operators on the commutative algebras of functions and the modules of sections of vector bundles. These are commonly known in the literature as differential operators (see \cite{nestruev2006smooth}). We can thus define a general \textbf{local Lie algebra} structure on a vector bundle $A$ as a Lie bracket on sections $(\Sec{A},[,])$ that acts as a differential operator in each entry. This is a vast class of objects and it is not yet well understood whether there are well-defined categories of local Lie algebras of finite rank beyond the cases discussed below of Lie algebroids and Jacobi structures. For this reason, we limit ourselves to differential operators of order $\leq 1$, which will be simply referred to as differential operators, for the remainder of this section.\newline

Let two vector bundles $A$ and $B$ over the same base $M$, the space of \textbf{differential operators} is defined as:
\begin{equation*}
    \Diff_1(A,B):=\{\Delta : \Sec{A} \to \Sec{B} |\quad \Real\text{-linear }, c_{f}\circ c_{g}(\Delta)=0 \quad \forall f,g\in\Cin{M}\},
\end{equation*}
where we have simplified the notation for commutators of the module multiplications by using the maps $c_f:\text{Hom}_\Real(\Sec{A},\Sec{B})\to \text{End}_\Real(\Sec{A},\Sec{B})$ defined as
\begin{equation*}
    c_f(\Delta)(a):=[\Delta,f](a)=\Delta(f\cdot a )-f\cdot \Delta(a)
\end{equation*}
for $f\in\Cin{M}$, $a\in\Sec{A}$. The standard construction of the jet bundle $\Jet^1 A$ and the jet map $j^1: \Sec{A}\to \Sec{\Jet^1A}$ allow us to regard differential operators as sections of some vector bundle. This is realised via the correspondence
\begin{equation*}
    \text{Hom}_{\Cin{M}}(\Jet^1 A,B)\ni\Phi \mapsto\Phi \circ j^1\in\Diff_k(A,B).
\end{equation*}
which induces an isomorphism of $\Cin{M}$-modules
\begin{equation*}
    \Diff_1(A,B)\cong \Sec{(\Jet^1A)^*\otimes B},
\end{equation*}
see \cite[Prop. 11.51]{nestruev2006smooth} for details.\newline

According to our definition of differential operator above, a $\Cin{M}$-linear map is a particular case of differential operator. In general, we can define the \textbf{symbol} of a differential operator $\Delta\in \Diff_1(A,B)$ as:
\begin{equation*}
    \sigma_\Delta(f,a):=[\Delta,f](a)=\Delta(f\cdot a)-f\cdot \Delta(a).
\end{equation*}
Since $\Delta$ is a differential operator, $\sigma_\Delta(f,-):\Sec{A}\to \Sec{B}$ is a $\Cin{M}$-linear map and a simple computation shows that $\sigma_\Delta(-,a)$ acts as a $\Cin{M}$-module derivation. This makes the assignment of symbols to differential operators into a map of the form
\begin{align*}
\sigma: \Diff_1 (A,B) & \to \Sec{\Tan M\otimes A^*\otimes B}\\
\Delta & \mapsto \sigma_\Delta
\end{align*}
Symbols, in turn, characterise differential operators via the Leibniz identity since a $\Real$-linear map $\Delta:\Sec{A}\to \Sec{B}$ is a differential operator iff there exist a unique $\xi\in \Sec{\Tan M\otimes A^*\otimes B}$ such that
\begin{equation*}
    \Delta(f\cdot a)=f\cdot \Delta(a) + \xi(df,a).
\end{equation*}

In the case of a single vector bundle, we denote the differential operators of $A$ into itself via $\Diff_1(A):=\Diff_1(A,A)$. The classical example of differential operators are linear partial differential equations on smooth functions, which we recover as $\Diff_1(M)=\Diff_1(\Real_M)$ in the general formalism. Conventional differential operators (of order $\leq 1$) on a manifold $\Diff_1(\Real_M)$ are not closed under composition in general (the composition is a differential operator of order $\leq 2$) but a quick computation shows that, due to the commutativity of the product of functions, they do close under commutator. Since differential operators on a general vector bundle $A$ are the sections of the vector bundle $(\Jet^1 A)^* \otimes A$, one may expect this to be a first canonical example of local Lie algebra after endowing $\Diff_1(A)$ with the commutator bracket. The proposition below shows that this is not the case in general.

\begin{prop}[Differential Operators as Local Lie Algebras] \label{DiffLocal}
Let $A$ be a vector bundle over a manifold $M$. Differential operators do not close under commutator in general
\begin{equation*}
\normalfont
    [\Diff_1(A),\Diff_1(A)] \nsubseteq \Diff_1(A).
\end{equation*}
\end{prop}
\begin{proof}
Let two differential operators $\Delta, \Delta'\in\Diff_1(A)$ and two functions $f,g \in \Cin{M}$. After some manipulations, it follows from the definition of differential operator that
\begin{equation*}
    [[[\Delta, \Delta'],f],g]=[[\Delta, f],[\Delta',g]]+[[\Delta,g],[\Delta',f]],
\end{equation*}
however note that the terms in brackets in the RHS are precisely the commutators of symbols of the differential operators. Since general $\Cin{M}$-linear maps do not commute, those brackets do not vanish in general. This can be seen more explicitly by regarding the symbols of each operator as maps of the form $\delta,\delta':\Cot M\to A^*\otimes A$. Using the Leibniz characterisation of the differential operators $\Delta$ and $\Delta'$ we can write the following Leibniz identity for the commutator of differential operators
\begin{equation*}
    [\Delta, \Delta'](f\cdot a)=f\cdot [\Delta, \Delta'](a)+ ([\delta(df),\Delta']+[\Delta,\delta'(df)])(a).
\end{equation*}
We can identify the second term as a map $\lambda_{[\Delta,\Delta']}:\Sec{\Cot M}\to \text{Hom}_\Real(\Sec{A},\Sec{A})$ which is $\Cin{M}$-linear from the fact that symbols are $\Cin{M}$-linear. If $\lambda_{[\Delta,\Delta']}(df):\Sec{A}\to \Sec{A}$ were checked to be $\Cin{M}$-linear for all $f\in \Cin{M}$, then the Leibniz characterisation of differential operators would tell us that $\lambda_{[\Delta,\Delta']}$ is the unique symbol of $[\Delta,\Delta']$, hence showing that the commutator is a differential operator. A simple computation shows that
\begin{equation*}
    \lambda_{[\Delta,\Delta']}(df)(g\cdot a)=g\cdot \lambda_{[\Delta,\Delta']}(df)(a)+([\delta(df),\delta'(dg)]+[\delta(dg),\delta'(df)])(a).
\end{equation*}
We see that the failure for $\lambda_{[\Delta,\Delta']}(df)$ to be $\Cin{M}$-linear is the second term of the expression above which consists of commutators of generic $\Cin{M}$-linear maps of sections of $A$, which don't vanish in general.
\end{proof}

The proof of Proposition \ref{DiffLocal} above suggests an obvious strategy in the search of subspaces of differential operators on a vector bundle $A$ that will carry a local Lie algebra structure. Such a subspace of differential operators will have to be characterised by the fact that their symbols commute as $\Cin{M}$-linear maps for all arguments so that the map $\lambda_{[\Delta,\Delta']}$ becomes the symbol of the commutator of two elements $\Delta,\Delta'\in\Diff_1(A)$ of the subspace. A canonical choice in this direction is to consider differential operators whose symbols are multiples of the identity map $\Id_{\Sec{A}}$, i.e. $\Real$-linear maps $D:\Sec{A}\to\Sec{A}$ such that
\begin{equation*}
    D(f\cdot a)=f\cdot a+\sigma_D(df,a)=f\cdot a+X_D[f]\cdot a
\end{equation*}
where $X_D\in \Sec{\Tan M}$ is a vector field acting as a $\Cin{M}$-derivation. Note that this is precisely the identity (\ref{ModuleDerivation}) that characterised module derivations of sections of line bundles. The subspace of differential operators in $\Diff_1(A)$ satisfying this property is called the \textbf{derivations} of $A$ and is denoted by $\Dr{A}$. Since the rank of the vector bundles involved plays no role in the discussion about derivations and der bundles of line bundles in Section \ref{DerJetBundles}, similar results can be found in the general case of vector bundles by restricting to fibre-wise invertible vector bundle morphisms. Symbols of derivations correspond to multiplication by functions, hence their commutators trivially vanish and thus der bundles of vector bundles give the first class of canonical examples of local Lie algebras.\newline

The realisation that vector bundle derivations give natural examples of local Lie algebras motivates the definition of a subclass of local Lie algebras. A vector bundle $A$ is said to carry a \textbf{derivative Lie algebra} structure when there is a Lie bracket on sections $(\Sec{A},[,])$ that acts as a derivation in each of its arguments. More explicitly, a derivative Lie algebra on a vector bundle $A$ over the manifold $M$ is a $\Real$-bilinear Lie bracket $(\Sec{A},[,])$ satisfying
\begin{equation*}
    [a,f\cdot b]=f\cdot [a,b]+\lambda_a[f]\cdot b
\end{equation*}
for all $a,b\in\Sec{A}$ and $f\in\Cin{M}$, and where the map $\lambda:\Sec{A}\to \Sec{\Tan M}$ is a differential operator $\lambda\in\Diff_1(A,\Tan M)$. The map $\lambda$ is called the \textbf{symbol} of the derivative Lie algebra and its symbol as a differential operator $\Lambda^\sharp:=\sigma_\lambda\in \Sec{\Tan M \otimes A^*\otimes \Tan M}$ is called the \textbf{squiggle} of the derivative Lie algebra. These appear explicitly when we consider $\Cin{M}$-linear combinations in both entries of the Lie bracket:
\begin{equation*}
    [f\cdot a,g\cdot b]=fg\cdot[a,b]+f\lambda_a[g]\cdot b-g\lambda_b[f]\cdot a +\Lambda^\sharp(df\otimes a)[g]\cdot b
\end{equation*}
which we hence call the \textbf{symbol-squiggle expansion} of the derivative Lie algebra. Note that $\lambda$ and $\Lambda^\sharp$ are maps from and into Lie algebras that are defined via commutators of differential operators so we should expect to find compatibility conditions between them. These are summarised in the next proposition.

\begin{prop}[Symbol-Squiggle Identities] \label{SymbolSquiggle}
Let $A$ be a vector bundle over a manifold $M$ carrying a derivative Lie algebra structure $(\Sec{A},[,])$ with symbol $\lambda$ and squiggle $\Lambda^\sharp$, then the following identities hold for all $a,b,c\in\Sec{A}$ and $f,g,h\in\Cin{M}$:
\begin{enumerate}
\normalfont
    \item $\lambda_{f\cdot a}=f\cdot \lambda_a +\Lambda^\sharp(df\otimes a)$
    \item $\Lambda^\sharp(df\otimes a)[g]\cdot b=-\Lambda^\sharp(dg\otimes b)[f]\cdot a$
    \item $\lambda_{[a,b]}=[\lambda_a,\lambda_b]$
    \item $[\lambda_a,\Lambda^\sharp(df\otimes b)]=\Lambda^\sharp(d\lambda_a[f]\otimes b + df\otimes [a,b])$
    \item $\Lambda^\sharp(df\otimes a)[\Lambda^\sharp(dg\otimes b)[h]]\cdot c + \Lambda^\sharp(dg\otimes b)[\Lambda^\sharp(dh\otimes c)[f]]\cdot a + \Lambda^\sharp(dh\otimes c)[\Lambda^\sharp(df\otimes a)[g]]\cdot b \\= \lambda_b[f]\Lambda^\sharp(dg\otimes a)[h]\cdot c + \lambda_c[g]\Lambda^\sharp(dh\otimes b)[f]\cdot a + \lambda_a[h]\Lambda^\sharp(df\otimes c)[g]\cdot b$
\end{enumerate}
\end{prop}
\begin{proof}
Identity 1. is simply the Leibniz formula for $\lambda$ as a differential operator. 2. follows directly from antisymmetry of the derivative Lie algebra bracket in the symbol-squiggle expansion $[f\cdot a, g\cdot b]$. Identities 3., 4. and 5. are a consequence of the Jacobi identity of the derivative Lie algebra bracket. We can obtain them by direct computation by considering a nested bracket of the form $[[f\cdot a, g\cdot b],h\cdot c]$. By setting $f=g=1$, we can expand using the symbol of the derivative Lie algebra
\begin{equation*}
    [[a,b],h\cdot c]=h\cdot [[a,b],c]+\lambda_{[a,b]}\cdot c.
\end{equation*}
On the other hand, by using the Jacobi identity of the Lie bracket first and then expanding the symbol we obtain, after cancellations:
\begin{equation*}
    [[a,b],h\cdot c]=[[a,h\cdot c],b]+[a,[b,h\cdot c]]=h\cdot ([[a,c],b]+[a,[b,c]])+(\lambda_a[\lambda_b][h] - \lambda_b[\lambda_a][h])\cdot c.
\end{equation*}
Since both expressions must agree for all $a,b,c\in\Sec{A}$ and $h\in\Cin{M}$, using the Jacobi identity of the Lie bracket once more, it follows that $\lambda_{[a,b]}=[\lambda_a,\lambda_b]$, hence proving 3. Using a similar procedure but now setting $f=1$ and using 3. repeatedly gives 4. Finally, a long but routine computation shows that expanding the general nested bracket $[[f\cdot a, g\cdot b],h\cdot c]$ and using 3. and 4. implies 5.
\end{proof}

Similar to how a differential operator may be given by a $\Real$-linear map and its symbol via the Leibniz characterisation, a derivative Lie algebra can be characterised by $\Real$-linear Lie brackets and its symbol and squiggle.

\begin{prop}[Extension by Symbol] \label{ExtensionSymbol}
Let $A$ be a vector bundle over a manifold $M$ and $\Sigma\subset \Sec{A}$ a subspace of spanning sections, $\normalfont \text{span}_{\Cin{M}}(\Sigma)=\Sec{A}$, then the datum of a derivative Lie algebra structure on $A$ can be given by the following data:
\begin{enumerate}
    \item a $\Real$-bilinear Lie bracket $(\Sigma,[,])$,
    \item a $\Real$-linear map $\lambda:\Sec{A}\to \Sec{\Tan M}$ and
    \item an antisymmetric form $\Lambda\in\Sec{\wedge^2(\Tan M\otimes A^*)\otimes A}$
\end{enumerate}
satisfying the compatibility conditions 1.-5. in Proposition \ref{SymbolSquiggle}.
\end{prop}
\begin{proof}
Since general sections of $A$ are given by $\Cin{M}$-linear combinations of elements in $\Sigma$, it suffices to specify a bracket of the form $[f\cdot s, g\cdot r]$ where $s,r\in \Sigma$ and $f,g\in \Cin{M}$. The bilinear form $\Lambda$ defines a squiggle map via
\begin{equation*}
    \Lambda^\sharp(df\otimes a)[g]\cdot b:=\Lambda(df\otimes a, dg \otimes b),
\end{equation*}
which satisfies condition 2. from the antisymmetry of $\Lambda$. By condition 1. $\lambda$ is a differential operator with symbol given precisely by $\Lambda^\sharp$. Then, the bracket $[f\cdot s, g\cdot r]$ can be defined via the symbol-squiggle expansion:
\begin{equation*}
    [f\cdot s, g\cdot r]=fg\cdot[s,r]+f\lambda_s[g]\cdot r-g\lambda_r[f]\cdot s +\Lambda^\sharp(df\otimes s)[g]\cdot r.
\end{equation*}
This bracket is clearly antisymmetric and conditions 3.-5. imply that it satisfies the Jacobi identity, thus showing that $\Sec{A}$ is endowed with a derivative Lie algebra structure.
\end{proof}

When the squiggle of a derivative Lie algebra $A$ vanishes $\Lambda^\sharp=0$, i.e. when the symbol $\lambda:\Sec{A}\to \Sec{\Tan M}$ is $\Cin{M}$-linear, the above results simplify greatly and the extension by symbol conditions reduce to the symbol being a Lie algebra morphism. In this case, the symbol can be considered the push-forward of a vector bundle morphism covering the identity $\rho:A\to \Tan M$ which is then called an \textbf{anchor}. This, of course, recovers the well-known notion of Lie algebroid.\newline

It turns out that derivative Lie algebra structures are quite sensitive to the rank of the underlying vector bundle being $1$ or greater than $1$. The next two propositions illustrate this fact.

\begin{prop}[Local Lie Algebras of Rank $1$] \label{LocalLieRk1}
A local Lie algebra structure $(\Sec{A},[,])$ with $\normalfont\rkk{A}=1$ is necessarily a derivative Lie algebra.
\end{prop}
\begin{proof}
This is a direct consequence of the fact that the endomorphism bundle of a vector bundle $\alpha:A\to M$ with 1-dimensional fibres is trivial
\begin{equation*}
    \text{End}(A)=A^*\otimes A \cong \Real_M,
\end{equation*}
and thus all $\Cin{M}$-linear maps of sections can be regarded as $\Cin{M}$-multiples of the identity map. It is clear then that the symbol of any differential operator $\Delta\in\Diff_1(A)$ will be uniquely determined by a derivation on the ring of functions, i.e. a vector field. In other words,
\begin{equation*}
    \rkk{A}=1 \quad \Rightarrow \quad \Diff_1(A)=\Dr{A}.
\end{equation*}
Then a Lie bracket on $A$ that acts as a differential operator in each entry, necessarily acts as a derivation in each entry.
\end{proof}

\begin{prop}[Derivative Lie Algebras of Rank $> 1$]\label{LocalLieRk2}
The symbol of any derivative Lie algebra $(\Sec{A},[,])$ with $\normalfont\rkk{A}> 1$ is necessarily an anchor.
\end{prop}
\begin{proof}
Restricting to a trivializing neighbourhood $U\subset M$ we can choose local sections $a,b\in \Gamma_U(A)$ that are $\Cin{U}$-linearly independent, this is guaranteed generically in sufficiently small open neighbourhoods since the rank of the vector bundle is 2 or greater: we are always able to choose two independent directions at any given fibre and then extend smoothly. Take two such sections $a,b\in \Gamma_U(A)$, two local functions $f,g\in \Cin{U}$ and consider the symbol-squiggle bracket $[f\cdot a, g\cdot b]$. Applying the Leibniz identity on each side in two different orders we get two expressions that must agree:
\begin{equation*}
    \lambda_{f\cdot a}[g]\cdot b - g\cdot\lambda_b[f]\cdot a + fg\cdot [a,b]=[f\cdot a, g\cdot b]=f\cdot \lambda_{a}[g]\cdot b - \lambda_{g\cdot b}[f]\cdot a + fg\cdot [a,b]
\end{equation*}
Since $a$ and $b$ are assumed to be $\Cin{U}$-independent, each factor accompanying them must vanish independently, thus giving
\begin{equation*}
    \lambda_{f\cdot a}=f\cdot\lambda_{a}
\end{equation*}
for all $a\in \Gamma_U(A)$, $f\in\Cin{U}$. This gives $\Cin{U}$-linearity of the symbol map restricted to the trivializing neighbourhood $\lambda|_U$, however this is clearly a trivialization-independent property as introducing other trivializations will give $\Cin{U}$-linear combinations of the local sections. Hence, $\lambda$ is globally $\Cin{M}$-linear, showing that it is an anchor, as desired.
\end{proof}

These results give us two good reasons to study derivative Lie algebra structures on line bundles: firstly, beyond rank 1 derivative Lie algebras are simply Lie algebroids, whose general theory is well understood, and secondly, in rank 1 there aren't any local Lie algebra structures which are not derivative. These motivate the definition of Jacobi manifolds in Section \ref{JacobiManifolds} below. Furthermore, it will be shown that most of the well-known examples of Jacobi structures found in the literature, such as Poisson or contact manifolds, are, in fact, not Lie algebroids.

\subsection{Jacobi Manifolds} \label{JacobiManifolds}

We encourage the reader to compare our presentation of the topic of Jacobi manifolds in this section with a standard presentation of Poisson manifolds (an excellent reference is \cite{fernandes2014lectures}). As it will be apparent below, our treatment using the categories of line bundles and lvector bundles produces a formulation of Jacobi geometry that is entirely analogous to that of Poisson geometry. This will justify the choice of similar terminology where appropriate. Furthermore, in Section \ref{UnitPoisson} below it will be shown that conventional Poisson manifolds are recovered as a very special case of Jacobi manifolds within a spectrum of generalisations including, for instance, unit-free Poisson and conformal Poisson structures.\newline

A \textbf{Jacobi manifold} or \textbf{Jacobi structure} is defined as a line bundle $\lambda:L\to M$ whose sections carry a local Lie algebra structure $(\Sec{L},\{,\})$. In virtue of proposition \ref{LocalLieRk1}, the locality condition is tantamount to the adjoint map of the Lie bracket being a differential operator of the form:
\begin{equation*}
    \text{ad}_{\{,\}}:\Sec{L}\to \Dr{L}.
\end{equation*}
The fact that the Lie bracket $\{,\}$ is a derivation on each argument allows us to write
\begin{equation*}
    \{s,r\}=\Pi(j^1s,j^1r) \quad \text{ with }\quad  \Pi\in \Sec{\wedge^2(\Der L)}.
\end{equation*}
We see that the bilinear form $\Pi$ is the analogue of the Poisson bivector and we appropriately call it the \textbf{Jacobi biderivation}. A simple check using the Schouten bracket of multiderivations introduced in Section \ref{JacobiAlgebroids} shows that
\begin{equation*}
    \Pi\in \Sec{\wedge^2(\Der L)} \text{ is a Jacobi biderivation}\quad \Leftrightarrow \quad \llbracket \Pi , \Pi \rrbracket=0.
\end{equation*}
Noting that $(\Jet^1L)^*\otimes L=((\Der L)^{*L})^{*L}\cong \Der L$, the Jacobi biderivation induces a musical map
\begin{equation*}
    \Pi^\sharp:\Jet^1 L \to \Der L.
\end{equation*}
This musical map connects the otherwise disconnected der and jet sequences:
\begin{equation*}
\begin{tikzcd}
0 \arrow[r] &  \Real_M \arrow[r] & \Der L \arrow[r, "\delta"]  & \Tan M \arrow[r]  & 0 \\
0 &  L \arrow[l] & \Jet^1L  \arrow[u, "\Pi^\sharp"] \arrow[l]& \Tan^{*L} M \arrow[u,"\Lambda^\sharp"] \arrow[l,"i"] & 0 \arrow[l] 
\end{tikzcd}
\end{equation*}
where we have defined the bundle map $\Lambda^\sharp:=\delta \circ \Pi^\sharp\circ i$ so that the diagram commutes. By construction, we see that this bundle map is indeed a musical map for the bilinear form defined as the fibre-wise pull-back of the Jacobi biderivation via the injective map of the jet sequence: $\Lambda=i^*\Pi\in\Sec{\wedge^2(\Tan^{*L} M)\otimes L}$. We call this form the \textbf{Jacobi lbivector}.\newline

Since a Jacobi structure is a derivative Lie algebra, there is a notion of \textbf{Hamiltonian derivation} of a section $s\in\Sec{L}$ given by the adjoint map as usual
\begin{equation*}
    D_s :=\{s, -\}= \Pi^\sharp (j^1s),
\end{equation*}
but also a notion of \textbf{Hamiltonian vector field} given by the symbol of the Jacobi bracket, which in our conventions we obtain as the symbol of the Hamiltonian derivation
\begin{equation*}
    X_s:= \delta(D_s).
\end{equation*}
The symbol-squiggle identities of the Jacobi bracket (see Proposition \ref{SymbolSquiggle}) establish the basic relationship between the Hamiltonian maps and bilinear forms associated to a Jacobi structure. These are summarised in the following three expressions:
\begin{align*}
    D_{f\cdot s} &=f\cdot D_s+\Pi^\sharp(i(df\otimes s))\\
    D_{s}(f\cdot r) &=f\cdot D_s(r)+X_s[f]\cdot r\\
    X_{f\cdot s} &=f\cdot X_s+\Lambda^\sharp(df\otimes s)
\end{align*}
for all $f\in\Cin{M}$ and $s,r\in\Sec{L}$.\newline

Let $i:S\hookrightarrow M$ be an embedded submanifold of a line bundle $\lambda:L\to M$ and let us denote by $L_S:=i^*L$ the induced line bundle over $S$ and by $\Gamma_S\subset \Sec{L}$ the submodule of vanishing sections. Note that at each point $x\in M$ the jet space $\Jet^1_x L$ carries a $L_x$-valued bilinear form given by the Jacobi biderivation $\Pi_x$ and so the notion of isotropic subspace of a jet space (one where the bilinear form restricts to zero) is well-defined point-wise. We can then define coisotropic submanifolds of Jacobi manifolds in an entirely analogous way to the coisotropic submanifolds of Poisson manifolds. We say that $S$ is \textbf{coisotropic} if the annihilator of its der bundle $(\Der L_S)^{0L}\subset \Jet^1L$ is an isotropic subbundle with respect to the lbiderivation $\Pi$. Equivalently, $S$ is coisotropic if
\begin{equation*}
    \Pi^\sharp_x((\Der_xL_S)^{0L})\subset \Der_xL_S \qquad \forall x\in S.
\end{equation*}
The following proposition gives several equivalent characterizations of coisotropic submanifolds.
\begin{prop}[Coisotropic Submanifolds of a Jacobi Manifold, cf. {\cite[Lemma 2.44]{tortorella2017deformations}}]\label{CoisotropicSubmanifoldsJacobi}
Let $i:S\hookrightarrow M$ be an embedded submanifold of a Jacobi structure $(\Sec{L},\{,\})$, then the following conditions are equivalent
\begin{enumerate}
\normalfont
    \item $S$ is a coisotropic submanifold;
    \item $\Lambda^\sharp_x(\Tan_x S^0\otimes L_x)\subset \Tan_x S \quad \forall x\in S$;
    \item the vanishing sections $\Gamma_S$ form a Lie subalgebra, $\{\Gamma_S,\Gamma_S\}\subset \Gamma_S$;
    \item the Hamiltonian vector fields of vanishing sections are tangent to $S$, $X_s\in\Sec{\Tan S}$ for $s\in\Gamma_S$.
\end{enumerate}
\end{prop}
\begin{proof}
Equivalence between 1. and 2. follows from the definition of the bilinear forms, $\Lambda=i^*\Pi$, and the fact that
\begin{equation*}
    i(\alpha\otimes l)\in (\Der_xL_S)^{0L} \quad \Leftrightarrow \quad i(\alpha\otimes l)(a)=0 \quad \forall a\in \Der_xL_S,
\end{equation*}
which, since $i$ is the $L$-dual to the anchor $\delta$, is tantamount to demanding
\begin{equation*}
    \alpha(\delta(a))=0 \quad \forall a\in \Der_xL_S.
\end{equation*}
Recall that $\delta(\Der L_S)=\Tan S$, then 
\begin{equation*}
    i(\alpha\otimes l)\in (\Der_xL_S)^{0L} \quad \Leftrightarrow \quad \alpha\otimes l \in \Tan_x S^0\otimes L_x.
\end{equation*}
To show equivalence between 2. and 3. we can use a local trivialization to write $\Ker{i^*}=\Gamma_S=I_S\cdot \Sec{L}$, then it follows from the definition of the Jacobi lbivector that
\begin{equation*}
    i^*\{f\cdot s,g\cdot s\}=\Lambda^\sharp(df\otimes s)[g]\cdot i^*s=0
\end{equation*}
for $f,g\in I_S$ and any section $s\in\Sec{L}$. Finally, equivalence between 3. and 4. follows by considering $f\in I_S$, $s\in\Sec{L}$, $s'\in\Gamma_S$ and writing
\begin{equation*}
    i^*\{s',f\cdot s\}=i^*f\cdot i^*\{s',s\} + i^*(X_{s'}[f])\cdot i^*s= i^*(X_{s'}[f])\cdot i^*s
\end{equation*}
thus implying $X_{s'}[f]\in I_S$, completing the equivalence.
\end{proof}

Let $(\Sec{L_1},\{,\}_1)$ and $(\Sec{L_2},\{,\}_2)$ be two Jacobi structures, a \textbf{Jacobi map} is a factor $B:L_1\to L_2$ such that its pull-back on sections
\begin{equation*}
    B^*:(\Sec{L_2},\{,\}_2)\to (\Sec{L_1},\{,\}_1)
\end{equation*}
is a Lie algebra morphism. The next proposition shows that the categorical product in $\Line_\Man$ allows us to define a product of Jacobi manifolds.
\begin{prop}[Product of Jacobi Manifolds, cf. {\cite[Proposition 5.1]{ibanez1997coisotropic}}]\label{ProductJacobi}
Let Jacobi structures $(\Sec{L_1},\{,\}_1)$ and $(\Sec{L_2},\{,\}_2)$, then there exists a unique Jacobi structure in the line product $(\Sec{L_1\utimes L_2},\{,\}_{12})$ such that the canonical projection factors
\begin{equation*}
\begin{tikzcd}
L_1 & L_1\utimes L_2 \arrow[l,"P_1"']\arrow[r,"P_2"] & L_2.
\end{tikzcd}
\end{equation*}
are Jacobi maps. The Jacobi structure $(\Sec{L_1\utimes L_2},\{,\}_{12})$ is called the \textbf{product Jacobi structure}.
\end{prop}
\begin{proof}
Recall from the definition of line product that $\Sec{L_1\utimes L_2}=\Cin{M_1\dtimes M_2} \cdot p_1^*\Sec{L_1}$, therefore pull-backs of sections span all sections of the line product. Our strategy is to use an extension by symbol argument to construct a total Jacobi bracket on $L_1\utimes L_2$ by fixing an $\Real$-linear Lie bracket on spanning sections and finding appropriate symbol $X^{12}:\Sec{L_1\utimes L_2}\to \Sec{\Tan M_1\dtimes M_2}$ and squiggle $\Lambda^{12}:\Sec{\Tan ^{*L} (M_1\dtimes M_2)\otimes L_1\utimes L_2}\to \Sec{\Tan (M_1\dtimes M_2)}$ to apply Proposition \ref{ExtensionSymbol}. The bracket on pull-backs is defined by:
\begin{equation*}
    \{P_1^*s_1,P_1^*s_1'\}_{12}:=P_1^*\{s_1,s_1'\}_1 \qquad \{P_2^*s_2,P_2^*s_2'\}_{12}:=P_2^*\{s_2,s_2'\}_2\qquad \{P_1^*s_1,P_2^*s_2\}_{12}:=0
\end{equation*}
for all $s_i,s_i'\in\Sec{L_i}$, $i=1,2$. Proposition \ref{SpanningFunctionsLineProduct} ensures that it suffices to define the action of the symbol and the squiggle locally on spanning functions of the base product, thus we set
\begin{align*}
    X^{12}_{P_1^*s_1}[\tfrac{a}{b}]&=\tfrac{\{s_1,a\}_1}{b}   & X^{12}_{P_2^*s_2}[\tfrac{b}{a}]&=\tfrac{\{s_2,b\}_2}{a}\\
    \Lambda^{12}(d\tfrac{a}{b}\otimes P_1^*s_1)[\tfrac{a'}{b'}]&=\tfrac{\{a,a'\}_1}{b}\tfrac{s_1}{b'} &\Lambda^{12}(d\tfrac{b}{a}\otimes P_2^*s_2)[\tfrac{b'}{a'}]&=\tfrac{\{b,b'\}_2}{a}\tfrac{s_2}{a'}
\end{align*}
for all $s_i\in\Sec{L_i}$, $a,a'\in\Sec{L_1^\bullet}$ and $b,b'\in\Sec{L_2^\bullet}$. Checking that $X^{12}$ and $\Lambda^{12}$ so defined satisfy the symbol-squiggle compatibility conditions of Proposition \ref{SymbolSquiggle} follows by a long direct computation that becomes routine once formulas for the action of $X^{12}$ and $\Lambda^{12}$ on all the possible combinations of spanning functions are derived from the symbol-squiggle identity. Appendix \ref{SymbolSquiggleProductJacobi} contains all these formulas. We also point out that, although the above formulas may appear asymmetric for the local non-vanishing sections $a\in\Sec{L_1^\bullet}$ and $b\in\Sec{L_2^\bullet}$, the definition of local ratio functions is such that $\tfrac{a}{b}\tfrac{b}{a}=1$ and so, by construction, we can relate both sides of the definition by the identities
\begin{equation*}
    d\tfrac{a}{b}=-(\tfrac{a}{b})^2d\tfrac{b}{a} \qquad P_1^*a=\tfrac{a}{b}P_2^*b.
\end{equation*}
\end{proof}

Given a Jacobi structure $(\Sec{L},\{,\})$ we define the \textbf{opposite Jacobi structure} simply by taking the negative Lie bracket on sections $\overline{L}:=(\Sec{L},-\{,\})$. The proposition below shows that we can regard Jacobi maps as a particular case of \textbf{coistropic relations}, which are defined, in general, as coisotropic submanifolds of the product Jacobi manifold $L_1\utimes \overline{L}_2$.
\begin{prop}[Jacobi Maps as Coisotropic Relations, cf. {\cite[Theorem 5.3]{ibanez1997coisotropic}}]\label{JacobiMapCoisotropicRelation}
Let two Jacobi structures $(\Sec{L_{M_1}},\{,\}_1)$ and $(\Sec{L_{M_2}},\{,\}_2)$ and a factor $B:L_{M_1}\to L_{M_2}$, then $B$ is a Jacobi map iff its lgraph
\begin{equation*}
\normalfont
    \LGrph{B}\subset M_1\dtimes M_2
\end{equation*}
is a coisotropic submanifold of the product Jacobi structure $L_{M_1}\utimes \overline{L}_{M_2}$.
\end{prop}
\begin{proof}
We will use characterization 3 in proposition \ref{CoisotropicSubmanifoldsJacobi} to identify a coisotropic submanifold with a Lie subalgebra of the local Lie algebra of a Jacobi structure. Note that, by construction of $\LGrph{B}$, the vanishing sections $\Gamma_{\LGrph{B}}\subset \Sec{L_1\utimes \overline{L}_2}$ are generated by those of the form
\begin{equation*}
    P_1^*B^*s-P_2^*s\qquad \forall s\in\Sec{L_2}.
\end{equation*}
We can directly compute the brackets using the defining relations of the product Jacobi structure
\begin{equation*}
    \{P_1^*B^*s-P_2^*s,P_1^*B^*s'-P_2^*s'\}_{12}=P_1^*\{B^*s,B^*s'\}_1-P_2^*\{s,s'\}_2
\end{equation*}
which must hold for all $s,s'\in\Sec{L_2}$. Then it follows that
\begin{equation*}
    \{P_1^*B^*s-P_2^*s,P_1^*B^*s'-P_2^*s'\}_{12}\in \Gamma_{\LGrph{B}}\quad \Leftrightarrow \quad \{B^*s,B^*s'\}_1 = B^*\{s,s'\}_2
\end{equation*}
as desired.
\end{proof}

Let a line bundle $\lambda:L\to M$ with a Jacobi structure $(\Sec{L},\{,\})$ and consider a submanifold $i:C\hookrightarrow M$ with its corresponding embedding factor $\iota: L_C\to L$. For another line bundle $\lambda':L'\to M'$, assume there exists a submersion factor $\pi:L_C\to L'$ covering a surjective submersion $p:C\twoheadrightarrow M'$. Then, we say that a Jacobi structure $(\Sec{L},\{,\})$ \textbf{reduces} to the Jacobi structure $(\Sec{L'},\{\,\}')$ via $\pi:L_C\to L'$ when for all pairs of sections $s_1,s_2\in\Sec{L'}$ the identity
\begin{equation*}
    \pi^*\{s_1,s_2\}'=\iota^*\{S_1,S_2\}
\end{equation*}
holds for all choices of extensions $S_1,S_2$, i.e all choices of sections $S_1,S_2\in\Sec{L}$ satisfying 
\begin{equation*}
    \pi^*s_1=\iota^*S_1 \qquad \pi^*s_2=\iota^*S_2.
\end{equation*}
The following proposition shows that, in analogy with the case of reduction of Poisson manifolds, coistropic submanifolds play a distinguished role in the reduction of Jacobi manifolds.

\begin{prop}[Coisotropic Reduction of Jacobi Manifolds, cf. {\cite[Proposition 2.56]{tortorella2017deformations}}] \label{CoisotropicReductionJacobi}
Let $\lambda:L\to M$ and $\lambda':L'\to M'$ be line bundles,  $(\Sec{L},\{,\})$ be a Jacobi structure, $i:C\hookrightarrow M$ a closed coisotropic submanifold and $\pi:L_C\to L'$ a submersion factor covering a surjective submersion $p:C\twoheadrightarrow M'$ so that we have the reduction diagram:
\begin{equation*}
    \begin{tikzcd}[sep=tiny]
    L_C \arrow[rr,"\iota"] \arrow[dd,"\pi"'] \arrow[dr]& & L \arrow[dr] & \\
    & C \arrow[rr,"i"', hook] \arrow[dd, "p",twoheadrightarrow] & & M \\
    L' \arrow[dr] & & & \\
    & M' &  & 
    \end{tikzcd}
\end{equation*}
Assume the following compatibility condition between the coisotropic submanifold and the submersion factor:
\begin{equation*}
\normalfont
    \delta(\Ker{\Der \pi})=\Lambda^\sharp((\Tan C)^{\text{0}L}),
\end{equation*}
where $\delta:\Der L \to \Tan M$ is the anchor of the der bundle, $\normalfont(\Tan C)^{0L}\subset \Tan^{*L}M$ is the annihilator and $\normalfont\Lambda^\sharp: \Tan^{*L}M\to \Tan M$ is the squiggle, then there exists a unique Jacobi structure $(\Sec{L'},\{,\}')$ such that $(\Sec{L},\{,\})$ reduces to it via the submersion factor $\pi:L_C\to L'$.
\end{prop}
\begin{proof}
Recall the proposition \ref{CoisotropicSubmanifoldsJacobi} identified a coisotropic submanifold $C$ with the Lie subalgebra of its vanishing submodule $\Gamma_C\subset \Sec{L}$. Since, at least locally, this characterization is accompanied by the vanishing ideal $I_C\subset \Cin{M}$ so that $\Gamma_C=I_C\cdot \Sec{L}$, we can use Propositions \ref{SubManLineBundle} and \ref{QuotientLineBUndle} to regard the functions and sections as quotients of functions and sections that restrict to fibres of the surjective submersion: $\Cin{M'}\cong (\Cin{M}/I_C)_p$ and $\Sec{L'}\cong (\Sec{L}/\Gamma_C)_\pi$. Then we can construct a Lie algebra $(\Sec{L'},\{,\})$ by noting that the compatibility condition ensures that sections restricting to $p$-fibres correspond to the elements in the Lie idealizer $N(\Gamma_C)$; hence $\Sec{L'}\cong N(\Gamma_C)/\Gamma_C$ as Lie algebras. To complete the proof we need to check that this Lie algebra indeed corresponds to a local Lie algebra structure on the line bundle $\lambda':L'\to M'$. This follows from the fact that the compatibility condition can be read as the point-wise requirement of Hamiltonian vector fields of the vanishing sections to be tangent to the $p$-fibres, then a few routine calculations show that the local Lie algebra properties of the bracket $\{,\}$ carry over to the bracket $\{,\}'$ in a natural way.
\end{proof}
As a particular case of coisotropic reduction we find \textbf{Jacobi submanifolds}, i.e. submanifolds $i:C\hookrightarrow M$ whose vanishing submodule $\Gamma_C$ is not only a Lie subalgebra but a Lie ideal, $\{\Gamma_C,\Sec{L}\}\subset \Gamma_C$. In this case, the distribution of Hamiltonian vector fields clearly vanishes, thus satisfying all the requirements of the proposition above trivially, so that the restricted line bundle itself inherits a Jacobi structure $(L_C,\{,\}_C)$. Another example of coistropic reduction is given by the presence of a \textbf{Hamiltonian group action}: a line bundle Lie group action on a Jacobi manifold $G \Acts L$ via Jacobi maps with infinitesimal action $\Psi:\mathfrak{g}\to \Dr{L}$, and a (co)moment map $\overline{\mu}:\mathfrak{g}\to \Sec{L}$ satisfying the defining conditions
\begin{equation*}
    \Psi(\xi)=D_{\overline{\mu}(\xi)},\qquad \{\overline{\mu}(\xi),\overline{\mu}(\zeta)\}=\overline{\mu}([\xi,\zeta]) \qquad \forall \xi,\zeta\in\mathfrak{g}.
\end{equation*}

A \textbf{precontact manifold} is a pair $(M,\text{H}\subset \Tan M)$ with $\text{H}$ a hyperplane distribution, i.e. $\dimm(\text{H}_x)=\dimm (M)-1$ for all $x\in M$. Note that a hyperplane distribution on the tangent bundle is equivalent to the datum of a (generically non-trivial) line bundle $\lambda:L\to M$ and a non-vanishing $L$-valued 1-form $\theta:\Tan M\to L$. The equivalence is realized by setting
\begin{equation*}
    \text{H}=\Ker{\theta}, \quad \text{ which then gives } \quad \Tan M/\text{H}\cong L.
\end{equation*}
Let us denote the $\Cin{M}$-submodule of vector fields tangent to the hyperplane distribution by $\Sec{\text{H}}$. We can define the following antisymmetric map for vector fields tangent to the hyperplane distribution
\begin{align*}
\omega: \Sec{\wedge^2 \text{H}} & \to \Sec{L}\\
(X,Y) & \mapsto \theta([X,Y]).
\end{align*}
The kernel of this map clearly measures the degree to which $\text{H}$ is integrable as a tangent distribution noting that, in particular, when $[\Sec{\text{H}},\Sec{\text{H}}]\subset\Sec{\text{H}}$ the map $\omega$ is identically zero. It follows by construction that $\omega$ is in fact $\Cin{M}$-bilinear and thus defines a bilinear form $\omega_\text{H}:\wedge^2 \text{H}\to L$ called the \textbf{curvature form} of the hyperplane distribution $\text{H}$. A hyperplane distribution $\text{H}$ is called \textbf{maximally non-integrable} when its curvature form $\omega_\text{H}$ is non-degenerate, i.e. when the musical map $\omega^\flat_\text{H}:\text{H}\to \text{H}^{*L}$ has vanishing kernel. Such a hyperplane distribution $\text{H}\subset \Tan M$ is called a \textbf{contact structure} on $M$ and we refer to the pair $(M,\text{H})$ as a \textbf{contact manifold}. Simple dimension counting applied to any tangent space of a contact manifold reveals that a manifold supporting a contact hyperplane is necessarily odd-dimensional. A \textbf{contact map}, defined as a smooth map $\varphi:(M_1,\text{H}_1)\to(M_2,\text{H}_2)$ whose tangent assigns the hyperplane distributions isomorphically $\Tan \varphi (\text{H}_1)=\text{H}_2$, is necessarily a local diffeomorphism from the fact that hyperplane distributions have codimension $1$ everywhere. This condition is equivalent to the tangent map $\Tan \varphi$ inducing a well-defined morphism of line bundles $\Phi:L_1\cong \Tan M_1/\text{H}_1\to L_2\cong \Tan M_2/\text{H}_2$. The presence of a hyperplane distribution on a precontact manifold allows for the identification of \textbf{isotropic submanifolds} as integral submanifolds of the tangent distribution of hyperplanes, i.e. submanifolds $S\subset M$ with $\Tan S\subset \text{H}$. Note that this terminology is appropriate since the curvature form restricted to an isotropic manifold vanishes $\omega_{\text{H}}|_S=0$.\newline

A Jacobi structure $(\Sec{L}, \{,\})$ is called \textbf{non-degenerate} when its Jacobi biderivation $\Pi$ induces a musical isomorphism of lvector bundles
\begin{equation*}
        \begin{tikzcd}
        \Jet^1 L\arrow[r, "\Pi^\sharp",yshift=0.7ex] & \Der L \arrow[l,"\Pi^\flat",yshift=-0.7ex]
        \end{tikzcd}
\end{equation*}
As it is well-known in the standard Poisson geometry literature, symplectic manifolds, conventionally regarded as non-degenerate presymplectic manifolds, are equivalent to non-degenerate Poisson manifolds. It turns out that there is an entirely analogous connection between contact manifolds and non-degenerate Jacobi manifolds.

\begin{prop}[Non-Degenerate Jacobi Manifolds are Contact Manifolds, {\cite[Prop. 2.31-2.32]{tortorella2017deformations}}]\label{NonDegenerateJacobiContact}
Let $M$ be a smooth manifold, then the datum of a contact structure $(M,\text{H}\subset \Tan M)$ is equivalent to the datum of a non-degenerate Jacobi structure $(\Sec{L_M},\{,\})$.
\end{prop}
\begin{proof}
Let us assume $(M,\text{H}\subset \Tan M)$ is a contact structure first and show that it defines a Lie bracket, a symbol and a squiggle, which, in virtue of Proposition \ref{ExtensionSymbol}, will induce a Jacobi structure that is non-degenerate. The curvature form is non-degenerate and thus has an inverse $\omega^\sharp:\text{H}^{*L}\to \text{H}\subset \Tan M$. Set $L:=\Tan M/\text{H}$ and define $\theta\in\Sec{\Tan^{*L} M}$ via the canonical projection. This, in turn, gives an $\Real$-linear map $X:\Sec{L}\to\Sec{\Tan M}$ defined by the condition $\theta(X_s)=s$. This enables us to induce the local Lie bracket on $L$ from the Lie bracket of vector fields via the identity $\{s,r\}:=\theta([X_s,X_r])$. Let us define the squiggle by
\begin{equation*}
    \Lambda^\sharp(df\otimes s):=\omega^\sharp(df|_\text{H}\otimes s).
\end{equation*}
It follows from the definition of curvature form and the Jacobi identity of Lie bracket of vector fields that $X$ and $\Lambda$ so defined satisfy the symbol-squiggle identities. The point-wise decomposition of tangent spaces as $T_xM\cong \text{H}\oplus \Real\cdot X_u(x)$ for some locally non-vanishing section $u$ then shows that the Jacobi biderivation associated to this Jacobi structure is indeed non-degenerate. Conversely, let us assume $\lambda:L\to M$ is a line bundle whose sections carry a non-degenerate Jacobi structure $(\Sec{L},\{,\})$ with squiggle $\Lambda^\sharp:\Tan^{*L}M\to \Tan M$. A point-wise dimension count shows that $\text{H}=\Lambda^\sharp(\Tan^{*L}M)\subset \Tan M$ is a hyperplane distribution and by direct computation we show that it is, in fact, contact and precisely the converse construction to the previous case.
\end{proof}
A submanifold of a contact manifold $S\subset M$ that is isotropic with respect to the hyperplane distribution and that is coisotropic with respect to the associated non-degenerate Jacobi structure is called \textbf{Legendrian}. The non-degeneracy of the Jacobi structure forces Legendrian submanifolds to be maximally isotropic, then, if the odd dimension of the contact manifold is $\dimm M=2n+1$, the dimension of a Legendrian submanifold must be $\dimm S=n$. Similarly to the definition of the Weinstein category of symplectic manifolds, we can use Legendrian submanifolds to define \textbf{the category of contact manifolds} $\Cont_\Man$ whose objects are smooth manifolds with a contact hyperplane distribution $(M,\text{H})$, or equivalently non-degenerate Jacobi structures on line bundles $(\Sec{L_M},\{,\})$, and whose morphisms are Legendrian relations $R:L_1\dashrightarrow L_2$, i.e. Legendrian submanifolds $R\subset M_1\dtimes M_2$ of the line product of non-degenerate Jacobi structures $L_1\utimes \overline{L_2}$. In particular, note that the lgraph of a contact map is clearly maximally coisotropic, thus we see that contact maps are recovered as isomorphisms in this category. As in the case of composition of Lagrangian relations, the composition of Legendrian relations is subject to cleanness of intersection issues.

\subsection{Lie Brackets on Functions Induced by Choice of Unit} \label{UnitPoisson}

In this section we will show that Jacobi manifolds recover well-known structures on functions when restricted to trivialising neighbourhoods, i.e. by choosing units. These structures were originally identified alongside Poisson manifolds by Lichnerowicz \cite{lichnerowicz1978jacobi} in the 1970s and have since been studied extensively (see \cite{ibort1997reduction,valcazar2018contact}). These structures have been referred to as `Jacobi' in the literature but we shall rename them in the interest of a systematic discussion about how they relate to our general notion of Jacobi manifold.\newline

Consider a smooth manifold $M$ with a choice of a bivector field $\pi\in\Sec{\wedge^2 \Tan M}$ and a vector field $R\in\Sec{\Tan M}$. The canonical antisymmetric bracket on functions that can be formed from these
\begin{equation*}
    \{f,g\}:=\pi(df,dg)+fR[g]-gR[f]
\end{equation*}
for all $f,g\in\Cin{M}$, is called a \textbf{Lichnerowicz bracket}. When this bracket satisfies the Jacobi identity $(M,\pi,R)$ is called a \textbf{Lichnerowicz manifold}. It can be shown that the Jacobi identity of this bracket is tantamount to the following Gerstenhaber algebra conditions
\begin{equation*}
    [R,\pi]=0\qquad [\pi,\pi]+2R\wedge\pi =0.
\end{equation*}
A \textbf{Lichnerowicz map} is a smooth map between Lichnerowicz manifolds whose pull-back on functions is a Lie algebra morphisms for the Lichnerowicz brackets or, equivalently, when the bivector and vector fields are related by the tangent map of the smooth map. In the case of $R=0$ we recover the notion of \textbf{Poisson manifolds} and when, furthermore, $\pi$ is non-degenerate we recover \textbf{symplectic manifolds}. In these cases, Lichnerowicz maps recover Poisson maps and symplectomorphisms, respectively.\newline

In a Lichnerowicz manifold $(M,\pi,R)$ with $R\neq 0$ non-degeneracy appears as the condition that bilinear form:
\begin{equation*}
    h:=R\otimes R + \pi \in \Sec{\otimes^2 \Tan M}
\end{equation*}
must induce a musical isomorphism
\begin{equation*}
        \begin{tikzcd}
        \Cot M \arrow[r, "h^\sharp",yshift=0.7ex] & \Tan M \arrow[l,"h^\flat",yshift=-0.7ex]
        \end{tikzcd}.
\end{equation*}
This implies that $\text{rank}(\pi^\sharp)=\dimm M -1$ everywhere, hence we can define a hyperplane distribution by setting $\text{H}:= \pi^\sharp(\Cot M)$ and since $(R\otimes R)^\sharp$ has a transversal image to H, we can identify $\text{H}:=\Ker{h^\flat R}$. The integrability conditions of the Lichnerowicz pair $(\pi, R)$ ensure that H is a contact hyperplane distribution. This corresponds to the well-known case of a co-orientable hyperplane distribution, i.e. one given by the kernel of a non-vanishing ordinary 1-form
\begin{equation*}
    \text{H}=\Ker{\theta},\qquad \theta\in\Omega^1(M).
\end{equation*}
When the hyperplane distribution is a contact structure $\theta$ is called a \textbf{contact form} and $(M,\theta)$ is called an \textbf{exact contact manifold}. It follows by construction that the curvature form of the hyperplane distribution is given by the restriction of the exterior derivative of the contact form
\begin{equation*}
    \omega_\text{H}=-d\theta|_\text{H}=:\omega|_\text{H},
\end{equation*}
and the fact that it is non-degenerate on $\text{H}$ is tantamount to the bilinear form
\begin{equation*}
    \eta:=\theta \otimes \theta + \omega \in \Sec{\otimes^2\Cot M}
\end{equation*}
being non-degenerate, i.e. inducing a musical isomorphism
\begin{equation*}
        \begin{tikzcd}
        \Tan M \arrow[r, "\eta^\flat",yshift=0.7ex] & \Cot M \arrow[l,"\eta^\sharp",yshift=-0.7ex]
        \end{tikzcd}.
\end{equation*}
It then follows that the datum of an exact contact structure $(M,\theta)$ is equivalent to the datum of a non-degenerate Lichnerowicz manifold $(M,\pi,R)$. This equivalence can be seen from the fact that the presence of a musical isomorphism allows for the natural assignment:
\begin{equation*}
    \theta = \eta^\flat R \qquad \omega=\eta^{\flat\flat}\pi
\end{equation*}
where
\begin{equation*}
    \eta^{\flat\flat}\pi(v,w):=\pi(\eta^\flat v, \eta^\flat w) \qquad \forall v,w\in\Tan M.
\end{equation*}

Let us now return to general Jacobi manifolds. Let $\lambda:L\to M$ a line bundle with a Jacobi structure $(\Sec{L},\{,\})$ and consider a unit $u\in \Sec{L^\bullet}$ defined locally on an open subset $U\subset M$. Since local sections on $U$ are $\Cin{M}$-spans of $u$, we can define the following local vector field and bivector field by restricting the symbol and squiggle of the Jacobi structure to $U$:
\begin{equation*}
    R_u:=X_u,\qquad \pi_u(\alpha,\beta)\cdot u := \Lambda(\alpha \otimes u,\beta \otimes u).
\end{equation*}
\begin{prop}[Local Jacobi Structures are Lichnerowicz Structures]\label{LocalJacobi}
Let a Jacobi structure $(\Sec{L},\{,\})$ and a unit $u\in \Sec{L^\bullet}$ defined on $U\subset M$ as above, then $(\pi_u,R_u)$ as defined above form a Lichnerowicz pair and thus define a Lie bracket on local functions satisfying the Jacobi identity given explicitly by
\begin{equation*}
    \{f,g\}_u:=\pi_u(df,dg)+fR_u[g]-gR_u[f]
\end{equation*}
for $\normalfont f,g\in \Cin{U}$.
\end{prop}
\begin{proof}
This is an immediate consequence of the symbol-squiggle identities of Proposition \ref{SymbolSquiggle} for the bracket $\{f\cdot u, g\cdot u\}$. Note first that the bivector can be explicitly given by its musical map
\begin{equation*}
    \pi^\sharp_u(\alpha):= \Lambda^\sharp(\alpha\otimes u).
\end{equation*}
Then, since $\{u,u\}=0$, condition 4. in Proposition \ref{SymbolSquiggle} can be used to directly compute $[R_u,\pi_u]=0$ and condition 5. gives precisely $[\pi_u,\pi_u]+2R_u\wedge\pi_u =0$.
\end{proof}
As an immediate corollary we see that given an arbitrary unit $u$ defined on $U\subset M$ in a Jacobi manifold, we can identify a Poisson subalgebra of local functions as the $R_u$-invariant functions $(\text{C}^\infty_u(U),\{,\}_u)$.\newline

We can consider a different unit $u'\in \Sec{L^\bullet}$ defined on some open $U'$ with $U\cap U'\neq \varnothing$. On the intersection, the two units are related by a conversion factor, i.e. a non-vanishing local function $z\in\Cin{U\cap U'}$, via $u'=z\cdot u$.
\begin{prop}[Conformal Transformations of Lichnerowicz Structures]\label{ConformalJacobi}
Let $u$ and $u'$ two units related by $u'=z\cdot u$ in a Jacobi structure $(\Sec{L},\{,\})$ as above, then, on the intersection of their domains, the local Lichnerowicz brackets are related by
\begin{equation*}
    \{f,g\}_{u'}=z\{f,g\}_u +f\pi_u(dz,dg)-g\pi_u(dz,df)
\end{equation*}
for $\normalfont f,g\in \Cin{U\cap U'}$.
\end{prop}
\begin{proof}
This follows by direct computation using the symbol and squiggle of the Jacobi bracket.
\end{proof}

We are now in the position to identify the key concept that allows us to formally connect Poisson structures with Jacobi structures. A unit $u\in \Sec{L^\bullet}$ defined locally on an open subset $U\subset M$ of a Jacobi structure $(\Sec{L},\{,\})$ is called a \textbf{Poisson unit} when $X_u=0$. It is clear from Proposition \ref{LocalJacobi} that a Poisson unit induces a Poisson algebra structure on local functions $(\Cin{U},\{,\}_u)$. However, a change to another arbitrary unit $u'=z\cdot u$ will not induce another Poisson algebra in general but a Lichnerowicz bracket related by conformal transformation as detailed in Proposition \ref{ConformalJacobi} above.  When both $u$ and $u'$ are Poisson units the conversion factor $z$ satisfies
\begin{equation*}
    \pi_u^\sharp(dz)=0 \qquad \pi_{u'}^\sharp(dz)=0,
\end{equation*}
whence we obtain the conformal transformation of the local Poisson brackets induced by each Poisson unit
\begin{equation*}
    \{f,g\}_{u'}=z\{f,g\}_u
\end{equation*}
for $f,g\in \Cin{U\cap U'}$. It is then natural to define a \textbf{(locally) conformal Poisson structure} as a Jacobi structure $(\Sec{L},\{,\})$ that admits a cover by Poisson units. In the special case when a Jacobi structure admits a global Poisson unit we recover the ordinary notion of Poisson manifold in the context of line bundles, these are appropriately called \textbf{unit-free Poisson structures}. These structures encapsulate families of Poisson brackets that are related via conformal factors. From the perspective of conventional Poisson geometry, these conformal transformations are a manifestation of the trivial fact that for any Poisson bracket $(\Cin{M},\{,\})$, given any function $k\in\Cin{M}$ in the kernel of the adjoint map, $\{k,-\}=0$, the bracket defined by
\begin{equation*}
    \{f,g\}_k:=k\{f,g\}
\end{equation*}
is clearly a Poisson bracket.

\subsection{Jacobi Algebroids associated with Jacobi Manifolds} \label{JacobiManifoldAlgebroids}

In this section we show how the intimate connection between Poisson structures and Lie algebroids generalises to the context of line bundles. To this end, the unit-free formalism developed so far will prove a powerful tool to transparently state the link between Jacobi structures and Jacobi algebroids following arguments and constructions that are entirely analogous to those used in ordinary Poisson geometry.\newline

Recall that the der bundle of an arbitrary line bundle $\Der L$ trivially carries the structure of a Jacobi algebroid. The first connection between Jacobi manifolds and Jacobi algebroids appears as the presence of a Jacobi algebroid structure on the jet bundle.
\begin{prop}[Jet Jacobi Algebroids, {\cite[Prop. 2.24]{tortorella2017deformations}}]\label{JetJacobiAlgebroids}
Let $\lambda: L\to M$ be a line bundle, then there is a 1:1 correspondence between Jacobi structures $(\Sec{L},\{,\})$ and Jacobi algebroid structures on the jet bundle $(\Jet^1 L,\rho,[,])$ such that $[\alpha,\beta]=\beta(\rho(\alpha))$. This equivalence is realised via the formula
\begin{equation*}
    [\alpha,\beta]_\Pi:=\LDer_{\Pi^\sharp\alpha}\beta -\LDer_{\Pi^\sharp\beta}\alpha - j^1\Pi(\alpha,\beta)
\end{equation*}
for $\alpha,\beta\in\Sec{\Jet^1 L}$ and where $\normalfont \Pi\in\Sec{\wedge^2 \Der L}$ is the Jacobi biderivation.
\end{prop}
\begin{proof}
This equivalence is a consequence of the fact that a Jacobi biderivation induces a musical map $\Pi^\sharp:\Jet^1 L\to \Der L$ and that, conversely, a musical map defines an antisymmetric bilinear form. The above formula of the Jacobi algebroid bracket is clearly antisymmetric and a routine computation shows that its Jacobiator is proportional to $\llbracket \Pi,\Pi \rrbracket$, thus showing the connection between the integrability condition of the biderivation and the Jacobi identity of Lie brackets.
\end{proof}
It also follows from this result that Jacobi structures on $L$ can be regarded as LDirac structures of the standard LCourant algebroid $\mathbb{D}L$ by considering the graph of a Jacobi biderivation $\Pi$ as a subbundle $\text{graph}(\Pi^\sharp)\subset \Der L \oplus \Jet^1 L$.\newline

When working with geometric structures on the total space of ordinary vector bundles, it is common to identify subspaces of special functions that are compatible with the vector bundle structure, i.e. fibre-wise constants and fibre-wise linear functions. A well-known result in Poisson geometry is the 1:1 correspondence between linear Poisson manifolds and Lie algebroids. By identifying the analogous notions of fibre-wise constant and linear functions in lvector bundles it will be proved that there is a similar correspondence between linear Jacobi manifolds and Jacobi algebroids.\newline

Let $E^L$ be a lvector bundle and denote the projection to the base manifold by $\epsilon:E\to M$. There is a natural line bundle structure over the total space of a lvector bundle given simply by the pullback line bundle $\lambda:\epsilon^*L\to E$. By construction, sections of this line bundle are
\begin{equation*}
    \Sec{L_E}:=\Sec{\epsilon^*L}=\text{span}_{\Cin{E}}(\epsilon^*\Sec{L})\cong \Cin{E}\otimes_{\Cin{M}}\Sec{L},
\end{equation*}
where the isomorphism is as $\Cin{M}$-modules. There are two special subspaces of sections in $\Sec{L_E}$ sigled out from the presence of a lvector bundle structure: the \textbf{fibre-wise constant sections} $\epsilon^*\Sec{L}$ and the \textbf{fibre-wise linear sections} $l\Sec{E^{*L}}$, given by inclusion of section of the ldual bundle $l:\Sec{E^{*L}}\to \Sec{L_E}$. We fix the following notation for the interaction of the $\Cin{M}$-module structures:
\begin{equation*}
    \epsilon^*f\epsilon^*s=\epsilon^*(f\cdot s), \qquad l_\alpha \epsilon^*s=l_{\alpha\otimes s}, \qquad \epsilon^*fl_\theta=l_{f\cdot\theta}, \qquad l_\theta l_\eta\notin l\Sec{E^{*L}}
\end{equation*}
where $f\in\Cin{M}$, $s\in \Sec{L}$, $\alpha\in\Sec{E^*}$ and $\theta,\eta\in\Sec{E^{*L}}$ and where we are abusing notation by denoting the inclusion of fibre-wise functions on $E$ by the same symbol $l$ as the inclusion of fibre-wise sections on $L_E$.\newline

A \textbf{(fibre-wise) linear Jacobi structure} on a lvector bundle $E^L$ is a Jacobi structure $(\Sec{L_E},\{,\})$ such that
\begin{align*}
    \{l\Sec{E^{*L}},l\Sec{E^{*L}}\} &\subset l\Sec{E^{*L}}\\
    \{l\Sec{E^{*L}},\epsilon^*\Sec{L}\} &\subset\epsilon^*\Sec{L} \\
    \{\epsilon^*\Sec{L},\epsilon^*\Sec{L}\} &= 0.
\end{align*}
Since jets of fibre-wise linear functions span the jet bundle everywhere, the above conditions on spanning sections uniquely determine a Jacobi bracket on all sections of $\Sec{L_E}$.

\begin{prop}[Linear Jacobi Structures and Jacobi Algebroids, {cf. \cite[Prop. 2.28]{tortorella2017deformations}}]\label{LinearJacobi}
There is a 1:1 correspondence between linear Jacobi structures and Jacobi algebroids. This correspondence is realised by lduality in the category of lvector bundles.
\end{prop}
\begin{proof}
Let us first assume that $E^L$ carries a linear Jacobi structure $(\Sec{L_E},\{,\})$ and aim to construct a Jacobi algebroid on its ldual $E^{*L}$. Note that the defining axioms above imply the following identities for the symbol and squiggle of the linear Jacobi bracket
\begin{equation*}
    X_{\epsilon^*s}[\epsilon^*f]=0, \qquad \Lambda^\sharp(d(\epsilon^*f)\otimes \epsilon^*s)[\epsilon^*g]=0.
\end{equation*}
These, together with injectivity of $l$ and $\epsilon^*$, allow us to define a bracket and anchor via the following equations:
\begin{equation*}
    l_{[\theta,\eta]}=\{l_\theta,l_\eta\}, \qquad \epsilon^*(\rho_*\theta[s])=\{l_\theta,\epsilon^*s\}
\end{equation*}
where $s\in\Sec{L}$ and $\theta,\eta\in \Sec{E^{*L}}$. Clearly, the bracket $[,]$ inherits its antisymmetry and Jacobi identity from $\{,\}$ by injectivity of $l$. A simple computation shows that the anchor $\rho_*$ is $\Cin{M}$-linear and a Lie algebra morphism, thus showing that $E^{*L}$ indeed inherits a Jacobi algebroid structure. Conversely, let us assume a Jacobi algebroid structure $(E^L,\rho,[,])$ and aim to construct a Jacobi structure on $\Sec{L_{E^{*L}}}$. The strategy here is to define the Jacobi brackets on spanning sections in the obvious way: for $a,b\in\Sec{E}$ and $s,r\in \Sec{L}$
\begin{align*}
    \{l_a,l_b\} &:= l_{[a,b]}\\
    \{l_a,\epsilon^*s\} &:=\epsilon^*\rho_*a[f] \\
    \{\epsilon^*s,\epsilon^*r\} &:= 0.
\end{align*}
and then use Proposition \ref{ExtensionSymbol} to extend by symbol-squiggle. Note that we are using the lvector bundle isomorphism $E^L\cong(E^{*L})^{*L}$ to regard sections of $E^L$ as fibre-wise linear sections of $E^{*L}$. In order to define the symbol and squiggle as vector fields on the total space of $E^{*L}$, we further consider spanning functions of $\Cin{E^{*L}}$, which correspond to pullbacks $\epsilon^*\Cin{M}$ and inclusions of fibre-wise linear functions from $\Sec{E\otimes L^*}$. The peculiarity that $L$ is a vector bundle of rank $1$ appears in the following identities relating several combinations of the $\Cin{M}$-module structures on spanning functions:
\begin{equation*}
    \sigma(s)\cdot r = \sigma(r) \cdot s, \qquad l_{a\otimes \sigma}\epsilon^*s=\epsilon^*(\sigma(s))l_a=l_{\sigma(s)\cdot a}
\end{equation*}
for $s,r\in \Sec{L}$, $a\in \Sec{E}$ and $\sigma\in\Sec{L^*}$. With this in mind, it suffices to impose the following identities to define the symbol and squiggle:
\begin{align*}
    X_{l_a} [l_{b\otimes \sigma}]&:=l_{[a,b]\otimes \sigma + b\otimes \rho_*a[\theta]} \\
    X_{l_a} [\epsilon^* f]&:=\epsilon^*\delta(\rho_*a)[f] \\
    X_{\epsilon^* s} [l_{b\otimes \sigma}]&:= \epsilon^*\sigma(\rho_*b[s])\\
    \Lambda^\sharp(d(l_{a\otimes \sigma})\otimes \epsilon^* s)[l_{b\otimes \chi}] &:= l_{\sigma(s)[a,b]\otimes \chi + \rho_*a[\chi]\cdot b\otimes \sigma-\rho_*b[\sigma]\cdot a \otimes \chi}\\
    \Lambda^\sharp(d(l_{a\otimes \sigma})\otimes \epsilon^* s)[\epsilon^* f] &:= \epsilon^* \sigma(s)\delta(\rho_*a)[f]
\end{align*}
where the action of derivations on dual sections has been used when necessary. These can be checked to satisfy the symbol-squiggle identities by direct computation, making $\{,\}$ into a Jacobi structure on $E^{*L}$ which is linear by construction. It is easy to see that these two constructions are reciprocal, hence implying the 1:1 correspondence.
\end{proof}

\section{Unit-Free Hamiltonian Mechanics} \label{UnitFreeHamiltonian}

In this section we shall see how the formalisms of line bundles, lvector bundles and Jacobi structures can be used to generalise the conventional formulation of Hamiltonian mechanics to incorporate the notion of unit-free observable. The goal here is to provide a coherent categorical framework that generalises classical mechanics in a meaningful way. As a preliminary step, we first present a condensation of the physical principles underlying the concept of phase space and Hamiltonian dynamics and give a compact categorical formulation of the so-called canonical formalism of classical mechanics.

\subsection{Review of Ordinary Hamiltonian Mechanics} \label{OrdinaryHamiltonian}

There is a vast amount of literature documenting the application of modern symplectic and Poisson geometry to the field of classical mechanics (see the treaties of Arnold \cite{arnold2013mathematical} and Abraham-Marsden \cite{abraham1978foundations,marsden2013introduction}). In this section we extract the essential features of the common definitions of geometric mechanics regarding phase spaces, observables, kinematics and dynamics, and formulate them in a concise categorical form. This will prove vital for our argument that a similar categorical structure can be found in the context of line bundles and Jacobi structures later on in Section \ref{UnitFreeCanonicalHamiltonian}.\newline

In grossly general terms, the fundamental task of a theory of mechanics is to find a mathematical model that takes as inputs the mathematical parameters corresponding to a given experimental preparation $m_0$ at time $t_0$ and produces as output a temporal series $m(t)$ that will predict experimental outcomes at later times $t>t_0$. What follows below is a summary of the main physical intuitions behind the formal content of a dynamical theory in Hamiltonian phase spaces.
\begin{itemize}
    \item \textbf{Principle of Realism.} A physical system exists independent of the observer studying it. Discernible configurations of the system, sets of equivalent experimental preparations and outcomes, are called \textbf{states} $s$ and are identified with points in a smooth manifold $P$, called the \textbf{phase space}.
    \item \textbf{Principle of Characterization.} Properties or characteristics of a system are smooth assignments of measurement outcome values to each point of the phase space. In the case of conventional classical mechanics this gives the usual definition of \textbf{observable} as a real-valued function $f\in \Obs{P}:=\Cin{P}$. Note that observables $\Obs{P}$ form a ring with the usual operations of point-wise addition and multiplication.
    \item \textbf{Principle of Kinematics.} The observer studying the system exists simultaneously with the system. In the same way that the observer's memory state is mapped uniquely into the time interval used to array the experimental temporal series, the physical system is thought to be in a single state corresponding to each of the observer's time parameter values. A \textbf{motion} is defined as a smooth curve in phase space $c:I\to P$ parameterized by the observer's time $t\in I\subset \Real$. Phase spaces are assumed to be path-connected so that any state is  connected to any other state by a motion, at least virtually, not necessarily physically. A family of motions $\{c(t)\}$ is called an \textbf{evolution} $\mathcal{E}$ on the phase space $P$ if all the states are included in the path of some curve. More precisely, we call $\mathcal{E}=\{c(t)\}$ an evolution if
    \begin{equation*}
        \forall s_0\in P \quad \exists c(t)\in \mathcal{E} \quad | \quad c(t_0)=s_0 \text{ for some } t_0\in I.
    \end{equation*}
    \item \textbf{Principle of Observation.} The observed time series of measurement outcomes are the result of observables taking values along a particular motion. More concretely, given a motion $c(t)$ and an observable $f\in\Obs{P}$, the predicted temporal series is simply given by
    \begin{equation*}
        m(t)=(f\circ c)(t).
    \end{equation*}
    \item \textbf{Principle of Reproducibility.} Similar experimental preparations of a physical system should give similar observational outcomes\footnote{By ``observational outcome'' here we mean aggregates of experimental results on which statistical analysis is necessary. The usual classical and quantum measurement paradigms fit into this description.}. An evolution $\mathcal{E}$ implementing this principle will satisfy the following property for all pairs of motions $c(t),c'(t)\in\mathcal{E}$:
    \begin{equation*}
        t_0\in I,\quad c(t_0)=c'(t_0) \quad \Rightarrow \quad c(t)=c'(t) \quad \forall t\in I,
    \end{equation*}
    in other words, $\mathcal{E}$ is a family of non-intersecting curves parameterized by $t\in I$.
    \item \textbf{Principle of Dynamics.} Future states of a physical system are completely determined\footnote{We are careful not to call this determinism since it is the abstract state of a system, not the measurement outcomes, what are assumed to evolve deterministically. Even with measurement paradigms such as a collapse mechanism that forces the definition of a pre-measurement state and post-measurement state, if the theory relies on ordinary differential equations for the modelling of time evolution, the principle of dynamics will be used, if, perhaps, implicitly.} by any given present state, at least locally. Enforcing this condition on motions for arbitrary small time intervals leads to the condition that a motion must be uniquely specified by the values of its tangent everywhere. In other words, a motion satisfying the Principle of Dynamics must be an integral curve of some smooth vector field, thus justifying the chosen name for this principle. Indeed, local existence and uniqueness of solutions to ODEs imply that a smooth vector field $X\in\Sec{\Tan P}$ gives an evolution defined by the family of its integral curves $\mathcal{E}^X$ automatically satisfying the Principle of Reproducibility and the Principle of Dynamics. Arguing from a different angle that focuses on the temporal series of measurement outcomes $m(t)$, conjunction of the Principle of Observation with the Principle of Dynamics leads, via the local notion of directional derivative, to $m(t)$ being the integration of some derivation on observables. We thus conclude that an evolution on a phase space $P$ satisfying all the principles stated above is given by a choice of \textbf{dynamics} $X\in\Dyn{P}$, which is equivalently understood as a vector field on the phase space or a derivation of observables
    \begin{equation*}
        \Sec{\Tan P}\cong : \Dyn{P} :\cong \Dr{\Cin{P}}.
    \end{equation*}
    \item \textbf{Principle of Conservation.} Conserved quantities are a fundamental building block of experimental mechanics: one can only study time-dependent phenomena effectively when enough variables can be assumed to be constants to the effects of the experiment at hand. We could promote this to the more general and abstract requirement that any observable $f\in\Obs{P}$ has an associated evolution $\mathcal{E}^f$, given by some vector field $X_f$, along which the predicted time series are constant. We call such an assignment a \textbf{Hamiltonian map}
    \begin{equation*}
        \eta : \Obs{P}\to \Dyn{P},
    \end{equation*}
    which is required to satisfy
    \begin{equation*}
        \eta(f)[f]=0 \quad \forall f\in\Obs{P}.
    \end{equation*}
    In conventional Hamiltonian mechanics this is assumed to be given by the slightly stronger structure of a Poisson bracket on observables $(\Obs{P},\{,\})$, thus making the phase space into a Poisson manifold $(P,\pi)$. The Hamiltonian map is then $\eta:=\pi^\sharp \circ d$ and, following from the Jacobi identity of the Poisson bracket, it is a Lie algebra morphism.
    \item \textbf{Principle of Reductionism.} The theoretical description of a physical system specified as a subsystem of a larger system must be completely determined by the theoretical description of the larger system and the information of how the smaller system sits inside. This principle is implemented by demanding that knowledge about submanifolds in a phase space allows us to construct new phase spaces. In conventional Hamiltonian mechanics this corresponds the problem of reduction in Poisson manifolds.
    \item \textbf{Principle of Combination.} The theoretical description of a system formed as a combination of other two systems must be completely determined by the theoretical descriptions of each of the parts and the information of how they interact. This is implemented by demanding that there is a \textbf{combination product} construction for phase spaces
    \begin{equation*}
        \between :  (P_1,P_2)\mapsto P_1 \between P_2.
    \end{equation*}
    In the case of conventional Hamiltonian mechanics, this is simply given by the usual product of Poisson manifolds.
    \item \textbf{Principle of Symmetry.} A theoretical description of a system containing states that are physically indistinguishable should contain all the information to form a faithful theoretical description of the system. Physically indistinguishable states are commonly regarded to be orbits of some Lie group action on the phase space, thus an implementation of this principle will require that from the information of a Lie group action preserving some existing structure on a phase space, a new phase space is constructed whose states are classes of physically-indistinguishable states. In Hamiltonian mechanics this is implemented via the theory of Poisson group actions and equivariant moment maps.
\end{itemize}

Blending all these principles together and casting them into categorical form, we introduce the general notion of a \textbf{theory of phase spaces} consisting of the following categorical data: A \textbf{category of phase spaces} $\textsf{Phase}$, which can be identified with some category of smooth manifolds carrying natural notions of subobjects, quotients and a categorical product. A \textbf{category of observables} $\textsf{Obs}$ whose objects carry local\footnote{Here local is used in the sense introduced in the beginning of Section \ref{LocalLieAlgebras}} algebraic structures reflecting the measurement paradigm in which the theory will fit. A \textbf{category of dynamics} $\textsf{Dyn}$ whose objects carry local algebraic structures reflecting the time evolution of the quantities to be measured. These three categories fit in the following commutative diagram of functors:
\begin{equation*}
\begin{tikzcd}[row sep=small]
& \textsf{Obs} \arrow[dd,"\text{Evl}"] \\
\textsf{Phase} \arrow[ur,"\text{Obs}"] \arrow[dr,"\text{Dyn}"'] & \\
& \textsf{Dyn} 
\end{tikzcd}
\end{equation*}
where the \textbf{observable functor} $\text{Obs}$ represents the assignment of measurable properties to a given system, the \textbf{dynamics functor} $\text{Dyn}$ represents the correspondence between motions and smooth curves on phase spaces and the \textbf{evolution functor} $\text{Evl}$ represents how motions should induce the time evolution of observables infinitesimally. A theory of phase spaces is called \textbf{Hamiltonian} if for all phase spaces $P\in\textsf{Phase}$ there exists a canonical choice of Hamiltonian map
\begin{equation*}
    \eta_P: \Obs{P}\to \Dyn{P}
\end{equation*}
which is compatible with the algebraic structures and that captures the notion of conservative evolution.\newline

The fundamental notion at the core of classical mechanics is the concept of \textbf{configuration space}: the set of static states of a physical system, such as the possible spatial positions of moving particles or the possible shapes of a vibrating membrane. Invoking a sort of Principle of Refinement by which mathematical objects representing physical entities are assumed to be continuous and smooth, our definition of a configuration space will be simply that of a smooth manifold whose points $q\in Q$ are identified with the different static states of a given physical system. We then propose the definition of \textbf{the category of configuration spaces} simply to be the category of smooth real manifolds $\Man$. Understanding that smooth manifolds represent sets of static states of a physical system, we now give physical interpretation to the natural categorical structure present in $\Man$:
\begin{itemize}
    \item The measurable static properties of a physical system with configuration space $Q\in\Man$, what we call \textbf{static observables}, are the real-valued functions $\Obs{Q}:=\Cin{Q}$. The assignment of static observables is a contravariant functor $\text{C}^\infty:\Man\to\Ring$, which we regard as the categorical version of the Principle or Characterization for configuration spaces.
    \item A subsystem is characterised by restricting possible positions of a larger system, that is, by an inclusion of an embedded submanifold $i:S\to Q$. This notion implements the Principle of Reductionism for configuration spaces.
    \item Physically-indistinguishable static states in a configuration space $Q\in\Man$ are related by equivalence relations $\sim$ that have quotients faithfully characterizing the physical system, that is, there is a surjective submersion $p:Q\to Q/\sim$. In particular, a free and proper action of a Lie group $G\Acts Q$ gives an example implementing the Principle of Symmetry via $p:Q\to Q/G$.
    \item Given two configuration spaces $Q_1,Q_2\in\Man$ representing the possible positions of two physical systems, the combined system will have static states given by all the possible pairs of static states in each of the two systems. The categorical manifestation of the Principle of Combination for configuration spaces is then simply the presence of the Cartesian product of smooth manifolds $Q_1\times Q_2\in\Man$.
    \item A temporal series of static states will be called a \textbf{path} of the physical system. Paths in a configuration space $Q\in\Man$ will be given by smooth curves $r:I\to Q$ parameterized by the observer's time parameter $t\in I\subset \Real$. Dynamics on the static states of a configuration space $Q\in\Man$ are identified with smooth vector fields $\Dyn{Q}=\Sec{\Tan Q}$ or, equivalently, with derivations $\Dyn{Q}=\Dr{\Cin{Q}}$.
\end{itemize}

We thus see that the category of configuration spaces provides the first natural example of a theory of phase spaces where the observable functor is simply the assignment of the ring of smooth functions to a manifold $\text{Obs}=\text{C}^\infty:\Man \to \Ring$, the dynamics functor is the tangent functor $\text{Dyn}=\Tan : \Man \to \Lie_\Man$ and the evolution functor is given by taking the vector bundle of smooth ring derivations $\text{Evl}=\text{Der}: \Ring \to \Lie_\Man$, i.e. the tangent bundle regarded as a Lie algebroid. Indeed, these three functors fit in the phase space theory commutative diagram
\begin{equation*}
\begin{tikzcd}[row sep=small]
& \Ring \arrow[dd,"\text{Der}"] \\
\Man \arrow[ur,"\text{C}^\infty"] \arrow[dr,"\Tan"'] & \\
& \Lie_\Man 
\end{tikzcd}
\end{equation*}
Note, however, that, since the categorical information of each object of $\Man$ is strictly a smooth manifold with no canonical choice of extra structure on it, the category of configuration spaces does not provide an example of a Hamiltonian theory of phase spaces.\newline

When approached from this categorical angle, the mathematical implementation of the conventional \textbf{canonical formalism of classical mechanics} can be understood quite simply as the search for a category of phase spaces associated to the category of configuration spaces that forms a Hamiltonian theory of phase spaces in a natural or \emph{canonical} way. As shown below, this will be achieved by the identification of the canonical symplectic structures on the cotangent bundles of smooth manifolds.\newline

Another, more physically-informed, approach to the canonical formalism of classical mechanics is the search for a category of phase spaces whose observables encompass both the static and dynamics states of a physical system as motivated by the basic postulates of Newtonian mechanics, where positions and velocities are the initial data for the deterministic evolution of the system. More precisely, a phase space $P$ associated to a configuration space $Q$ should carry a space of observables naturally containing the static observables of $Q$ and its dynamics $\Obs{Q},\Dyn{Q}\subset\text{Obs}(P)$. We see that the cotangent bundle $P=\Cot Q$ appears, again, as the natural or \emph{canonical} choice for such a phase space since it is clear that the static observables $\Obs{Q}=\Cin{Q}$ and the dynamics $\Dyn{Q}=\Sec{\Tan Q}$ are recovered as the fibre-wise constant and fibre-wise linear functions of the cotangent bundle $\Cin{\Cot Q}=\Obs{\Cot Q}$.\newline

In order to formally identify the categorical properties of what will become the category of canonical phase spaces we recall some well-known facts about the symplectic geometry of cotangent bundles.\newline

The cotangent bundle of any smooth manifold $Q$ carries a canonical symplectic structure
\begin{equation*}
    \begin{tikzcd}
    \Cot Q \arrow[d, "\pi_Q"]\\
    Q
    \end{tikzcd}
    \qquad \qquad \omega_Q:=-d\theta_Q, \quad \text{ with }\quad  \theta_Q|_{\alpha_q}(v):=\alpha_q(\Tan_{\alpha_q}\pi_Q (v)) \quad \forall v\in \Tan_{\alpha_q}(\Cot Q).
\end{equation*}
The non-degenerate Poisson structure $(\Cin{\Cot Q},\cdot,\{,\}_Q)$, with the natural inclusions of vector fields and functions on $Q$ as functions on $\Cot Q$:
\begin{equation*}
    l_Q:\Sec{\Tan Q}\to \Cin{\Cot Q}, \qquad \pi_Q^*:\Cin{Q}\to \Cin{\Cot Q},
\end{equation*}
is shown to be linear. It is, in fact, dual to the canonical Lie algebroid structure on $\Tan Q$:
\begin{align*}
    \{l_Q(X),l_Q(Y)\}_Q & =  l_Q([X,Y])\\
    \{l_Q(X),\pi_Q^*f\}_Q & =  \pi_Q^* X [f]\\
    \{\pi_Q^*f,\pi_Q^*g\}_Q & = 0.
\end{align*}
The cotangent bundle of a Cartesian product is canonically symplectomorphic to the vector bundle product of cotangent bundles with the induced product symplectic forms:
\begin{equation*}
    (\Cot (Q_1\times Q_2),\omega_{Q_1\times Q_2})\cong (\Cot Q_1\boxplus \Cot Q_2, \Proj_1^*\omega_{Q_1} \oplus \Proj_2^*\omega_{Q_2}).
\end{equation*}
A smooth map $\varphi:Q_1\to Q_2$ induces a Lagrangian relation $\Cot \varphi\subset \Cot Q_1 \times \overline{\Cot Q_2}$ called the cotangent lift of $\varphi$  and defined by
\begin{equation*}
   \Cot \varphi:=\{(\alpha_q,\beta_p)| \quad \varphi(q)=p, \quad \alpha_q=(\Tan_q\varphi)^*\beta_p\}.
\end{equation*}
Here $\overline{\Cot Q_2}$ denotes $(\Cot Q_2,-\omega_{Q_2})$. When $\varphi:Q_1\to Q_2$ is a diffeomorphism, its cotangent lift becomes the graph of a symplectomorphism $\Cot \varphi :\Cot Q_2\to \Cot Q_1$. Composition of cotangent lifts as Lagrangian relations is always strongly transversal in the Weinstein symplectic category $\Symp_\Man$, where we regard the above construction as a morphism $\Cot \varphi:\Cot Q_2\dashrightarrow \Cot Q_1$, and thus the cotangent bundle construction can be seen as a contravariant functor
\begin{equation*}
    \Cot : \Man \to \Symp_\Man.
\end{equation*}
An embedded submanifold $i:S\to Q$, what is known as holonomic constraints in standard classical mechanics, gives the coisotropic submanifold $\Cot Q|_S\subset (\Cot Q,\omega_Q)$ inducing the coisotropic reduction diagram
\begin{equation*}
    \begin{tikzcd}[sep=small]
        \Cot Q|_S \arrow[r, hookrightarrow] \arrow[d, twoheadrightarrow] & (\Cot Q,\omega_Q) \\
        (\Cot S,\omega_S).
    \end{tikzcd}
\end{equation*}
Note that the surjective submersion results from quotienting by the foliation given by the conormal bundle of $S$, in other words
\begin{equation*}
        \Cot S\cong \Cot Q|_S/(\Tan S)^0
\end{equation*}
as vector bundles over $S$. Lastly, for a free and proper group action $\phi: G\times Q\to Q$, with infinitesimal action $\psi:\mathfrak{g}\to \Sec{\Tan Q}$ and smooth orbit space $\Tilde{Q}:=Q/G$, the cotangent lift gives an action by symplectomorphisms $\Cot \phi : G\times \Cot Q\to \Cot Q$ with equivariant (co)moment map given by
\begin{equation*}
        \overline{\mu}:=l_Q\circ \psi : \mathfrak{g}\to \Cin{\Cot Q},
\end{equation*}
this induces the symplectic reduction diagram
\begin{equation*}
    \begin{tikzcd}[sep=small]
    \mu^{-1}(0) \arrow[r, hookrightarrow] \arrow[d, twoheadrightarrow] & (\Cot Q,\omega_Q) \\
    (\Cot \Tilde{Q},\omega_{\Tilde{Q}}).
    \end{tikzcd}
\end{equation*}

In light of these results, we are compelled to define the \textbf{category of canonical symplectic phase spaces} simply as the image of the category of configuration spaces under the cotangent functor $\Cot (\Man)$. This is clearly a theory of phase spaces with notions of observable, dynamics and evolution functors as in the case of configuration spaces. Furthermore, the presence of a canonical symplectic structure on each phase space allows to define Hamiltonian maps simply by assignment of the Hamiltonian vector field to a function
\begin{equation*}
    \eta_Q:=\omega_Q^\sharp \circ d = \text{ad}_{\{,\}_Q}: \Obs{\Cot Q}\to \Dyn{\Cot Q}.
\end{equation*}
This makes the category $\Cot (\Man)$ into a Hamiltonian theory of phase spaces, thus achieving the motivating goal of finding a Hamiltonian theory of phase spaces canonically associated with the category of configuration spaces.\newline 

We are now in the position to argue that the structural content of standard canonical Hamiltonian mechanics can be encapsulated in the cotangent functor. The cotangent functor is the motivating example of the general notion of a \textbf{Hamiltonian functor} sending a generic theory of phase spaces into a theory of Hamiltonian phase spaces while preserving all the categorical structure. This is summarised in the following table:
\begin{center}
\begin{tabular}{ c c c }
Configuration Spaces & Hamiltonian Functor  & Phase Spaces \\
\hline
 $\Man$ & $\begin{tikzcd}\phantom{A} \arrow[r,"\Cot"] & \phantom{B} \end{tikzcd}$ & $\Symp_\Man$ \\ 
 $Q$ & $\begin{tikzcd} \phantom{Q} \arrow[r, "\Cot ",mapsto] & \phantom{Q} \end{tikzcd}$ & $(\Cot Q,\omega_Q)$ \\ 
 $\Obs{Q}$ & $\begin{tikzcd} \phantom{Q} \arrow[r,hookrightarrow, "\pi_Q^*"] & \phantom{Q}  \end{tikzcd}$ & $\Obs{\Cot Q}$ \\
 $\Dyn{Q}$ & $\begin{tikzcd} \phantom{Q} \arrow[r,hookrightarrow, "l_Q"] & \phantom{Q}  \end{tikzcd}$ & $\Obs{\Cot Q}$ \\
 $Q_1\times Q_2$ & $\begin{tikzcd} \phantom{Q} \arrow[r,mapsto, "\Cot "] & \phantom{Q}  \end{tikzcd}$ & $\Cot Q_1\boxplus \Cot Q_2$ \\
 $\varphi:Q_1\to Q_2$ & $\begin{tikzcd} \phantom{Q} \arrow[r,mapsto, "\Cot "] & \phantom{Q}  \end{tikzcd}$ & $\Cot \varphi:\Cot Q_2\dashrightarrow \Cot Q_1$\\
 $S\subset Q$ & $\begin{tikzcd} \phantom{Q} \arrow[r,mapsto, "\Cot"] & \phantom{Q}  \end{tikzcd}$ & $\Cot Q|_S\subset \Cot Q$ coisotropic \\
 $G\Acts Q$ & $\begin{tikzcd} \phantom{Q} \arrow[r,mapsto, "\Cot"] & \phantom{Q}  \end{tikzcd}$ & $G\Acts \Cot Q$ Hamiltonian \\
\end{tabular}
\end{center}
The Hamiltonian functor connects the category of configuration spaces and the category of canonical symplectic phase spaces regarded as theories of phase spaces as summarized by the following diagram:
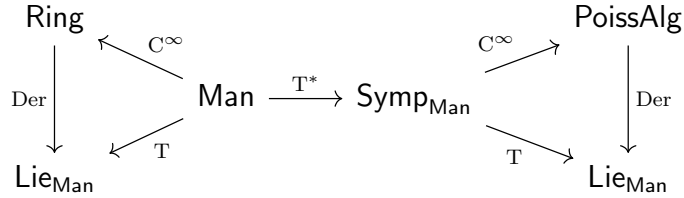
\begin{equation*}
\begin{tikzcd}[row sep=small]
 & \Ring \arrow[dd,"\text{Der}"']& & & \textsf{PoissAlg} \arrow[dd,"\text{Der}"] \\
 & & \Man \arrow[ul, "\text{C}^\infty"']\arrow[dl,"\Tan"]\arrow[r,"\Cot"] & \Symp_\Man \arrow[ur,"\text{C}^\infty"] \arrow[dr,"\Tan"'] & \\
 & \Lie_\Man & & & \Lie_\Man
\end{tikzcd}
\end{equation*}

Once a physical system is assigned a configuration space $Q$, together with its canonically associated phase space $\Cot Q$ as above, the only remaining task left for the observer studying the system is to determine which choice of dynamics on $\Cot Q$ will produce evolutions of the system that match experimental temporal series of measurements. The Poisson structure present in $\Cot Q$ reduces this problem to the choice of an observable $h\in\Obs{\Cot Q}$ since, under the Hamiltonian map, this automatically gives a choice of conservative dynamics $\eta_Q(h)\in\Dyn{\Cot Q}$. This distinguished observable is often called the \textbf{energy} of the system. It generates the time evolution of the system and it is a fundamental conserved quantity. Algebraically, this is a trivial fact by construction since
\begin{equation*}
    \eta_Q(h)[h]=\{h,h\}_Q=0
\end{equation*}
by antisymmetry of the Lie bracket. This is the implementation of the Principle of Conservation within the category of symplectic phase spaces. Given two systems with a choice of energy $(Q_1,h_1)$ and $(Q_2,h_2)$, where $h_i\in\Cin{\Cot Q_i}$, their combination product has a canonical choice of energy given by the sum of pull-backs $h_1+h_2:=\Proj_1^*h_1+\Proj_2^*h_2\in\Cin{\Cot Q_1\times \Cot Q_2}$. This gives an extra line of assignments to the Hamiltonian functor:
\begin{center}
\begin{tabular}{ c c c }
Configuration Spaces & Hamiltonian Functor  & Phase Spaces \\
\hline
  $(Q_1,h_1)\times (Q_2,h_2)$ & $\begin{tikzcd} \phantom{Q} \arrow[r,mapsto, "\Cot "] & \phantom{Q}  \end{tikzcd}$ & $h_1+h_2 \in\Cin{\Cot Q_1\times \Cot Q_2}$ \\
\end{tabular}
\end{center}

The phase space formalism described so far in this section is general enough to account for a vast class of mechanical systems, however, this generality comes at a price: the Hamiltonian functor above fails to select a preferred choice of energy observable for a given configuration space. Turning to one of the earliest examples of mechanics we find inspiration to redefine the category of configuration spaces in order to account for some extra physical intuitions. In Newtonian mechanics, configuration spaces are submanifolds of Euclidean space and Cartesian products thereof, thus always carrying a Riemannian metric that encodes the physical notion of distance and angle; often also with a choice of potential, which is a function on the configuration space. This motivates us to refine our notion of configuration space and define the \textbf{category of Newtonian spaces} $\textsf{Newton}_\Man$ whose objects are triples $(Q,g,V)$, where $Q\in\Man$, $g\in \Sec{\odot ^2 \Cot Q}$ Riemannian metric and $V\in\Cin{Q}$, and whose morphisms $\varphi:(Q_1,g_1,V_1)\to (Q_2,g_2,V_2)$ are smooth maps $\varphi:Q_1\to Q_2$ such that $g_1-\varphi^*g_2$ is positive semi-definite and $V_1=\varphi^*V_2$. When $\varphi$ is a diffeomorphism, a morphism in this category is an isometry between $Q_1$ and $Q_2$. Note that a metric defines a quadratic form on tangent vectors $K_g:\Tan Q\to \Real$ given simply by $K_g(v):=\tfrac{1}{2}g(v,v)$, then we see that the datum of a Newtonian configuration space $(Q,g,V)$ is equivalent to a choice of function on static states $V\in\Cin{Q}$ and a choice of quadratic function on velocities $K\in\Cin{\Tan Q}$. The Cartesian product of configuration spaces gets updated to a categorical product in $\textsf{Newton}_\Man$ by setting
\begin{equation*}
    (Q_1,g_1,V_1) \times (Q_2,g_2,V_2) := (Q_1\times Q_2,g_1+g_2,V_1+V_2)
\end{equation*}
where
\begin{align*}
    g_1+g_2 &:= \Proj^*_1g_1 \oplus \Proj_2^*g_2\in \Sec{\odot ^2 (\Tan Q_1\boxplus \Tan Q_2))}\cong\Sec{\odot ^2 \Tan (Q_1\times Q_2)}, \\
    V_1+V_2 &:= \Proj^*_1V_1+\Proj_2^*V_2\in \Cin{Q_1\times Q_2}.
\end{align*}
A metric $g$ on $Q$ gives the usual musical isomorphism
\begin{equation*}
        \begin{tikzcd}
        \Tan Q \arrow[r, "g^\flat",yshift=0.7ex] & \Cot Q \arrow[l,"g^\sharp",yshift=-0.7ex]
        \end{tikzcd},
\end{equation*}
which can be used to regard the quadratic function $K$ identified with the metric $g$ as a quadratic function on the cotangent bundle by pull-back: $\sharp^*K_g\in\Cin{\Cot Q}$. Once a Newtonian space $(Q,g,V)$ is fixed, we now see that there is a canonical choice of energy:
\begin{equation*}
    E_{g,V}:=(g^\sharp)^*K_g+\pi_Q^*V\in \Cin{\Cot Q},
\end{equation*}
which is called the \textbf{Newtonian energy}. This name is further justified by the fact that a direct computation shows that solving the Hamiltonian dynamics of this observable is equivalent to solving Newton's equations on a Riemannian manifold background $(Q,g)$ and with force field $F=-(g^\sharp)(dV)$. For two Newtonian spaces $(Q_1,g_1,V_1)$ and $(Q_2,g_2,V_2)$, the categorical product construction above gives the following additivity property of Newtonian energy:
\begin{equation*}
    E_{g_1+g_2,V_1+V_2}=E_{g_1,V_1}+E_{g_2,V_2}.
\end{equation*}

The introduction of a metric and a potential in a configuration space does not affect the canonical constructions of the Hamiltonian functor. Then, in the category of Newtonian spaces $\textsf{Newton}_\Man$ we update the Hamiltonian functor table above by adding the preferred choice of Newtonian energy and the additivity of energy for combined systems without interaction:
\begin{center}
\begin{tabular}{ c c c }
Newtonian Spaces & Hamiltonian Functor  & Phase Spaces \\
\hline
 $\Man_{\textsf{Newton}}$ & $\begin{tikzcd}\phantom{A} \arrow[r,"\Cot"] & \phantom{B} \end{tikzcd}$ & $\Symp_\Man$ \\ 
 $(Q,g,V)$ & $\begin{tikzcd} \phantom{Q} \arrow[r, "\Cot ",mapsto] & \phantom{Q} \end{tikzcd}$ & $E_{g,V}\in\Obs{\Cot Q}$ \\
 $(Q_1,g_1,V_1)\times (Q_2,g_2,V_2)$ & $\begin{tikzcd} \phantom{Q} \arrow[r, "\Cot ",mapsto] & \phantom{Q} \end{tikzcd}$ & $(\Cot (Q_1\times Q_2),E_{g_1,V_1}+E_{g_2,V_2})$
\end{tabular}
\end{center}
Note that this result in the category of Newtonian configuration spaces motivates the additivity of energy of the combination product of two general phase spaces $(\Cot Q_1\boxplus \Cot Q_2,h_1+h_2)$ discussed above.

\subsection{The Unit-Free Generalisation of Phase Spaces} \label{UnitFreeGeneralisation}

Our goal is to introduce the notion of observables of different physical dimension, with varying choices of unit for them, over the same geometric phase space of a physical system. According to the conventions of dimensional analysis in practical science and engineering, a defining feature of physical quantities is the existence of conversion factors between units of the same dimension. Our discussion about proportionality factors in 1-dimensional real vector spaces surrounding Proposition \ref{RatMaps} in Section \ref{Lines} clearly suggests that lines are the appropriate mathematical representation of the freedom of choice of unit of physical quantities.\newline

Inspired by the general physical principles of phase space formalisms outlined at the beginning of Section \ref{OrdinaryHamiltonian} and guided by the intuition that the characteristics to be measured of a physical system should be implemented by elements of lines, we are compelled to define the \textbf{category of unit-free configuration spaces} as the category of line bundles over smooth manifolds $\Line_\Man$. Although this generalization is motivated  from physical principles in this section, the degree of its success will be measured by whether we can recover a canonical theory of phase spaces via a Hamiltonian functor that mirrors the cotangent functor for ordinary configuration spaces. Having all the technical results about line bundles presented in Section \ref{LineBundles} at hand, we give a physical interpretation for the categorical structure naturally present in $\Line_\Man$:
\begin{itemize}
    \item The physical interpretation the base manifold of a unit-free configuration space $L_Q$ is identical to a configuration space, they represent the \textbf{static states} of the physical system. In this sense, the space of static states of a physical system is independent of the particular dimensions of the physical quantities that will be measured from it.
    \item The measurable static properties a physical system are identified with the smooth sections of some fixed line bundle $\Sec{L_Q}$. The line bundle $L_Q$ will be appropriately called the \textbf{configuration space $Q$ of dimension $L$}. We thus identify the collection of all possible measurable properties of a fixed physical dimension with a choice of unit-free configuration space $L_Q$. We call these the \textbf{static observables of dimension $L$} of the configuration space $Q$ and denote them by $\Obs{L_Q}:=\Sec{L_Q}$. Properties of factor pull-backs in the category of line bundles then ensure that we have an observable contravariant functor
    \begin{equation*}
        \text{Obs}:\Line_\Man\to\textsf{RMod}.
    \end{equation*}
    \item We recover (unit-less) observables of ordinary configuration spaces via the notion of units on line bundles, which we recall are local non-vanishing sections $u\in\Sec{L_Q^\bullet}$ defined on some open subset $U\subset Q$. Restricting to the open subset $U$ allows us to see any other static observable $s\in\Obs{L_Q}$ as a local real-valued function $f_u$ determined uniquely by the equation $s=f_u\cdot u$. Thus, a unit $u$ allows (locally) for a functorial assignment of the form
    \begin{equation*}
        u:\Obs{L_Q}\to \Obs{Q}.
    \end{equation*}
    \item A subsystem is characterised by restricting possible positions of a larger system, that is, by an inclusion of an embedded submanifold $i:S\to Q$. Our discussion about submanifolds of line bundles and embedding factors in Section \ref{LineBundles} ensures that the the information of a submanifold on the base is enough for the line bundle structure to follow.
    \item Physically-indistinguishable static states must produce measurement outcomes that are indistinguishable as elements of a line bundle over the configuration space and thus should give a basic quotient of line bundles. Submersion factors and group actions on line bundles as discussed in Section \ref{LineBundles} encapsulate this notion.
    \item Given two unit-free configuration spaces $L_{Q_1}$ and $L_{Q_2}$, the line product construction of Proposition \ref{ProductLineBundle} gives the direct analogue of the Cartesian product of conventional configuration spaces. We thus regard $L_{Q_1}\utimes L_{Q_2}$ as the categorical implementation of the Principle of Combination for unit-free configuration spaces.
    \item Paths of a physical system, i.e. temporal series of static states, are simply recovered as smooth curves on the base space of a unit-free configuration space. In this manner, conventional dynamics $\Dyn{Q}$ are simply recovered as the vector fields on the base manifold. However, the extra structure introduced by the presence of the line bundle induces a new dynamical aspect of configuration spaces. Given a unit-free configuration space $L_Q$, all the non-zero fibre elements over a point $L_q$ represent different choices of unit for the same type of physical quantity. Then, any measurement performed on a system moving along a path $c(t)$ passing through $q$ at a time $t_0$ will have to be unit-compatible with any measurement performed at a later time $t>t_0$. This means that the choice of unit should be preserved along the motion of a path. Considering a unit $u$ as a local section, this is ensured locally by construction, however, taking all the possible arbitrary choices of local unit around the point $q$, forces the existence of a $1$-parameter family of fibre-wise isomorphisms covering the smooth curve $c(t)$. These are nothing but smooth families of line bundle automorphisms, which are given infinitesimally by line bundle derivations, and thus we identify the \textbf{unit-free dynamics} of a unit-free configuration space $L_Q$ as the der bundle of the line bundle $\Dyn{L_Q}:=\Der L_Q$. The anchor of the der bundle $\delta:\Der L_Q\to \Tan Q$ allows to connect dimensioned dynamics with ordinary dynamics via:
    \begin{equation*}
        \delta:\Dyn{L_Q}\to \Dyn{Q}.
    \end{equation*}
\end{itemize}

Similarly to the case of ordinary configuration spaces discussed in Section \ref{OrdinaryHamiltonian} we see that unit-free configuration spaces become another example of a theory of phase spaces. The observable functor is the assignment of sections of line bundles $\text{Obs}=\Gamma :\Line_\Man \to \textsf{RMod}$, the dynamics functor is the der functor $\text{Dyn}=\Der : \Line_\Man\to \textsf{Jacb}_\Man$ and the evolution functor is given by taking the vector bundle of module derivations of the spaces of sections regarded as Jacobi algebroids $\text{Evl}=\text{Der}:\textsf{RMod}\to \textsf{Jacb}_\Man$. The identification of module derivations as the sections of the der bundle
\begin{equation*}
    \Sec{\Der L}\cong \Dr{L},
\end{equation*}
implies that these functors fit in the following commutative diagram
\begin{equation*}
\begin{tikzcd}[row sep=small]
& \textsf{RMod} \arrow[dd,"\text{Der}"] \\
\Line_\Man \arrow[ur,"\Gamma"] \arrow[dr,"\Der"'] & \\
& \textsf{Jacb}_\Man 
\end{tikzcd}
\end{equation*}
thus making the category of unit-free configuration spaces into a theory of phase spaces. Since no additional structure is assumed on generic line bundles, the category of unit-free configuration spaces is not a Hamiltonian theory of phase spaces.

\subsection{Canonical Contact Structures on Jet Bundles} \label{CanonicalContact}

The strategy for the search of a Hamiltonian theory of phase spaces that is canonically associated to the category of unit-free configuration spaces defined in Section \ref{UnitFreeGeneralisation} above will be entirely analogous to the case of ordinary configuration spaces: given a line bundle regarded as a unit-free configuration space we construct the \emph{unit-free cotangent bundle}, that is, the jet bundle. In order to show that jet bundles are to line bundles what cotangent bundles are to ordinary manifolds we must prove some technical results about jet bundles using the formalisms of lvector bundles and Jacobi structures introduced throughout this paper. The facts presented in this section will lead to the eventual conclusion in Section \ref{UnitFreeCanonicalHamiltonian} below that unit-free phase spaces indeed generalise ordinary ones.\newline

Firstly, we identify the \textbf{canonical contact structure} found in the jet bundles of line bundles.

\begin{prop}[Canonical Contact Manifold Associated to a Line Bundle]\label{JetCanonicalContactStructure}
Let $\lambda:L\to Q$ be a line bundle and $\normalfont \pi:\Jet^1L \to Q$ its jet bundle, then there is a canonical contact structure $\normalfont (\Jet^1 L,\text{H}_L)$ such that the contact line bundle is isomorphic to the pull-back of the line bundle on the base manifold:
\begin{equation*}
\normalfont
    \Tan (\Jet^1 L)/\text{H}_L\cong \pi^*L.
\end{equation*}
We denote this line bundle by $\normalfont L_{\Jet^1 L}$. Furthermore, the non-degenerate Jacobi structure induced in the line bundle $\normalfont (L_{\Jet^1 L},\{,\}_L)$ is fibre-wise linear and it is completely determined by the algebraic structure of derivations acting on sections
\begin{align*}
    \{l_a,l_b\}_L & =  l_{[a,b]}\\
    \{l_a,\pi^*s\}_L & =  \pi^* a [s]\\
    \{\pi^*s,\pi^*r\}_L & = 0
\end{align*}
for all $s,r\in\Sec{L}$, $\normalfont a,b\in\Sec{\Der L}$ and where $\normalfont l:\Sec{\Der L}\to \Sec{L_{\Jet^1 L}}$ is the inclusion of derivations as fibre-wise linear sections of $\normalfont L_{\Jet^1 L}$.
\end{prop}
\begin{proof}
The hyperplane distribution $\text{H}_L\subset \Tan (\Jet^1 L)$ is the usual Cartan distribution defined in general jet bundles of vector bundles, however, in our case it can be regarded as the kernel of the \textbf{canonical contact form} $\theta\in\Sec{\Tan^{*(\pi^*L)}(\Jet^1 L)}$ defined as the line bundle analogue of the Liouville $1$-form on the cotangent bundle. More precisely, denoting by $\varpi:\Jet^1 L\to L$ the surjective bundle map of the jet sequence, the canonical contact form is explicitly given at any point $j^1_xu\in\Jet^1 L$ of the jet bundle by
\begin{equation*}
    \theta_{j^1_xu}:=\Tan_{j^1_xu}\varpi - \Tan_x u\circ \Tan_{j^1_xu}\pi
\end{equation*}
where we note that the explicit use of $\Tan_xu$ is well-defined from the fact that $j^1u$ is defined as the equivalence class of all sections agreeing in value and tangent map at $x\in Q$. The map above is mapping tangent spaces of the vector bundles $\theta_{j^1_xu}:\Tan_{j^1_xu}\Jet^1 L\to \Tan_{u(x)}L$, in order to make it into a $L$-valued $1$-form we will use the fact that vertical subspaces of the total space of a vector bundle are canonically isomorphic to the fibres, $\Tan_{u(x)}^{\text{Vert}}L=\Ker{\Tan_{u(x)} \lambda}\cong L_x$ and that the image of $\theta_{j^1_xu}$ is always vertical
\begin{equation*}
    \Tan_{u(x)} \lambda\circ \theta_{j^1_xu}=\Tan_{j^1_xu}(\lambda\circ \varpi)-\Tan_{x}(\lambda\circ u)\circ \Tan_{j^1_xu}\pi=\Tan_{j^1_xu}\pi - \Tan_{j^1_xu}\pi =0,
\end{equation*}
where we have used the fact that sections and jet prolongations fit in the commutative diagram
\begin{equation*}
\begin{tikzcd}
\Jet^1 L \arrow[rr, "\varpi"] \arrow[dr, "\pi"] & & L \arrow[dl, "\lambda"'] \\
 & Q \arrow[ul,bend left,"j^1u"]\arrow[ur,bend right,"u"'] & 
\end{tikzcd}
\end{equation*}
We then define $\text{H}_L:=\Ker{\theta}$, which is shown to be a hyperplane distribution from simple point-wise dimension counting. Applying the first isomorphism theorem for vector spaces fibre-wise, we find $\Tan (\Jet^1 L)/\text{H}_L\cong \pi^*L$, as desired. The non-degenerate Jacobi structure on $\pi^*L$ appears as the ldual to the Jacobi algebroid structure present in $\Der L$ analogously to the canonical symplectic structure being the linear Poisson structure on $\Cot Q$ dual to the Lie algebroid $\Tan Q$. The construction of the Jacobi structure $(L_{\Jet^1 L},\{,\}_L)$ is done as a straightforward application of Proposition \ref{LinearJacobi}. To complete the proof we simply need to show that the contact structure $\text{H}_L$ identified above coincides with the one induced by the Jacobi structure. This results follows the fact that $\Lambda^\sharp(\Tan^{*(\pi^*L)}(\Jet^1 L)) = \text{H}_L$, which is proved via a simple dimension count argument and the observation that, by construction, $\theta(\Lambda^\sharp(\alpha\otimes \pi^*s))=0$ for all $\alpha\in \Sec{\Cot (\Jet^1 L)}$ and $s\in\Sec{L}$.
\end{proof}

The following result establishes the interaction between the jet bundle construction and taking line products of line bundles.

\begin{prop}[Jet Bundle of a Line Product]\label{JetLineProduct}
Let $\lambda_1:L_1\to Q_1$, $\lambda_2:L_2\to Q_2$ two line bundles and denote their line product by $L_1\utimes L_2$, then there is a canonical diffeomorphic factor:
\begin{equation*}
\normalfont
\begin{tikzcd}
L_{\Jet^1 (L_1\utimes L_2)} \arrow[r, "W"] \arrow[d] & L_{\Jet^1 L_1} \utimes L_{\Jet^1 L_2} \arrow[d] \\
\Jet^1 (L_1\utimes L_2) \arrow[r, "w"'] & \Jet^1 L_1\dtimes \Jet^1 L_2
\end{tikzcd}
\end{equation*}
Furthermore, the factor $W$ is an isomorphism of Jacobi manifolds.
\end{prop}
\begin{proof}
Let us first construct the isomorphism factor $W$ explicitly. Recall from proposition \ref{DerLineProduct} that $\Der(L_1\utimes L_2)\cong \Der L_1 \boxplus \Der L_2$, the fact that the jet bundle is the ldual of the der bundle in the category of lvector bundles allows us to write:
\begin{equation*}
    \Jet^1 (L_1\utimes L_2)\cong(p_1^*\Der L_1\oplus p_2^*\Der L_2)^*\otimes (L_1\utimes L_2),
\end{equation*}
then, using the swapping isomorphism factor for the second term, we find the following isomorphism of vector bundles
\begin{equation*}
    t:\Jet^1 (L_1\utimes L_2)\to p_1^*\Jet^1 L_1\oplus p_2^*\Jet^1 L_2.
\end{equation*}
covering the identity map on $Q_1\dtimes Q_2$. Now, by the definition of line product, it is clear that we can define the following map
\begin{align*}
z: p_1^*\Jet^1 L_1\oplus p_2^*\Jet^1 L_2 & \to \Jet^1 L_1\dtimes \Jet^1L_2\\
(\alpha_{q_1}\oplus \beta_{q_2})_{B_{q_1q_2}} & \mapsto C_{\alpha_{q_1}\beta_{q_2}}
\end{align*}
where $C_{\alpha_{q_1}\beta_{q_2}}=B_{q_1q_2}$, which is well-defined since the line bundle over a jet bundle is defined as the pull-back line bundle from the base. Conversely, given a fibre-wise invertible map between pull-back line bundles we can project to a fibre-wise map between the base line bundles, denote this projection by $\overline{\pi}:\Jet^1 L_1 \dtimes \Jet^1 L_2 \to Q_1 \dtimes Q_2$, so there is an obvious inverse for the map above
\begin{equation*}
    z^{-1}(C_{\alpha_{q_1}\beta_{q_2}})=(\alpha_{q_1}\oplus \beta_{q_2})_{\overline{\pi}(C_{\alpha_{q_1}\beta_{q_2}})}.
\end{equation*}
We thus find the desired diffeomorphism
\begin{equation*}
    w:=z\circ t:\Jet^1 (L_1\utimes L_2) \to \Jet^1 L_1\dtimes \Jet^1L_2.
\end{equation*}
Let us write the line product commutative diagram for the line bundles over the jets as
\begin{equation*}
\begin{tikzcd}
L_{\Jet^1L_1} \arrow[d, "\mu_1"'] & L_{\Jet^1L_1}\utimes L_{\Jet^1L_2} \arrow[l,"R_1"']\arrow[d, "\mu_{12}"]\arrow[r,"R_2"] & L_{\Jet^1L_2} \arrow[d,"\mu_2"] \\
\Jet^1L_1 & \Jet^1L_1 \dtimes \Jet^1L_2 \arrow[l,"r_1"]\arrow[r,"r_2"'] & \Jet^1L_2
\end{tikzcd}
\end{equation*}
and denote the compositions $o_i:=\pi_i\circ r_i:\Jet^1L_1\dtimes \Jet^1L_2\to Q_i$. It then follows by construction that
\begin{equation*}
    L_{\Jet^1L_1}\utimes L_{\Jet^1L_2}=o_1^*L_1
\end{equation*}
and
\begin{equation*}
    L_{\Jet^1(L_1\utimes L_2)}\cong w^*o_1^*L_1
\end{equation*}
thus showing
\begin{equation*}
    L_{\Jet^1(L_1\utimes L_2)}\cong w^*(L_{\Jet^1L_1}\utimes L_{\Jet^1L_2}).
\end{equation*}
This is, of course, tantamount to there being a factor $W:L_{\Jet^1(L_1\utimes L_2)}\to L_{\Jet^1L_1}\utimes L_{\Jet^1L_2}$ covering the diffeomorphism $w$ that essentially acts as the fibre-wise identity on $L_1$. In order to show that this factor is, in fact, a Jacobi map we can write the pull-backs of (non-zero) brackets of spanning sections directly
\begin{align*}
    W^*\{R_i^*l_{a_i},R_i^*l_{b_i}\}_{12}&=W^*R_i^*\{l_{a_i},l_{b_i}\}_i=(R_i\circ W)^*l_{[a_i,b_i]}\\
    W^*\{R_i^*l_{a_i},R_i^*\pi_i^*u_i\}_{12}&=W^*R_i^*\{l_{a_i},\pi_i^*u_i\}_i=(R_i\circ W)^*\pi_i^*a_i[u_i]
\end{align*}
for $u_i\in\Sec{L_i}$, $a_i,b_i\in\Sec{\Der L_i}$, $i=1,2$, and where the definition of product Jacobi structure has been used to write the RHS expressions. Note that the construction of $W$ above is such that
\begin{equation*}
    (R_i\circ W)^*l_{a_i}=\overline{l}_{k_i(a_i)}\qquad (R_i\circ W)^*\pi_i^*u_i=\overline{\pi}^*P_i^*u_i
\end{equation*}
where $\overline{\pi}:\Jet^1(L_1\utimes L_2)\to Q_1\dtimes Q_2$ is the jet bundle projection for the line product, $k_i:\Sec{\Der L_i}\to \Sec{\Der (L_1\utimes L_2)}$ are the natural Lie algebra injections of derivations from proposition \ref{DerivationsLineProduct} and $\overline{l}:\Sec{\Der (L_1\utimes L_2)}\to L_{\Jet^1(L_1\utimes L_2)}$ is the inclusion of fibre-wise linear sections. A direct computation then gives the brackets of pull-backs of spanning sections
\begin{align*}
    \{W^*R_i^*l_{a_i},W^*R_i^*l_{b_i}\}&=\overline{l}_{[k_i(a_i),k_i(b_i)]}=(R_i\circ W)^*l_{[a_i,b_i]}\\
    \{W^*R_i^*l_{a_i},W^*R_i^*\pi_i^*u_i\}&=\overline{\pi}^*k_i(a_i)[P_i^*u_i]=\overline{\pi}^*P_i^*a_i[u_i]=(R_i\circ W)^*\pi_i^*a_i[u_i]
\end{align*}
which agree with the pull-backs of brackets above, thus completing the proof.
\end{proof}

Consider a factor between some pair of line bundles $B:L_1\to L_2$ covering a smooth map $b:Q_1\to Q_2$. Recall that the der map always gives a well-defined morphism of lvector bundles $\Der B:\Der L_1\to \Der L_2$, however factors covering general smooth maps cannot be dualized to give a well-defined map between the jet bundles. We can, nevertheless, define the \textbf{jet lift} of the factor $B$ as a subamanifold of the base product of jet bundles:
\begin{equation*}
    \Jet^1 B:=\{C_{\alpha_{q_1}\beta_{q_2}}\in \Jet^1 L_1\dtimes \Jet^1 L_2|\quad q_2=b(q_1),\quad (\Der_{q_1}B)^{*B_{q_1}}\beta_{q_2}=\alpha_{q_1}, \quad C_{\alpha_{q_1}\beta_{q_2}}=B_{q_1}\}
\end{equation*}
where the last condition is understood using the fact that the line bundles over the jet bundles are simply the pull-back bundles by the canonical projection to the base. When the factor $B$ covers a diffeomorphism $b$, the ldual of the der map $\Der B$ is a well-defined isomorphism of lvector bundles:
\begin{equation*}
    \Jet^1 B:(\alpha_q,l_q)\mapsto (\Jet^1 b(\alpha_q),B^{-1}_{q}l_q)
\end{equation*}
where
\begin{equation*}
    \Jet^1 b(\alpha_q)=(\Der_qB)^{*B_q}(\alpha_q)\in\Jet_{b^{-1}(q)}^1L_1.
\end{equation*}
The cotangent lift then induces the following commutative diagram of vector bundle morphisms
\begin{equation*}
    \begin{tikzcd}[sep=small]
     L_{\Jet^1L_2}\arrow[r,"\Jet^1B"]\arrow[d] & L_{\Jet^1L_1}\arrow[d] \\
     \Jet^1L_2\arrow[r,"\Jet^1 b"] \arrow[d] & \Jet^1L_1 \arrow[d] \\
     Q_2\arrow[r,"b^{-1}"] & Q_1
    \end{tikzcd}
\end{equation*}
and for any two diffeomorphic factors $B,F:L\to L$ and the identity factor $\Id_L:L\to L$ we find:
\begin{equation*}
    \Jet^1(B\circ F)=\Jet^1F\circ \Jet^1B, \qquad \Jet^1(\Id_L)=\Id_{L_{\Jet^1L}}.
\end{equation*}
In the case of a diffeomorphic factor, the jet lift is simply the lgraph of the induced factor $\LGrph{\Jet^1 B}$.\newline

We now show that the jet lift of a factor is, in fact, a Legendrian submanifold of the product Jacobi manifold, thus defining a Legendrian relation.

\begin{prop}[Jet Lift of a Factor]\label{JetLiftFactor}
Let $B:L_1\to L_2$ be a factor between two line bundles, then its jet lift
\begin{equation*}
\normalfont
    \Jet^1 B\subset \Jet^1 L_1 \dtimes \Jet^1 L_2
\end{equation*}
is a Legendrian submanifold of the Jacobi product structure $\normalfont L_{\Jet^1 L_1} \utimes \overline{L}_{\Jet^1 L_2}$.
\end{prop}
\begin{proof}
We first show that $\Jet^1B$ is a coisotropic submanifold by identifying a set of generating sections of its vanishing submodule $\Gamma_{\Jet^1B}\subseteq \Sec{L_{\Jet^1(L_1\utimes L_2)}}$ and showing that they indeed form a Lie subalgebra. Firstly, given $u_i\in\Sec{L_i}$, $a_i\in\Sec{\Der L_i}$, $i=1,2$, let us define the following sections of the line product of the canonical contact line bundles on the jet bundles $\Jet^1L_i$ following the notation introduced in proposition \ref{JetLineProduct}:
\begin{align*}
    l_{a_1a_2}&:=R_1^*l_{a_1}-R_2^*l_{a_2}\\
    \pi^*_{u_1u_2}&:= R_1^*\pi_1^*u_1-R_2^*\pi_2^*u_2.
\end{align*}
Consider a point on the jet lift $C_{\alpha_{q_1}\beta_{q_2}}\in \Jet^1 B$ so that $q_2=b(q_1)$, $(\Der_{q_1}B)^{*L}\beta_{q_2}=\alpha_{q_1}$ and $C_{\alpha_{q_1}\beta_{q_2}}=B_{q_1}$. Let us evaluate the defined sections on it:
\begin{align*}
l_{a_1a_2}(C_{\alpha_{q_1}\beta_{q_2}}) & = R_1|^{-1}_{C_{\alpha_{q_1}\beta_{q_2}}}l_{a_1}(\alpha_{q_1})-R_2|^{-1}_{C_{\alpha_{q_1}\beta_{q_2}}}l_{a_2}(\beta_{q_2})\\
& = l_{a_1}(\alpha_{q_1})-B^{-1}_{q_1}(l_{a_2}(\beta_{q_2}))\\
& = B^{-1}_{q_1}\beta_{q_2}(\Der_{q_1}B(a_1|_{q_1})-a_2|_{q_2})\\
& = B^{-1}_{q_1}\beta_{q_2}(\Der B \circ a_1 - a_2\circ b)(q_1)
\end{align*}
and
\begin{align*}
\pi^*_{u_1u_2} & = R_1|^{-1}_{C_{\alpha_{q_1}\beta_{q_2}}}u_1(q_1)-R_2|^{-1}_{C_{\alpha_{q_1}\beta_{q_2}}}u_2(q_2)\\
& = u_1(q_1)-B_{q_1}^{-1}u_2(q_2)\\
& = u_1(q_1)-B_{q_1}^{-1}u_2(b(q_1))\\
& = (u_1-B^*u_2)(q_2).
\end{align*}
It is then obvious that these spanning sections vanish on the jet lift $\Jet^1 B$ iff the derivations are $B$-related and the sections are mapped by the pull-back $B^*$, i.e.
\begin{align*}
    \Gamma_{\Jet^1 B}\ni l_{a_1a_2} &\Leftrightarrow   a_1\sim_B a_2\\
    \Gamma_{\Jet^1 B}\ni \pi^*_{u_1u_2}  &\Leftrightarrow  u_1=B^*u_2.
\end{align*}
These are the generating vanishing sections so it will suffice to check that evaluations of brackets among these on an arbitrary point of the jet lift $C_{\alpha_{q_1}\beta_{q_2}}\in \Jet^1 B$ vanish. From the defining conditions of a the linear Jacobi we see that $\{\pi^*_{u_1u_2},\pi^*_{u_1'u_2'}\}=0$, so we are left with the two other possible brackets. For the bracket of fibre-wise linear sections we compute explicitly using the defining properties of the product Jacobi bracket:
\begin{equation*}
    \{l_{a_1a_2} ,l_{a_1'a_2'} \}= R^*_1\{l_{a_1},l_{a_1'}\}_1-R^*_2\{l_{a_2},l_{a_2'}\}_2=R^*_1l_{[a_1,a_1']}-R^*_2l_{[a_2,a_2']}
\end{equation*}
for $a_i,a_i'\in\Sec{\Der L_i}$, $i=1,2$. Then if $l_{a_1a_2}$ and $l_{a_1'a_2'}$ are vanishing sections, the derivations are $B$-related and, by virtue of proposition \ref{DerFunctor}, where $\Der B$ is shown to be a Lie algebroid morphism, the brackets are also $B$-related $[a_1,a_1']\sim_B [a_2,a_2']$ making the expression above into a vanishing section. Only the cross bracket $\{l_{a_1a_2} ,\pi^*_{u_1u_2} \}= R^*_1\pi_1^*a_1[u_1]_1-R^*_2\pi_2^*a_2[u_2]$ remains, for which we evaluate on a point of the jet lift and show it vanishes by directly computing using $u_1=B^*u_2$ and $a_1\sim_B a_2$:
\begin{align*}
\{l_{a_1a_2} ,\pi^*_{u_1u_2} \}(C_{\alpha_{q_1}\beta_{q_2}}) & = a_1|_{q_1}(u_1)-B_{q_1}^{-1}a_2|_{q_2}(u_2)\\
& = a_1|_{q_1}(u_1)-B_{q_1}^{-1}\Der_{q_1}B(a_1|_{q_1})(u_2)\\
& = a_1|_{q_1}(B^*u_2)-B_{q_1}^{-1}B_{q_1}a_1|_{q_1}(B^*u_2)\\
& = a_1|_{q_1}(B^*u_2)-a_1|_{q_1}(B^*u_2)\\
& = 0.
\end{align*}
To show that $\Jet^1 B$ is maximally isotropic we simply do a dimension count. By definition, it is clear that $\dimm \Jet^1 B=\dimm Q_1\dtimes Q_2$, but it follows from proposition \ref{JetLineProduct} that $\dimm \Jet^1 L_1\dtimes \Jet^1 L_2 = 2(\dimm Q_1\dtimes Q_2)+1$, then clearly $\Jet^1 B\subset \Jet^1 L_1\dtimes \Jet^1 L_2$ is maximally isotropic and thus Legendrian.
\end{proof}

With this last proposition at hand, we can now see the jet bundle construction as a functor from the category of line bundles into the category of contact manifolds.

\begin{prop}[The Jet Functor] \label{JetFunctor}
The assignment of jet bundles to line bundles is a contravariant functor
\begin{equation*}
\normalfont
    \Jet^1 : \Line_\Man \to \Cont_\Man.
\end{equation*}
\end{prop}
\begin{proof}
A factor $B:L_1\to L_2$ gives a Legendrian relation $\Jet^1B:\Jet^1L_2\dashrightarrow \Jet^1L_1$ in virtue of proposition \ref{JetLiftFactor}, then it only remains to check functoriality. It is obvious by definition that the identity factor $\Id_L:L\to L$ gives the diagonal relation $\text{Ldiag}(L_{\Jet^1L})\subset \Jet^1 L\dtimes \Jet^1L$, where the diagonal is the natural subset of fibre-wise identity maps in a line product $L\utimes L$. Consider now two factors $B:L_1\to L_2$ and $B':L_2\to L_3$ covering the smooth maps $b:Q_1\to Q_2$ and $b':Q_2\to Q_3$. By definition of composition of relations, we find
\begin{align*}
    \Jet^1 B\circ \Jet^1B':= &\{C_{\alpha_{q_1}\gamma_{q_3}}\in \Jet^1 L_1\dtimes \Jet^1 L_3| \\ & q_3=b'(b(q_1)), \, (\Der_{q_1}B)^{*B_{q_1}}(\Der_{b(q_1)}B')^{*B'_{b(q_1)}}\gamma_{q_3}=\alpha_{q_1}, \, C_{\alpha_{q_1}\gamma_{q_3}}=B_{q_1}\circ B_{b(q_1)}'\}
\end{align*}
Then, it follows from the fact that lduality is a contravariant autofunctor of lvector spaces that the jet lift of factors is contravariant with respect to composition of relations
\begin{equation*}
    \Jet^1 B\circ \Jet^1B'=\Jet^1(B'\circ B).
\end{equation*}
\end{proof}

There are two alternative ways to regard the jet bundle of the line bundle induced on a submanifold of the base. On the one had, we could take the line bundle as a an embedded subbundle and construct its jet bundle from the ambient jet bundle. On the other, we could simply regard the restricted line bundle as an intrinsic line bundle and canonically construct its jet bundle. We now prove that these two constructions are equivalent.

\begin{prop}[Canonical Coisotropic Reduction in Jet Bundles]\label{JetCoisotropicReduction}
Let $i:S\hookrightarrow Q$ be a submanifold of a line bundle $L$, then the restriction of the ambient jet bundle to the submanifold $\normalfont (\Jet^1 L)|_S$ is a coisotropic submanifold with respect to the canonical contact structure on $\normalfont \Jet^1 L$. Furthermore, there is a submersion factor covering the surjective submersion $\normalfont z:(\Jet^1 L)|_S\twoheadrightarrow \Jet^1 L_S$ given by the fibre-wise quotient:
\begin{equation*}
\normalfont
    \Jet^1_qL/(\Der L_S)^{0L}\cong \Jet^1_q L_S\qquad q\in S
\end{equation*}
so that the canonical contact structure on $\normalfont \Jet^1 L$ Jacobi reduces to the canonical contact structure on $\normalfont \Jet^1 L_S$.
\end{prop}
\begin{proof}
Let us first prove that the submersion $ z:(\Jet^1 L)|_S\twoheadrightarrow \Jet^1 L_S$ fits in a Jacobi reduction scheme. It is a direct implication of proposition \ref{DerBundleLineSubmanifold} and the basic properties of lvector spaces applied fibre-wise, that we have the following diagram of line bundle morphisms
\begin{equation*}
\begin{tikzcd}[sep=tiny]
L_{\Jet^1L|_{S}} \arrow[rr,"\iota"] \arrow[dd,"\zeta"'] \arrow[dr]& & L_{\Jet^1L} \arrow[dr] & \\
& \Jet^1L|_S \arrow[rr,"i"', hook] \arrow[dd, "z",twoheadrightarrow] & & \Jet^1L \\
L_{\Jet^1L_S} \arrow[dr] & & & \\
 & \Jet^1L_S &  & 
\end{tikzcd}
\end{equation*}
where $\iota$ denotes, abusing notation, the embedding factor induced by the submanifold $i:S\hookrightarrow Q$ and $\zeta$ is defined as the fibre-wise identity of the pull-back line bundles covering the point-wise linear submersion
\begin{equation*}
    z_q:\Jet^1_q L\twoheadrightarrow \Jet^1_qL/(\Der L_S)^{0L}\cong \Jet^1_q L_S, \qquad q\in S.
\end{equation*}
Following proposition \ref{DerivationsLineSubmanifold}, we find natural isomorphisms $\Sec{L_S}\cong\Sec{L}/\Gamma_S$ and $\Dr{L_S}\cong\text{Der}_S(L)/\text{Der}_{0S}(L)$, and denoting the natural inclusions of spanning sections on $L_{\Jet^1 L_S}$ by $\overline{l}$ and $\overline{\pi}$, we can write the fibre-wise linear Jacobi structure on $\Jet^1L_S$ equivalently as
\begin{align*}
    \{\overline{l}_{\overline{a}},\overline{l}_{\overline{b}}\}_S & =  \overline{l}_{\overline{[a,b]}}\\
    \{\overline{l}_{\overline{a}},\overline{\pi}^*\overline{u}\}_S & =  \overline{\pi}^*\overline{a[u]}\\
    \{\overline{\pi}^*\overline{u},\overline{\pi}^*\overline{v}\}_S & = 0
\end{align*}
with $\overline{u},\overline{v}\in\Sec{L}/\Gamma_S$ and $\overline{a},\overline{b}\in\text{Der}_S(L)/\text{Der}_{0S}(L)$, which is well-defined precisely from the description of derivations as a subquotient of Lie algebras. The submersion factor $\zeta:L_{\Jet^1L|_{S}}\to L_{\Jet^1L_S}$ covering the quotient map $z:\Jet^1L|_{S}\to \Jet^1L_S$ has been defined such that it is the point-wise counterpart to the isomorphisms used above to rewrite the linear Jacobi bracket. Pull-backs via these factors satisfys the following identities by construction
\begin{equation*}
    \zeta^*\overline{\pi}^*\overline{u} =\iota^*\pi^*u \qquad \zeta^*\overline{l}_{\overline{a}} = \iota^* l_{a}
\end{equation*}
for all $u\in\Sec{L}$ and $a\in\text{Der}_S(L)$. This now clearly implies the reduction condition for all spanning sections
\begin{align*}
    \zeta^*\{\overline{l}_{\overline{a}},\overline{l}_{\overline{b}}\}_S & = \iota^* l_{[a,b]}=\iota^*\{l_a,l_b\}_L\\
    \zeta^*\{\overline{l}_{\overline{a}},\overline{\pi}^*\overline{u}\}_S & = \iota^* \pi^*a[u]=\iota^* \{l_a,\pi^*u\}_L\\
    \zeta^*\{\overline{\pi}^*\overline{u},\overline{\pi}^*\overline{v}\}_S & = 0 = \iota^* \{\pi^*u,\pi^*v\}_L
\end{align*}
thus showing that the linear Jacobi $(L_{\Jet^1L},\{,\}_L)$ reduces to $(L_{\Jet^1L_S},\{,\}_S)$. Lastly, it is easy to see that the vanishing sections of $\Jet^1L|_S$ seen as a submanifold of the jet bundle are precisely those of the form $l_{\text{Der}_{0S}(L)}$ and $\pi^*\Gamma_S$. It follows again from proposition \ref{DerivationsLineSubmanifold} that these form a Lie subalgebra of the linear Jacobi structure $(L_{\Jet^1L},\{,\}_L)$, thus explicitly showing that $\Jet^1L|_S\subseteq \Jet^1L$ is a coisotropic submanifold.
\end{proof}

Consider now a line bundle action $G\Acts L$. Recall that, in the case of a free and proper action, the orbit space is canonically a line bundle, denoted by $L/G$, and that there is a natural submersion factor $\sigma:L\to L/G$. The following proposition shows that the canonical contact structures associated to these two line bundles are related by a Hamiltonian reduction scheme.

\begin{prop}[Canonical Hamiltonian Reduction in Jet Bundles]\label{JetHamiltonianReduction}
Let $\Phi:G\times L\to L$ be a free and proper line bundle action of a connected Lie group $G$ on $L$, then the canonical contact structure on $\normalfont \Jet^1 L$ Jacobi reduces to the canonical contact structure on $\normalfont \Jet^1 (L/G)$. This reduction is, in fact, Hamiltonian: the jet lift of the line bundle action $\normalfont G\Acts L_{\Jet^1 L}$ preserves the canonical contact structure and has a natural comoment map given by
\begin{align*}
\overline{\mu}: \mathfrak{g} &\normalfont \to \Sec{L_{\Jet^1 L}}\\
\xi & \mapsto l_{\Psi(\xi)},
\end{align*}
where $\normalfont \Psi:\mathfrak{g}\to \Dr{L}$ is the infinitesimal line bundle action and $\normalfont l:\Dr{L}\to \Sec{L_{\Jet^1 L}}$ is the natural inclusion of derivations as fibre-wise linear sections on the jet bundle.
\end{prop}
\begin{proof}
The spanning sections $l_a,\pi^*u\in\Sec{L_{\Jet^1L}}$ transform under pull-back by a jet lift of a diffeomorphic factor according to the following expressions
\begin{equation*}
    \Jet^1B^*l_a=l_{B_*a}, \qquad \Jet^1B^*\pi^*u=\pi^*(B^{-1})^*u,
\end{equation*}
then it follows that the jet lift of a diffeomorphic factor $\Jet^1B$ is indeed a Jacobi map of the canonical contact structure on the jet bundle:
\begin{align*}
    \Jet^1B^*\{l_a,l_b\}_L & = \Jet^1B^*l_{[a,b]}=l_{[B_*a,B_*b]}=\{l_{B_*a},l_{B_*b}\}_L=\{\Jet^1B^*l_a,\Jet^1B^*l_b\}_L \\
    \Jet^1B^*\{l_a,\pi^*u\}_L & = \Jet^1B^*\pi^*a[u]=\pi^*(B_*a)[(B^{-1})^*u]=\{l_{B_*a},\pi^*(B^{-1})^*u\}_L=\{ \Jet^1B^*l_a, \Jet^1B^* \pi^*u\}_L \\
    \Jet^1B^*\{\pi^*u,\pi^*v\}_L & = 0 = \{\pi^*(B^{-1})^*u,\pi^*(B^{-1})^*v\}_L=\{\Jet^1B^*\pi^*u,\Jet^1B^*\pi^*v\}_L.
\end{align*}
The jet lift of the group action $G\Acts \Jet^1 L$ is defined by the jet lifts of the diffeomorphic factors corresponding to each group element
\begin{equation*}
    (\Jet^1 \Phi)_g:=\Jet^1\Phi_g,
\end{equation*}
which, in light of the above results for general jet lifts of diffeomorphic factors, is readily checked to be a group action that acts via Jacobi maps. This is a Hamiltonian action with comoment map simply given by $\overline{\mu}:=l\circ \Psi :\mathfrak{g}\to \Sec{L_{\Jet^1L}}$. Observe that the zero locus of the moment map is naturally identified with the annihilator of the subspace of derivations spanned by the infinitesimal generators regarded as a subbundle of the jet bundle
\begin{equation*}
    \mu^{-1}(0)=\Psi(\mathfrak{g})^{0L}\subseteq \Jet^1L.
\end{equation*}
Note that $G$-equivariance of the infinitesimal action $\Psi$ implies that the jet lifted action of $G$ restricts to a $G$-action on $\Psi(\mathfrak{g})^{0L}$, indeed we check for any $j_q^1u\in\Psi(\mathfrak{g})^{0L}$ and $\xi\in\mathfrak{g}$
\begin{equation*}
    \Jet^1\phi_g(j_q^1 u)(\Psi(\xi))=\Psi(\xi)[\Phi^*_gu]=(\Phi_g)_{b^{-1}(q)}\Der_q\Phi_g\Psi(\xi)=(\Phi_g)_{b^{-1}(q)}j_q^1u(\Psi(\text{Ad}_g(\xi))=0.
\end{equation*}
Using Proposition \ref{DerBundleGroupAction} and simple linear algebra of lvector bundles we find the following point-wise isomorphism
\begin{equation*}
    \Jet^1_{[q]}(L/G):=(\Der_{[q]}(L/G))^{*L/G}\cong(\Der_q/\Psi(\mathfrak{g})_q)^{*L_q}\cong \Psi(\mathfrak{g})_q^{0L_q}.
\end{equation*}
This allows us to write the following factor reduction diagram
\begin{equation*}
\begin{tikzcd}[sep=tiny]
L_{\Psi(\mathfrak{g})^{0L}} \arrow[rr,"\iota"] \arrow[dd,"\zeta"'] \arrow[dr]& & L_{\Jet^1L} \arrow[dr] & \\
& \Psi(\mathfrak{g})^{0L} \arrow[rr,"i"', hook] \arrow[dd, "z",twoheadrightarrow] & & \Jet^1L \\
L_{\Jet^1(L/G)} \arrow[dr] & & & \\
 & \Jet^1(L/G) &  & 
\end{tikzcd}
\end{equation*}
where $\iota$ is the embedding factor for the annihilator $\Psi(\mathfrak{g})^{0L}$ seen as a subbundle (submanifold) of the jet bundle and $\zeta$ is the submersion factor induced by the jet lifted action restricted $\Psi(\mathfrak{g})^{0L}$, which is clearly free and proper. Since $G$ is connected, recall that the sections and derivations of the quotient line bundle can be equivalently regarded as
\begin{equation*}
    \Sec{L/G}\cong \Sec{L}^\mathfrak{g}, \qquad \Dr{L/G}\cong \Dr{L}^{\mathfrak{g}}/\Psi(\mathfrak{g}),
\end{equation*}
thus the linear Jacobi structure determined by the spanning sections $\overline{l}_{\overline{a}},\overline{\pi}^*\overline{u}\in\Sec{L_{\Jet^1(L/G)}}$ can be fully characterised under these isomorphisms. The factors constructed above are such that we have the following explicit treatment of extensions of spanning sections
\begin{align*}
    \zeta^*\overline{\pi}^*\overline{u}=\iota^*\pi^*u \quad &\Leftrightarrow \quad u\in\Sec{L}^\mathfrak{g}\\
    \zeta^*\overline{l}_{\overline{a}}=\iota^*l_a \quad &\Leftrightarrow \quad \overline{a}=a+\Psi(\mathfrak{g}), a\in\Dr{L}^\mathfrak{g}
\end{align*}
The reduction condition is now easily checked
\begin{align*}
    \zeta^*\{\overline{l}_{\overline{a}},\overline{l}_{\overline{b}}\}_{L/G} & = \iota^* l_{[a,b]}=\iota^*\{l_a,l_b\}_L\\
    \zeta^*\{\overline{l}_{\overline{a}},\overline{\pi}^*\overline{u}\}_{L/G} & = \iota^* \pi^*a[u]=\iota^* \{l_a,\pi^*u\}_L\\
    \zeta^*\{\overline{\pi}^*\overline{u},\overline{\pi}^*\overline{v}\}_{L/G} & = 0 = \iota^* \{\pi^*u,\pi^*v\}_L
\end{align*}
for all spanning sections $\overline{l}_{\overline{a}},\overline{l}_{\overline{b}},\overline{\pi}^*\overline{u},\overline{\pi}^*\overline{v}\in\Sec{L_{\Jet^1(L/G)}}$ and extensions $l_a,l_b,\pi^*u,\pi^*v\in\Sec{L_{\Jet^1L}}$, thus concluding the proof.
\end{proof}

\subsection{The Unit-Free Hamiltonian Functor} \label{UnitFreeCanonicalHamiltonian}

We are now in the position to define \textbf{the category of canonical contact phase spaces} as the image of the category of unit-free configuration spaces under the jet functor, $\Jet^1(\Line_\Man)$, which, together with the same notions of observables $\text{Obs}=\Gamma: \Cont_\Man\to \textsf{RMod}$, dynamics $\text{Dyn}=\Der:\Cont_\Man\to \textsf{Jacb}_\Man$ and evolution $\text{Evl}=\text{Der}:\textsf{RMod}\to \textsf{Jacb}_\Man$, forms a theory of phase spaces. Clearly, the presence of the canonical contact structure on jet bundles ensures that $\Cont_\Man$ is, furthermore, a Hamiltonian theory of phase spaces, with Hamiltonian maps given by the Hamiltonian derivation of the Jacobi structures:
\begin{equation*}
    \eta_L:=\Pi_L^\sharp\circ j^1=\text{ad}_{\{,\}_L}:\Obs{\Jet^1 L}\to \Dyn{\Jet^1 L}.
\end{equation*}

The results proved in Section \ref{CanonicalContact} can be encapsulated in the \textbf{Hamiltonian functor} for unit-free configuration spaces:

\begin{center}
\begin{tabular}{ c c c }
Unit-Free Configuration Spaces & Hamiltonian Functor  & Unit-Free Phase Spaces \\
\hline
 $\Line_\Man$ & $\begin{tikzcd}\phantom{A} \arrow[r,"\Jet^1"] & \phantom{B} \end{tikzcd}$ & $\Cont_\Man$ \\ 
 $L_Q$ & $\begin{tikzcd} \phantom{Q} \arrow[r, "\Jet^1 ",mapsto] & \phantom{Q} \end{tikzcd}$ & $(\Jet^1 L_Q,\theta_{L_Q})$ \\ 
 $\Obs{L_Q}$ & $\begin{tikzcd} \phantom{Q} \arrow[r,hookrightarrow, "\pi^*"] & \phantom{Q}  \end{tikzcd}$ & $\Obs{L_{\Jet^1 L_Q}}$ \\
 $\Dyn{L_Q}$ & $\begin{tikzcd} \phantom{Q} \arrow[r,hookrightarrow, "l"] & \phantom{Q}  \end{tikzcd}$ & $\Obs{L_{\Jet^1 L_Q}}$ \\
 $L_1\utimes L_2$ & $\begin{tikzcd} \phantom{Q} \arrow[r,mapsto, "\Jet^1 "] & \phantom{Q}  \end{tikzcd}$ & $\Jet^1 L_1\dtimes \Jet^1 L_2$ \\
 $B:L_1\to L_2$ & $\begin{tikzcd} \phantom{Q} \arrow[r,mapsto, "\Jet^1 "] & \phantom{Q}  \end{tikzcd}$ & $\Jet^1B:\Jet^1L_2\dashrightarrow \Jet^1L_1$\\
 $\iota:L_S\hookrightarrow L_Q$ & $\begin{tikzcd} \phantom{Q} \arrow[r,mapsto, "\Jet^1"] & \phantom{Q}  \end{tikzcd}$ & $(\Jet^1 L_Q)|_S\subset \Jet^1 L_Q$ coisotropic \\
 $G\Acts L$ & $\begin{tikzcd} \phantom{Q} \arrow[r,mapsto, "\Jet^1"] & \phantom{Q}  \end{tikzcd}$ & $G\Acts \Jet^1 L$ Hamiltonian \\
\end{tabular}
\end{center}
This correspondence is clearly categorically analogous to the Hamiltonian functor for ordinary configuration spaces of Section \ref{OrdinaryHamiltonian}, hence, canonical contact manifolds appear as valid unit-free generalizations of conventional \emph{unit-less} phase spaces. Furthermore, a similar diagram connecting the category of unit-free configuration spaces and the category of canonical contact phase spaces as theories of phase spaces is given by the jet functor:
\begin{equation*}
\begin{tikzcd}[row sep=small]
 & \textsf{RMod} \arrow[dd,"\text{Der}"']& & & \textsf{LocLieAlg} \arrow[dd,"\text{Der}"] \\
 & & \Line_\Man \arrow[ul, "\Gamma"']\arrow[dl,"\Der"]\arrow[r,"\Jet^1"] & \Cont_\Man \arrow[ur,"\Gamma"] \arrow[dr,"\Der"'] & \\
 & \textsf{Jacb}_\Man & & & \textsf{Jacb}_\Man 
\end{tikzcd}
\end{equation*}

Similarly to ordinary symplectic phase spaces, the notion of \textbf{unit-free energy} as a choice of observable $h\in\Sec{L_{\Jet^1 L}}$ to determine the dynamics of a physical system appears naturally in the category of canonical contact phase spaces. The fact that a Jacobi structure is a Lie algebra allows for the interpretation of $h$ as a fundamental conserved quantity of the evolution of the system, $\eta_L(h)[h]=\{h,h\}_L=0$, and the construction of the categorical product of line bundles gives a natural additivity property for the unit-free energies of two physical systems under the Hamiltonian functor
\begin{center}
\begin{tabular}{ c c c }
Unit-Free Configuration Spaces & Hamiltonian Functor  & Unit-Free Phase Spaces \\
\hline
 $(L_1,h_1), (L_2,h_2)$ & $\begin{tikzcd} \phantom{Q} \arrow[r,mapsto, "\Jet^1 "] & \phantom{Q}  \end{tikzcd}$ & $(L_{\Jet^1 L_1}\utimes L_{\Jet^1 L_2},P_1^*h_1 + P_2^*h_2)$ 
\end{tabular}
\end{center}

This additivity property of energy can be motivated from the natural unit-free generalization of the notion of Riemannian metric. A \textbf{lmetric} on a unit-free configuration space $\lambda:L\to Q$ is a non-degenerate symmetric bilinear form $\gamma:\Der L\odot \Der L\to L$. The presence of a lmetric induces a musical isomorphism
\begin{equation*}
        \begin{tikzcd}
        \Der L \arrow[r, "\gamma^\flat",yshift=0.7ex] & \Jet^1 L \arrow[l,"\gamma^\sharp",yshift=-0.7ex]
        \end{tikzcd},
\end{equation*}
which, in turn, allows for the definition of \textbf{unit-free kinetic energy} as a quadratic section $K_\gamma\in \Sec{L_{\Jet^1L}}$ via
\begin{equation*}
    K_\gamma(\alpha):=\gamma(\gamma^\sharp(\alpha),\gamma^\sharp(\alpha)).
\end{equation*}
With a choice of a section of the line bundle $v\in \Sec{L}$, that we identify as the \textbf{unit-free potential}, we define the \textbf{unit-free Newtonian energy} as
\begin{equation*}
    \Sec{L_{\Jet^1L}}\ni E_{\gamma, v}:= K_\gamma +\pi^*v.
\end{equation*}
It follows from proposition \ref{DerLineProduct} that, given two line bundles with choices of unit-free Newtonian energy $(L_1,E_{\gamma_1, v_1})$ and $(L_2,E_{\gamma_2, v_2})$, their line product carries a natural choice of unit-free Newtonian energy
\begin{equation*}
    (L_1\utimes L_2,E_{\gamma_1, v_1}+E_{\gamma_2, v_2})
\end{equation*}
where
\begin{equation*}
    E_{\gamma_1, v_1}+E_{\gamma_2, v_2}:=P_1^*E_{\gamma_1, v_1}+P_2^*E_{\gamma_2, v_2}=E_{p_1^*\gamma_1 \oplus p_2^*\gamma_2, p_1^*v_1+p_2^*v_2}.
\end{equation*}

\section{Conclusion: `Unit-Free' as a stepping stone towards `Dimensioned'} \label{Conclusion}

We hope the reader agrees that the introduction of the unit-free approach for line bundle geometry streamlines the presentation of Jacobi geometry and enables the clear interpretation as a direct generalisation of Poisson geometry. This seems to suggest that Jacobi geometry is yet another example of the pattern that has already been identified in the literature where the \emph{Poisson flavour} reappears when one tries to generalise explicitly Poisson objects, e.g. Lie algebroids can be understood as a generalisation of the Lie bracket structures that appear on cotangent bundles of Poisson manifolds but it can be shown that generic Lie algebroids are, in turn, in one-to-one correspondence with linear Poisson manifolds.\newline

Our approach clearly suggests an application of Jacobi geometry to physics in the form of unit-free phase spaces. This can be regarded as a first attempt at a rigorous and systematic treatment of physical quantities and units of measurement in geometric mechanics. The arguments presented in Section \ref{UnitFreeHamiltonian} demonstrate that, at least in categorical and structural terms, our formalism of unit-free Hamiltonian mechanics seems to be on the right track. This success is particularly apparent when considering the obvious similarity between the Hamiltonian functor of Section \ref{OrdinaryHamiltonian} and the unit-free Hamiltonian functor of Section \ref{UnitFreeHamiltonian}.\newline

Despite the promising results provided by the unit-free formalism, any reader familiar with geometric mechanics and the usual mathematical framework of physical theories will notice two major shortcomings of the theory of unit-free Hamiltonian phase spaces: unit-free dynamics don't seem to recover ordinary dynamics and unit-free observables lack a commutative product.\newline

Firstly, the problem of unit-free dynamics appears when one takes a the jet bundle of a line bundle $\pi: \Jet^1 L\to Q$ as the notion of phase space and tries to recover equations of motion on the base configuration space. We can choose adapted coordinates $(q^i,p_i,z)$ in a trivialising neighbourhood $U\subset Q$ and choose a unit-free Hamiltonian as a function $h\in\Cin{\Real^{2n+1}}$. The dynamics induced by this choice are what we could call \textbf{unit-free Hamilton's equations}:
\begin{align*}
    \frac{dq^i}{dt} &=\frac{\partial h}{\partial p_i}\\
    \frac{dp_i}{dt} &=-\frac{\partial h}{\partial q^i}-p_i\frac{\partial h}{\partial z} \\
    \frac{dz}{dt} &=p_i\frac{\partial h}{\partial p_i}-h.
\end{align*}
These are manifestly different from ordinary Hamilton's equations and thus will produce different equations of motion on $Q$ in general. This apparent problem is resolved when the adequate interpretation for the $z$ coordinate is adopted. If $(q^i,p_i)$ can be regarded as ordinary positions and momenta, then $z$ represents the freedom of choice of unit of observable and thus $z(t)$ keeps track of the evolution of such freedom: if a unit $u_0$ is chosen is chosen at an initial moment $t=0$, then that choice of unit propagates to later times as $u(t)=z(t) u_0$. In particular, if a unit is fixed on the entire coordinate patch $U$, the local choice of Hamiltonian does not have any dependence on $z$ (since no change-of-unit information needs to be encoded) so $\tfrac{\partial h}{\partial z}=0$ and the above equations become the ordinary Hamilton's equations. In a general unit-free phase space, where only a general Jacobi structure is assumed, Hamiltonian dynamics of ordinary observables are recovered via local functions that are invariant with respect to a unit.\newline

Secondly, the issue of the algebraic structure of unit-free observables turns out to have much deeper ramifications and its resolution requires, in fact, an overhaul of the entire formalism. Some preliminary attempts at a solution to this problem can be found in \cite[Ch. 7]{zapata2019landscape} where dimensioned structures and the potential functor of line bundles are introduced. A complete treatment of this topic will appear in future work by the author but we give some general comments here for the reader's convenience.\newline

Unit-free observables are the sections of a generically non-trivial line bundle $\Sec{L_P}$ while ordinary observables are the real-valued functions of some smooth manifold $\Cin{P}$. Algebraically, unit-free observables are a projective module over the commutative algebra of functions while ordinary observables are a commutative algebra themselves. In the conventional treatment of mechanics if two physical quantities were represented by observables $f\in\Cin{P}$ and $g\in\Cin{P}$ then the observable $fg\in\Cin{P}$ represented a derived quantity with physical dimension equal to the product of physical dimensions of $f$ and $g$, e.g. angular momentum can be expressed as a product of position coordinates and momentum coordinates $l=xp$. It is clear that unit-free observables do not allow for this kind of construction.\newline

This issue is not just a matter of a failure to account for the structure of physical quantities faithfully but it also has important mathematical consequences that, to the author's knowledge, have not been explored in the existing literature. This was noted, for instance, in section \ref{JacobiAlgebroids} when the direct unit-free analogue of Lie algebroids could not be clearly defined since no natural Leibniz formula could be established for a product of sections with derivations. It is well known that the product of Poisson manifolds appears algebraically as the tensor product of the Poisson algebras of functions. Another issue that a reader familiar with Poisson geometry may have noticed is that no mention of the tensor product of Jacobi brackets was made when discussing the product of Jacobi manifolds.\newline

All this points towards a clear course of action: to consider all the tensor powers of a line bundle collectively, which are also line bundles for obvious dimensional reasons, as a generalisation of unit-free manifold. This breaks with the limitations imposed by the L-rooted approach discussed in Section \ref{Lines} and endows sections with the tensor product as a commutative product (see \cite[Prop. 7.5.1]{zapata2019landscape}). Furthermore, a Jacobi structure on the base line bundle then appears as a Lie bracket on the tensor powers of the line bundle satisfying the ordinary Leibniz identity with respect to the commutative product (see \cite[Th. 7.6.1]{zapata2019landscape}).

\section*{Acknowledgements}

I would like to thank Luca Vitagliano for his useful comments and his student Jonas Schnitzer, whom I met during my visit to IMPA, Rio de Janeiro, in 2017, for introducing me to the topic of Jacobi geometry and showing me the line bundle approach that has been so heavily used in this work. I would like to thank Henrique Bursztyn for inviting me to IMPA. I am also indebted to the organisers of the Rethinking Workshop 2017 for suggesting that I explore the topic of the formal treatment of units of measurement in physics. Lastly, I would like to thank Jos\'e Figueroa-O'Farrill for all his support during my PhD and the many fruitful conversations, at times, merrily accompanied by orxata or Valor's chocolate.

\printbibliography

\appendix

\section{Supplementary identities to be used in the proof of Proposition \ref{ProductJacobi}} \label{SymbolSquiggleProductJacobi}

For $i=1,2$, let two line bundles $\lambda_i:L_i\to M_i$ with Jacobi structures $(\Sec{L_i},\{,\}_i)$. Let $f_i,g_i\in\Cin{M}$ be smooth functions, $s_i\in \Sec{L_i}$ sections and $a,a'\in\Sec{L_1^\bullet}$, $b,b'\in\Sec{L_2^\bullet}$ local non-vanishing sections. Then, the line product construction $L_1\utimes L_2$ allows us to write the following identities
\begin{equation*}
    p_1^*f_1\tfrac{s_1}{b}=\tfrac{f_1\cdot s_1}{b} \qquad P_1^*s_1=\tfrac{s_1}{b}P^*_2b \qquad P_2^*s_2=\tfrac{s_2}{a}P^*_1a \qquad p_2^*f_2\tfrac{s_2}{a}=\tfrac{f_2\cdot s_2}{a}
\end{equation*}
The definition of the bracket is done via the following equations
\begin{equation*}
    \{P_1^*s_1,P_1^*s_1'\}_{12}:=P_1^*\{s_1,s_1'\}_1 \qquad \{P_2^*s_2,P_2^*s_2'\}_{12}:=P_2^*\{s_2,s_2'\}_2\qquad \{P_1^*s_1,P_2^*s_2\}_{12}:=0
\end{equation*}
which, in turn, imply
\begin{align*}
    X^{12}_{P_1^*s_1}[p_1^*f_1]&=p_1^*X^1_{s_1}[f] & X^{12}_{P_2^*s_2}[p_2^*f_2]&=p_2^*X^2_{s_2}[f]\\
    X^{12}_{P_1^*s_1}[p_2^*f_2]&=0 & X^{12}_{P_2^*s_2}[p_1^*f_1]&=0
\end{align*}
The definition of the symbol and squiggle is done via the following equations
\begin{align*}
    X^{12}_{P_1^*s_1}[\tfrac{a}{b}]&=\tfrac{\{s_1,a\}_1}{b}   & X^{12}_{P_2^*s_2}[\tfrac{b}{a}]&=\tfrac{\{s_2,b\}_2}{a}\\
    \Lambda^{12}(d\tfrac{a}{b}\otimes P_1^*s_1)[\tfrac{a'}{b'}]&=\tfrac{\{a,a'\}_1}{b}\tfrac{s_1}{b'} &\Lambda^{12}(d\tfrac{b}{a}\otimes P_2^*s_2)[\tfrac{b'}{a'}]&=\tfrac{\{b,b'\}_2}{a}\tfrac{s_2}{a'}
\end{align*}
which, in turn, imply
\begin{align*}
    \Lambda^{12}( dp_1^*f_1 \otimes P_1^*s_1 )[p_1^*g_1]&=p_1^*\Lambda^1(df_1\otimes s_1)[g_1] & \Lambda^{12}( dp_2^*f_2 \otimes P_2^*s_2 )[p_2^*g_2]&=p_2^*\Lambda^1(df_2\otimes s_2)[g_2]\\
    \Lambda^{12}( dp_1^*f_1 \otimes P_1^*s_1 )[p_2^*g_2]&=0 &\Lambda^{12}( dp_1^*f_1 \otimes P_1^*s_1 )[p_2^*g_2]&=0
\end{align*}
Note that we have omitted the $\sharp$ symbol from the squiggle for simplicity. Then, using the symbol-squiggle identity inserting different combinations of spanning functions to obtain consistency relations, we obtain the following identities for the squiggle of the product Jacobi bracket
\begin{align*}
    \Lambda^{12}( dp_1^*f_1 \otimes P_1^*s_1 )[\tfrac{a}{b}] &= -p^*_1X^1_a[f_1]\tfrac{s_1}{b} \\
    \Lambda^{12}( dp_2^*f_2 \otimes P_2^*s_2 )[\tfrac{b}{a}] &= -p^*_2X^2_b[f_2]\tfrac{s_2}{a} \\
    \Lambda^{12}( dp_1^*f_1 \otimes P_2^*s_2 )[\tfrac{b}{a}] &= p_1^*X^1_a[f_1]\tfrac{s_2}{a}\tfrac{b}{a}\\
    \Lambda^{12}( dp_2^*f_2 \otimes P_1^*s_1 )[\tfrac{a}{b}] &= p_2^*X^2_b[f_2]\tfrac{s_1}{b}\tfrac{a}{b}\\
    \Lambda^{12}( dp_1^*f_1 \otimes P_2^*s_2 )[p_1^*g_1] &= -p_1^*\Lambda^1(dg_1\otimes a)[f_1]\tfrac{s_2}{a}, \quad \forall a\in\Sec{L_1^\bullet}\\
    \Lambda^{12}( dp_2^*f_2 \otimes P_1^*s_1 )[p_2^*g_2] &= -p_2^*\Lambda^2(dg_2\otimes b)[f_2]\tfrac{s_1}{b}, \quad \forall b\in\Sec{L_2^\bullet}\\
    \Lambda^{12}( dp_1^*f_1 \otimes P_2^*s_2 )[p_2^*s_2] &= 0\\
    \Lambda^{12}( dp_2^*f_2 \otimes P_1^*s_1 )[p_1^*s_1] &= 0 
\end{align*}

\end{document}